\documentclass[]{siamart}
\usepackage{amsmath,amssymb}
\usepackage{a4wide}
\usepackage{graphicx,color}
\usepackage{mathrsfs}
\usepackage{tikz}
\usetikzlibrary{arrows}
\usepackage{bbm}
\usepackage{algorithm}
\usepackage{algorithmic}
\usepackage{makecell}
\DeclareMathOperator{\Ima}{Im}

\newtheorem{assumption}{Assumption}
\newtheorem{remark}[theorem]{Remark}

\newcommand{\w}[0]{\omega}
\newcommand{\PP}[0]{\mathbb{P}}
\newcommand{\E}[0]{\mathbb{E}}
\def\Var{{\textrm{Var}}\,}
\newcommand{\LQ}[0]{\left[}
\newcommand{\RQ}[0]{\right]}
\newcommand{\SD}[0]{\mathbb{S}}
\newcommand{\SDQ}[0]{\widehat{\mathbb{S}}}
\newcommand{\EQ}[0]{\widehat{\mathbb{E}}}
\newcommand{\Yu}[0]{\underline{Y}}
\newcommand{\Y}[0]{Y}
\newcommand{\vu}[0]{\underline{v}}
\newcommand{\uu}[0]{\underline{u}}
\newcommand{\pu}[0]{\underline{p}}
\newcommand{\qu}[0]{\underline{q}}
\newcommand{\xu}[0]{\underline{x}}
\newcommand{\ru}[0]{\underline{r}}

\newcommand{\yb}[0]{\mathbf{y}}
\newcommand{\pb}[0]{\mathbf{p}}
\newcommand{\vb}[0]{\mathbf{v}}
\newcommand{\Xhat}[0]{\widehat{\mathcal{X}}}
\newcommand{\SLR}[0]{S_{\text{LR}}}
\newcommand{\SLRM}[0]{S_{\text{LRM}}}
\newcommand{\SLRC}[0]{S_{\text{LRC}}}
\newcommand{\PLR}[0]{P_{\text{LR}}}
\newcommand{\PLRM}[0]{P_{\text{LRM}}}
\newcommand{\PLRC}[0]{P_{\text{LRC}}}
\newcommand{\POPM}[0]{P_{\text{OPM}}}
\newcommand{\POP}[0]{P_{\text{OP}}}
\newcommand{\SOP}[0]{S_{\text{OP}}}
\newcommand{\POPC}[0]{P_{\text{OPC}}}
\newcommand{\unob}[0]{\mathbbm{1}}
\newcommand{\tol}[0]{{\text{Tol}}}
\newcommand{\Nit}[0]{N_{\text{it}}}
\newcommand{\Ys}[0]{\mathcal{Y}}
\newcommand{\Ps}[0]{\mathcal{P}}
\newcommand{\C}[0]{\mathcal{C}}
\newcommand{\B}[0]{\mathcal{B}}
\newcommand{\Ex}[0]{\mathcal{E}}

\newcommand{\Ms}[0]{M_s}
\newcommand{\setR}[0]{\mathbb{R}}
\newcommand{\I}[0]{\mathcal{I}}
\newcommand{\IG}[0]{\mathcal{I}_{G^0}}
\newcommand{\ES}[0]{\mathscr{Y}}
\newcommand{\df}[0]{\kappa}
\definecolor{Tommaso}{rgb}{0, 0, 1}
\definecolor{Tommaso2}{rgb}{1, 0, 0}
\definecolor{Fabio+}{rgb}{1,1,0}

\begin{document}

\begin{center}

\begin{Large}
\textbf{Preconditioners for robust optimal control problems under uncertainty}
\end{Large}

\medskip

Fabio Nobile$^{1}$ and Tommaso Vanzan$^{2}$

\medskip

${}^1$ CSQI Chair, Institute de Math\'ematiques, \'Ecole Polytechnique F\'ed\'erale de Lausanne, CH-1015 Lausanne, Switzerland, fabio.nobile@epfl.ch.

${}^2$ CSQI Chair, Institute de Math\'ematiques, \'Ecole Polytechnique F\'ed\'erale de Lausanne, CH-1015 Lausanne, Switzerland, tommaso.vanzan@epfl.ch.

\end{center}
\begin{abstract}
The discretization of robust quadratic optimal control problems under uncertainty using the finite element method and the stochastic collocation method leads to large saddle-point systems, which are fully coupled across the random realizations. Despite its relevance for numerous engineering problems, the solution of such systems is notoriusly challenging. In this manuscript, we study efficient preconditioners for all-at-once approches
using both an algebraic and an operator preconditioning framework.
We show in particular that for values of the regularization parameter not too small, the saddle-point system can be efficiently solved by preconditioning in parallel all the state and adjoint equations. For small values of the regularization parameter, robustness can be recovered by the additional solution of a small linear system, which however couples all realizations. A mean approximation and a Chebyshev semi-iterative method are investigated to solve this reduced system.
Our analysis considers a random elliptic partial differential equation whose diffusion coefficient $\df(x,\w)$ is modeled as an almost surely continuous and positive random field, though not necessarily uniformly bounded and coercive. We further provide estimates on the dependence of the preconditioned system on the variance of the random field. Such estimates involve either the first or second moment of the random variables $1/\min_{x\in \overline{D}} \df(x,\w)$ and $\max_{x\in \overline{D}}\df(x,\w)$, where $D$ is the spatial domain.
The theoretical results are confirmed by numerical experiments, and implementation details are further addressed.
\end{abstract}

\section{Introduction}
Optimal Control Problems (OCPs) constrained by deterministic Partial Differential Equations (PDEs) have been extensively studied in the last decades since they
are essential tools in the design of complex engineering systems, see e.g. \cite{lions1971optimal,hinze2008optimization,troltzsch2010optimal}. However, the physical system under study is often affected by uncertainties, either due to a lack of knowledge on some parameters defining the model or due to an intrinsic randomness in the system.
To have more reliable results, it is important to account for the ubiquitous uncertainty in nature by considering OCPs constrained by random PDEs, which belong to the class of OCPs Under Uncertainty (OCPUU). 
In such OCPs, the objective functional involves suitable statistical measures, often called risk measures \cite[Chapter 6.3]{shapiro2014lectures}, of
the quantity of interest to be minimized. Examples of risk measures are an expectation, an expectation plus variance, a quantile, or a conditional expectation above a quantile, also called Conditional Value at Risk (CVaR) \cite{Kouri_ex,Kouri_Cvar}.
In this article, we consider a mean-variance quadratic model. Our objective functional involves an expectation of the distance of the PDE solution from a target state, an additional penalization on the variance of the PDE solution, and a standard penalization on the control.

There are two possible paradigms to minimize numerically such functionals involving expectations over a probability measure. The first one, called Stochastic Approximation (SA) method \cite[Chapter 5.9]{shapiro2014lectures}, includes iterative methods that at each iteration draw new realizations independent from the previous ones. Examples of such approaches are the stochastic gradient method and its variants, which have been recently studied for OCPUU in \cite{martin2018analysis,Martin_Nobile2,Geiersbach,ReporUQ}.

In this manuscript, we adopt a second approach called Sample Average Approximation (SAA) method \cite[Chapter 5.1]{shapiro2014lectures}, in which the original objective functional is replaced by an accurate approximation obtained discretizing once and for all the probability space using Stochastic Collocation Methods (SCMs), with Monte Carlo, Quasi-Monte Carlo \cite{doi:10.1137/19M1294952}, or Gaussian quadrature formulae. We do not consider approximations based on Multilevel Monte Carlo \cite{van2019robust} and sparse grids formulae \cite{Kourisparse,schillings2013sparse}, since they may not preserve the convexity of the objective functional.
After discretization, we obtain an extremely large global system involving $N$ state equations, $N$ adjoint equations and an optimality condition, where $N$ is the number of collocation points.


The properties of the global linear system depend strongly on whether the control is \textit{stochastic} or \textit{deterministic}. In the first case, one assumes that the realization of the randomness is observable, and thus an optimal control can be established for each single random realization, leading thus to a \text{stochastic} optimal control. On the one hand, such problems are easier to solve using SCMs since the global linear system is often decoupled across all random samples, so that one actually needs to solve a sequence of independent, deterministic OCPs, one for each sample, for which optimal preconditioners are available \cite{rees2010all,rees2010optimal,pearson2012new,zulehner2011nonstandard,mardal2011preconditioning,Borzi,Borzi2}.
On the other hand, a discretization based on Stochastic Galerkin Methods (SGMs) \cite{ghanem,pellissetti} leads to a fully coupled saddle point system across the random components, even when the optimal control is stochastic. In \cite{benner2016block}, the authors consider an OCP with a stochastic control discretized using SGM, obtaining the classical saddle point matrix of deterministic OCPs, but whose blocks  have a Kronecker structure due to the SGM discretization. 

In the second case, that is the one we are interested in, the randomness is not observable a-priori, and thus one computes a unique \textit{deterministic} control valid for all random realizations. This setting is often called robust OCPUU, since the deterministic control minimizes the risk of having large values of the objective functional. These problems are harder, since the global system fully couples all the random samples. 

Let us now review previous works concerning all-at-one approaches for robust OCPUU.
Gradient-based approaches, which are essentially Gauss-Seidel iterative methods, permit to obtain the solution iteratively, solving the three sets of equations (state, adjoint, optimality) sequentially at each iteration. Gradient descent combined with a Multilevel Monte Carlo estimation of the expectations has been proposed in \cite{van2019robust}.
In \cite{rosseel2012optimal}, the authors derive optimality conditions for a general quadratic OCP constrained by a random elliptic PDE. Their control can be either stochastic or deterministic. They interestingly remark that a discretization using SCM leads to a global system coupling all realizations, unless the control is fully stochastic as previously mentioned. Hence, since SCM loses its non-intrusivity property, they focus on SGM and they discuss numerical results using either a mean-based preconditioner \cite{powell2009block} or collective smoothing multigrid \cite{Borzi}. A MG/OPT algorithm based on a hierarchy of sparse grid approximations of the objective functional has been proposed in \cite{kouri2014multilevel}. At each level of the hierarchy, the approximated reduced optimality system is solved using the conjugate gradient algorithm. Another MG/OPT algorithm based on a classical hierarchy of geometric meshes has been analyzed in \cite{van2020mg}. Other works aimed to reduced the computatational costs are \cite{Kourisparse}, where the authors studied a trust region algorithm that adapts the number of collocation points during the optimization process, and \cite{zahr2019efficient} which combines the previous algorithm with model order reduction.

In this manuscript, we study optimal preconditioners for linear systems obtained from a SCM discretization of a robust quadratic OCP constrained by a random PDE. These preconditioners can then be used either directly to solve the optimality system of a robust quadratic OCP, or possibly inside one the previously cited optimization algorithms.

Despite the remarks of \cite{rosseel2012optimal} about the loss of non-intrusivity of SCM for OCPUU with deterministic control, we are interested in analysing SCMs for the following reasons. 

First, SCM maintains its advantages in terms of applicability with respect to general parameter distributions and ease of implementation \cite[Chapter 10]{smith2013uncertainty}.

Second, one can construct preconditioners whose action can be fully parallelized across the realizations of the randomness, one example being the preconditioner proposed in \cite{Kouri2018}. That is, while a global system involving all realizations has to be solved, the preconditioner does not couple the realizations, as it requires to solve approximately (i.e. to precondition) independently each forward and adjoint problem. In this perspective, this preconditioner has clearly advantages in terms of parallelization and memory distribution in a high performance setting with respect to other solutions strategies based on SGM. In this manuscript, we analyse, among others, the performance of the preconditioner proposed in \cite{Kouri2018}, by providing theoretical estimates for the spectrum of the preconditioned system.

As the regularization parameter gets smaller, the preconditioner introduced in \cite{Kouri2018} becomes inefficient. 
Thus, we introduce a first new preconditioner, named $\PLR$, which still preconditions each state and adjoint equation in parallel, but requires the additional solution of a small linear system. We partially characterize the spectrum of the preconditioned system and show numerically its $\beta$-robustness, $\beta$ denoting the regularization parameter on the control.
Finally, to derive a provably $\beta$-robust preconditioner, we study the optimality system at the fully-continuous level, and our analysis leads to a second new preconditioner, named $\POP$, for which a complete theory is available.  
Both the first and second preconditioner require to invert the sum of all inverses of the stiffness matrices. A mean approximation, combined possibly with a Chebyshev semi-iteration, is shown to be sufficient to efficiently approximate this inverse for quite a wide range of parameters, leading to \textit{practical} $\PLRM,\PLRC,\POPM,\POPC$ preconditioners, where the subscript $M$ stands for ``mean'' and $C$ for Chebyshev.

We remark that the development of robust preconditioners for small values of the regularization parameter is not obvious and poses some interesting mathematical and computational challenges which, surprisingly, are similar to those encountered in deterministic OCP when the control acts locally, either on a portion of the domain \cite{elvetun2016pde}, or on a portion of the boundary \cite{heidel2019preconditioning}. 

We further stress that our analysis does not assume that the random bilinear form is uniformly bounded and coercive with respect to the randomness, which is a frequent simplifying hypothesis in the literature, see, e.g., \cite{powell2009block,ullmann,benner2016block}.
Hence, our results will also cover the case of log-normally distributed random fields, which are common models in engineering applications, and they will cast light on how the preconditioners' performance is affected by the variance of the random fields.

To develop optimal preconditioners for our robust OCPUU, we rely on two different approaches. The first one, used to derive the preconditioner $\PLR$ is algebraic and has its roots in the seminal work of \cite{murphy2000note}, where the authors proposed an optimal, but expensive, preconditioner for saddle point matrices which relies on the exact Schur complement. For deterministic OCP, several preconditioners based on approximations of the exact Schur complements have been study in the last decade \cite{rees2010all,rees2010optimal,pearson2012new,kourikkt,liu2020parameter,pearson2018matching}.
Concerning OCPUU, the same strategy has been applied in \cite{benner2016block} in combination with SGMs, but we remark once more that their problem is not an instance of robust OCP.
The second approach, used to derive the preconditioner $\POP$, consists in the so-called ``operator preconditioning" paradigm, and is based on identifying the saddle point system as a linear operator acting between Hilbert spaces, and finding proper weighted-norms such that the continuity constants of the map and of its inverse are independent on the parameters of interest. We refer the interested reader to \cite{malek2014preconditioning,zulehner2011nonstandard,mardal2011preconditioning,functional_iterative,khan2019robust}. While studying this approach, we will discuss the well-posedness of the OCPUU and the development of robust preconditioners at the continuous level for log-normal fields, without relying on the framework developed in \cite{gittelson,schwab}.

The manuscript is organized as follows. In Section \ref{sec:notation} we introduce the notation, while in Section \ref{sec:problemsetting} we define our model problem, provide sufficient conditions for well-posedness and derive the optimality conditions. Section \ref{sec:discretization} introduces the discretization both in probability and in space. Section \ref{sec:AlgebraicPreconditioners} deals with algebraic preconditioners for saddle point matrices based on approximations of the Schur complement. Section \ref{sec:OP} derives preconditioners using the operator preconditioning approach.
Finally, Section \ref{sec:num} presents numerical experiments validating the theoretical results.

\section{Notation}\label{sec:notation}
Let us consider a bounded domain $D\subset \setR^d$, $d\in \{1,2,3\}$, with Lipschitz boundary and a complete probability space $(\Omega,\mathcal{F},\mathbb{P})$. 
For every $p\in [1,\infty]$, $L^p(D)$ denotes the space of $p-$Lebesque integrable functions over $D$ and $H^1(D)$ is the Sobolev space
\[ H^1(D):=\left\{ v\in L^2(D) : \partial_{x_i} v \in L^2(D),\quad \text{for }i=1,\dots,d\right\}.\]
The natural space for our analysis is $H^1_0(D)$, which is the subspace of $H^1(D)$ containing functions that vanish on $\partial D$, equipped with the norm $\|y\|_{H^1_0(D)}:=\|\nabla y\|_{L^2(D)}.$ The topological dual of $H^1_0(D)$ is $H^{-1}(D)$. We denote by $C_P$ the Poincar\'{e} constant
so that $\|v\|_{L^2(D)}\leq C_P \|v\|_{H^1_0(D)},$ $\forall v\in  H^1_0(D)$.
For the sake of brevity, we will denote $H^1_0(D)$ and $H^{-1}(D)$ by $\Y$ and $\Y^\star$.
Given an integer $N\in \mathbb{N}$ and an Hilbert space $V$, we denote by $\underline{V}:=\prod_{i=1}^N V$ the Cartesian product of $N$ copies of $V$, with the standard scalar product $(\uu,\vu)_{\underline{V}}=\sum_{i=1}^N (u_i,v_i)_{V}$.
Given a Banach space $U$, the duality pairing between $U$ and $U^\star$ is denoted by $\langle\cdot,\cdot\rangle$. The specific choice of $U$ will be clear from the context. Further, let $L^p(\Omega,\mathcal{F},\mathbb{P};V)$ be the Bochner space
\[ L^p(\Omega,\mathcal{F},\mathbb{P};V):=\left\{ v:\Omega\rightarrow V, \text{ $v$ strongly measurable}, \int_\Omega \|v(\cdot,\w)\|^p_{V} d\mathbb{P}(\omega) < +\infty\right\},\]
henceforth noted $L^p(\Omega,V)$, and equipped with the norm $\|v\|_{L^p(\Omega,V)}:=(\int_\Omega \|v(\cdot,\w)\|^p_{V} d\PP(\w))^\frac{1}{p}$.
For a Hilbert space $V$, $L^2(\Omega,V)$ is a Hilbert space as well, equipped with the scalar product
$(u,v)_{L^2(\Omega,V)}:=\int_{\Omega} (u(\cdot,\w),v(\cdot,\w))_V d\mathbb{P}(\w)$.
To stress better the dependence of function-valued random variables on an elementary random event $\omega$, we will use the notation $v_\w=v(\cdot,\w)$ for almost every (a.e.) $\w \in \Omega$. The expectation operator $\E : L^1(\Omega)\rightarrow \setR$ is defined as
\[\E\LQ X\RQ= \int_\Omega X(\w) d\PP(\w),\quad \forall X\in L^1(\Omega).\]
For $X \in L^2(\Omega)$, the variance $\Var: L^2(\Omega)\rightarrow \setR^+$ and standard deviation $\SD:L^2(\Omega)\rightarrow \setR^+$ are defined as 
\[\Var \LQ X\RQ:=\E\LQ (X-\E \LQ X \RQ)^2 \RQ=  \int_\Omega (X-\E \LQ X \RQ)^2 d\mathbb{P}(\omega),\quad\text{and}\quad \SD\LQ X \RQ:=\sqrt{\Var \LQ X \RQ}.\]
We will use repeatedly the Woodbury identity, that is
\[(A+UCV)^{-1}=A^{-1}-A^{-1}U(C^{-1}+VA^{-1}U)^{-1}VA^{-1},\]
where $A\in \setR^{n\times n}$, $C\in \setR^{r\times r}$, $U\in \setR^{n\times r}$, $V\in \setR^{r\times n}$, with $A$ and $C$ invertible. 
Finally, the spectrum of a matrix $H$ is denoted with $\sigma(H)$.

\section{Problem setting}\label{sec:problemsetting}
Our physical model is described by the elliptic random Partial Differential Equation (PDE)
\begin{equation}\label{eq:state}
\begin{aligned}
-\text{div}(\df(x,\w)\nabla y(x,\w))&= \phi(x),\quad &x\in D,\text{ } \w \in\Omega,\\
y(x,\w)&=0,\quad & x\in \partial D,\text{ }\w \in\Omega,
\end{aligned}
\end{equation}
where $\phi(x)$ is a deterministic force term and $\w$ is an elementary random event.

We do not require $\df(x,\w)$ to be uniformly bounded in $\w$, but we make the following assumption.
\begin{assumption}[On the random diffusion field]\label{ass:diff}
The random diffusion field $\df$ has almost surely (a.s.) continuous and positive realizations and the map $\w\mapsto \df(\cdot,\w)\in C^0(\overline{D})$ is measurable. Thus, the random variables $\df_{\min}(\w):=\min_{x\in \overline{D}} \df(x,\w)$ and $\df_{\max}(\w):=\max_{x \in \overline{D}} \df(x,\w)$ are well-defined. Further, there exists a $p\in [1,\infty]$ such that both $\df_{\max}(\w)$ and $\frac{1}{\df_{\min}(\w)}$ are in $L^p(\Omega)$.
\end{assumption}

These assumptions are clearly verified with $p=\infty$ by a continous random field which is uniformly bounded, i.e. if there exists $K_1$, $K_2\in\mathbb{R}^+$ such that
\begin{equation*}\label{eq:boundedfield}
K_1\leq \df(x,\w)\leq K_2,\forall x\in D, \text{ a.e. } \w \in \Omega.
\end{equation*}
Another instance is the log-normal random field $\df(x,\w)=\exp(g(x,\w))$, where $g(x,\w)$ is a Gaussian field with covariance function $\text{cov}[g](x,y):=k(\|x-y\|)$, and $k(\cdot)$ is a Lipschitz function. Both $\df_{\min}(\w)$ and $\df_{\max}(\w)$ are in $L^p(\Omega)$ for every $p\in [1,\infty)$ \cite{Charrier}.

For a.e. $\w \in \Omega$, $a_\w (\cdot,\cdot):Y\times Y\rightarrow \setR$, $a_\w(u,v):=\int_{D} \df(x,\w)\nabla u(x) \nabla v(x) dx$ is a symmetric, continuous and coercive bilinear form, but not necessarily uniformly in $\omega$ due to Assumption \ref{ass:diff}. It holds
\begin{equation}\label{eq:coerv_conti}
\df_{\min}(\w)\|u\|^2_{\Y}\leq a_\w(u,u)\leq \df_{\max}(\w) \|u\|^2_{\Y},
\end{equation}
and the operator $\mathcal{A}_\w:\Y\rightarrow \Y^\star$ is defined as $\langle\mathcal{A}_\w u,v\rangle:=a_\w(u,v)$. 

The weak formulation of \eqref{eq:state} on $Y$ for a.e. $\omega\in \Omega$ is 
\begin{equation}\label{eq:state_weak}
\text{find }y_\w \in \Y \text{ s.t. } a_\w(y_\w,v)=\langle \phi,v\rangle, \quad  \text{for every }v\in Y, \text{for a.e. }\w\text{ in }\Omega.
\end{equation}
Due to Assumption \ref{ass:diff}, the following classical result holds \cite{Charrier,Scheichl,lord_powell_shardlow_2014}.
\begin{lemma}\label{lemma:wellpose}
Problem \eqref{eq:state_weak} has a unique solution $y_\w$ for a.e. $\w \in \Omega$. Further, 
\begin{equation*}
\begin{aligned}
\|y_\w\|_{\Y} &\leq \frac{\|\phi\|_{\Y^\star}}{\df_{\min}(\w)},\quad \text{for a.e. } \w \in \Omega,\\  
  \|y\|_{L^p(\Omega,\Y)}&\leq \|\phi\|_{\Y^\star}\left\|\frac{1}{\df_{\min}}\right\|_{L^p(\Omega)}.
\end{aligned}
\end{equation*}
\end{lemma}
Integrating \eqref{eq:state_weak} with respect to the measure $\PP$, and defining 
the bilinear form 
\begin{equation}\label{eq:definition_bilinear_form}
a(u,v):= \int_\Omega \int_D \df(x,\w)\nabla u(x,\w) \nabla v(x,\w) dx d\PP(\w)= \E \LQ \langle \mathcal{A}_\w u_\w,v_\w\rangle \RQ ,
\end{equation}
the energy space $\ES:=\left\{v:\Omega \rightarrow Y: \mathcal{F}\slash\mathcal{B}(Y)\text{-measurable }, \|v\|_\mathcal{A}^2:=a(v,v) <\infty\right\}$, and the functional $\Phi(v):=\int_\Omega \langle \phi,v_\w\rangle d\PP(\w)=\E\LQ \langle \phi,v_\w\rangle \RQ$, one can consider the global weak formulation in both physical and probability spaces,
\begin{equation}\label{eq:state_weak_prob_space}
\text{find } y\in \ES \text{ s.t. } a(u,v)=\Phi(v),\quad \forall v\in \ES.
\end{equation}
We further introduce the operator $\mathcal{A}:\ES\rightarrow \ES^\star,$ $\langle \mathcal{A}u,v\rangle:=a(u,v)=\E \LQ \langle \mathcal{A}_\w u_\w,v_\w\rangle \RQ$. 
\begin{remark}
If $\df\notin L^\infty(\Omega,L^\infty(D))$, the weak formulation \eqref{eq:state_weak_prob_space} may not be well-defined over the Bochner space $L^2(\Omega,Y)$, and the energy norm is not equivalent to any standard Bochner norm. 
A sophisticated framework has been developed in \cite{gittelson,schwab} to derive a well-posed weak formulation for the log-normal model, obtained integrating \eqref{eq:state_weak} with respect to an auxiliary Gaussian measure $\widetilde{\PP}$. This setting permits to bound from above and below the energy norm with two different Bochner-norms, and it is used to study the convergence rate of polynomial approximations.
However, we will not need the framework of \cite{gittelson,schwab} in our analysis. For our purposes it is sufficient to show that \eqref{eq:state_weak_prob_space} has a unique solution and, using the special form of $\phi$, being a deterministic functional, we can show that the solution of \eqref{eq:state_weak_prob_space} has actually $L^p(\Omega,Y)$ regularity, which is sufficient for the well-posedness of the OCP ( if $p\geq 2$, see Lemma \ref{lemma:wellposednessOCP}).
We refer to Remark \ref{remark:OP2} in Section \ref{sec:OP} for a further discussion on the weak formulation \eqref{eq:state_weak_prob_space} and on \cite{gittelson,schwab} in the context of deriving preconditioners at the continuous level. These observations are formalized in Lemma \ref{Lemma:weak_formulation}.
\end{remark}
\begin{lemma}\label{Lemma:weak_formulation}
The solution of \eqref{eq:state_weak}, interpreted as the representative element of the equivalence class of functions coinciding  $\PP$-a.s. with it, is the unique solution of the linear variational problem \eqref{eq:state_weak_prob_space} and lies in $L^p(\Omega,Y)$.
\end{lemma}
\begin{proof}
Since the energy space $\ES$ is a Hilbert space, \cite[Proposition 3.6]{gittelson}, the existence and uniqueness of the solution of \eqref{eq:state_weak_prob_space} follows from Riesz's theorem if $\Phi\in \ES^\star$. Due to the specific form of $\Phi$, this is easily verified since for any $\phi \in Y^\star$,
\[|\Phi(v)|=\left|\int_{\Omega} \langle \phi,v(\cdot,\w)\rangle d\PP(\w)\right| \leq \|\phi\|_{Y^\star}\int_{\Omega}\|v(\cdot,
w)\|_Y d\PP(\w) \leq \|\phi\|_{Y^\star}\sqrt{\E \LQ \frac{1}{\df_{\min}(\w)}\RQ}\|v\|_\mathcal{A}.\]
Further, Corollary 3.8 in \cite{gittelson} shows that the solution of \eqref{eq:state_weak} coincides $\PP$-a.e with the unique solution of \eqref{eq:state_weak_prob_space}. Finally, using Lemma \ref{lemma:wellpose}, we obtain the desired regularity. 
\end{proof}

In this manuscript, we are interested in solving OCPs constrained by the state equation \eqref{eq:state_weak_prob_space}.
We will suppose that the deterministic force term $\phi$ can be decomposed in a given deterministic part called $f$, and a deterministic control $\widetilde{u}$. We suppose that $\widetilde{u}$ lies in the dual of an Hilbert space $U$, which will be either $L^2(D)$ or $Y$. In both cases, we will use the Riesz operator $\Lambda_U :U\rightarrow U^\star \subset \Y^\star$, such that $\widetilde{u}=\Lambda_U u$ and $\langle \widetilde{u},v\rangle =(u,v)_U,\forall v\in U$
The quantity we aim to compute is the Riesz representative $u$.

We focus on the quadratic objective functional\footnote{We tacitly omit the continuous embedding operator from $L^2(\Omega,Y)$ to $L^2(\Omega,L^2(D))$.}
\begin{equation*}
\begin{split}
J(u)&=\frac{1}{2}\E\LQ \|y_\w-y_d\|^2_{L^2(D)}\RQ +\frac{\gamma}{2}\|\SD\LQ y_\w \RQ \|^2_{L^2(D)}+\frac{\beta}{2}\|u\|^2_{U}\\
&=\frac{1}{2}\left(y-y_d,y-y_d\right)_{L^2(\Omega,L^2(D))} +\frac{\gamma}{2}\left(y-\mathbb{E}\LQ y_\w\RQ,y-\mathbb{E}\LQ y_\w \RQ\right)_{L^2(\Omega,L^2(D))} + \frac{\beta}{2}(u,u)_U,
\end{split}
\end{equation*}
where $\gamma\geq 0$, $\beta >0$, and $y_d\in L^2(D)$ is a deterministic target state. Exchanging the order of integration, it holds $\|\SD \LQ y_\w\RQ \|^2_{L^2}=\E\LQ \|y_\w-\E\LQ y_\w\RQ\|^2_{L^2}\RQ$.
The whole OCP can be formulated as
\begin{equation}\label{eq:OCP}
\begin{cases}
\min_{u\in U} J(u)=\frac{1}{2}\E\LQ \|y_\w(u)-y_d\|^2_{L^2}\RQ +\frac{\gamma}{2}\|\SD\LQ y_\w(u)\RQ \|^2_{L^2}+\frac{\beta}{2}\|u\|^2_{U},\\
\text{where } y_\w(u)\in \ES \text{ solves   }\\
\E\LQ \langle \mathcal{A}_\w y_\w(u),v_\w\rangle \RQ=\E \LQ  \langle f +\Lambda_U u,v_\w\rangle\RQ,\quad \forall v\in \ES.
\end{cases}
\end{equation}
We emphasize the dependence of $y$ on the control $u$ through the notation $y(u)$.

\begin{lemma}[Well posedness of the OCP]\label{lemma:wellposednessOCP}
If Assumption \ref{ass:diff} holds with $p\geq 2$, then the OCP \eqref{eq:OCP} admits a unique solution $u^*\in U$.
\end{lemma}
\begin{proof}
Due to Lemma \ref{Lemma:weak_formulation} and the hypothesis $p\geq 2$, the solution $y(u)$ to the state equation is in $L^2(\Omega,Y)$ for any $u\in U$, thus the objective functional is well-defined.
The control-to-state map $\widetilde{S}:U\rightarrow L^p(\Omega,\Y)$ defined by \eqref{eq:state_weak_prob_space} (with $\phi=f+u$) as $\widetilde{S}: u\mapsto y(u)$ is affine in $u$, while $S:U\rightarrow L^p(\Omega,Y)$, $S:u\mapsto y(u)-y(0)$ is linear and bounded with operator norm $\|S\|_{\mathcal{L}(U,L^p(\Omega,Y))}=:C_{S,p}<\infty$. 
The objective functional can be written as $J(u)=\frac{1}{2}\pi(u,u) -L(u) +C$, where
\begin{equation*}\label{eq:bilinearforms_OCP}
\begin{aligned}
\pi(u,v)&:=\left(S(u),S(v)\right)_{L^2(\Omega,L^2(D))} +\gamma\left(S(u)-\mathbb{E}\LQ S(u)\RQ,S(v)-\mathbb{E}\LQ S(v) \RQ\right)_{L^2(\Omega,L^2(D))} + \frac{\beta}{2}(u,v)_U,\\
L(v)&:=(y_d-y(0),S(v))_{L^2(\Omega,L^2(D))}+\gamma (\E\LQ y_\w(0)\RQ-y(0),S(v)-\E\LQ S(v)\RQ)_{L^2(\Omega,L^2(D))},\\
C&:=\|y(0)-y_d\|^2_{L^2(\Omega,L^2(D))}+\gamma \|y(0)-\E\LQ y_\w(0)\RQ\|^2_{L^2(\Omega,L^2(D))}.
\end{aligned}
\end{equation*}
$\pi(\cdot,\cdot)$ is clearly coercive, as $\pi(u,u)\geq \frac{\beta}{2}\|u\|^2_U$. For the continuity we observe that
\[(S(u),\E\LQ S(u)\RQ)_{L^2(\Omega,L^2(D))}=\int_{\Omega}\int_{D} S(u)\E\LQ S(u)\RQ dx d\mathbb{P}(\w)=\int_{D}\E\LQ S(u)\RQ^2=\|\E\LQ S(u)\RQ\|^{2}_{L^2(D)},\]
so that $(S(u)-\E\LQ S(u)\RQ,S(u)-\E\LQ S(u)\RQ)_{L^2(\Omega,L^2(D))}\leq (S(u),S(u))_{L^2(\Omega,L^2(D))} \leq C_P^2 \|S(u)\|^2_{L^2(\Omega,Y)}$, where $C_P$ is the Poincar\'{e} constant. Thus, $\pi(u,u)\leq \left((1+\gamma) C_P^2 C_{S,2}^2 +\frac{\beta}{2}\right)\|u\|^2_U$. Similarly it can be shown that $L(\cdot)$ is continuous.
As $\pi(\cdot,\cdot)$ is continuous and coercive and $L(\cdot)$ continuous, $J(u)$ is strictly convex and weakly lower semicontinuous. Thus, having verified the hypothesis of the classical theory of Lions \cite{lions1971optimal}, the OCP admits a unique solution $u_*\in U$, see also, e.g. \cite{troltzsch2010optimal,martinez2018optimal,Ciarlet}.
\end{proof}

To derive the optimality conditions, we rely on a optimize-then-discretize paradigm and a Lagragian approach, see e.g. \cite{hinze2008optimization}. For the sake of brevity, we omit the calculations of the directional derivatives evaluated in $(y,u,p)$, where $p\in L^2(\Omega,Y)$ is the adjoint variable, along the directions $\delta y, \delta p$, and $\delta u$. We refer the interested reader to \cite[Section 4]{van2019robust} and \cite{ReporUQ2}. The optimality system reads
\begin{equation}\label{eq:optimality_system}
\begin{array}{r l r r}
\E\LQ \langle \mathcal{A}_\w p_\w,v_\w\rangle \RQ + \E\LQ \langle \Lambda_{L^2}\right(y_\w+\gamma (y_\w-\E\LQ y_\w\RQ)\left),v_\w\rangle \RQ &= \E\LQ \langle \Lambda_{L^2} y_d,v_\w\rangle \RQ,\quad &\forall v\in \ES,&\\
\langle \beta \Lambda_U u- \Lambda_U \E\LQ  p_\w\RQ,v\rangle &=0,\quad &\forall v\in U,&\\
\E\LQ \langle \mathcal{A}_\w y_\w,v_\w\rangle\RQ -\E\LQ \langle \Lambda_U u,v_\w\rangle \RQ & = \E\LQ \langle f,v_\w\rangle\RQ,\quad &\forall v \in \ES.&
\end{array}
\end{equation}

\section{Discretization}\label{sec:discretization}
\subsection{Discretization in probability}
To numerically approximate the solution of \eqref{eq:OCP}, we rely on a Sample Average Approximation (SAA) \cite{shapiro2014lectures}. We replace the exact continuous expectation operator $\E\LQ\cdot\RQ$ with a suitable quadrature formula $\EQ \LQ\cdot \RQ$ with $N$ nodes. More specifically, given a random variable $X\in L^2(\Omega)$ we approximate,
\begin{align*}
\E\LQ X(\w)\RQ&=\int_\Omega X(\w)d\PP(\w)\approx \sum_{i=1}^N \zeta_i X(\w_i)=: \EQ \LQ X(\w) \RQ,\\
\SD\LQ X(\w) \RQ&=\sqrt{\E\LQ (X(\w)-\E \LQ X(\w)\RQ )^2\RQ}\approx\sqrt{\EQ\LQ (X(\w)-\EQ\LQ X(\w)\RQ )^2\RQ}=:\SDQ\LQ X(\w)\RQ,
\end{align*}
where $\zeta_i$ and $\w_i$ are, respectively, the weights and nodes of the quadrature formula with $\sum_{j=1}^N \zeta_j=1$.
We restrict ourselves to quadrature formulae with positive weights, thus excluding sparse grids and Multilevel Monte Carlo approximations, since the presence of negative weights may compromize the convexity of the OCP.
Among the quadrature formulae which satisfy the constraint of positive weights, we recall standard Monte Carlo, Quasi-Monte Carlo and Gaussian formulae. The construction of Gaussian quadrature formulae requires that the probability space can be parametrized by a sequence (finite or countable) of indipendent random variables $\left\{\xi_j\right\}_j$, each with distribution $\mu_j$, and the existence of a complete basis of tensorized $L^2_{\mu_j}$-orthonormal polynomials.
Once the probability space has been discretized, the vectors 
\[y(x)=(y_{\w_1}(x),\dots,y_{\w_N}(x))^\top\in \Yu \text{ and }p(x)= (p_{\w_1}(x),\dots,p_{\w_N}(x))^\top\in \Yu,\] contain snapshots of the function-valued random variables $\w\mapsto y(\cdot,\w)$ and $\w\mapsto p(\cdot,\w)$ at the $N$ collocation points.
We now introduce the operator $\widehat{\mathcal{A}}:\Yu\rightarrow \Yu^\star$, which approximates the bilinear form $a(\cdot,\cdot)$ in \eqref{eq:definition_bilinear_form},
\begin{equation}\label{eq:definitionA}
\langle \widehat{\mathcal{A}}\uu,\vu\rangle =\sum_{i=1}^N \zeta_i\langle \mathcal{A}_{\w_i} u_{\w_i},v_{\w_i}\rangle =\EQ\LQ\langle \mathcal{A}_\w \uu,\vu\rangle\RQ,\quad \forall \uu=(u_{\w_1},\dots,u_{\w_N}),\vu=(v_{\w_1},\dots,v_{\w_N})\in \Yu,
\end{equation}
the constant extension operator $\mathcal{I}: \Y^\star \rightarrow \Yu^\star$ such that $\mathcal{I} f=(f,\dots,f)^\top$ $\forall f \in \Y^\star$, and its adjoint $\mathcal{I}^\prime:\Yu\rightarrow \Y$ as $\mathcal{I}^\prime \vu=\sum_{i=1} v_{\w_i}$ $\forall \vu \in \Yu$, so that $\langle \mathcal{I}f,\vu\rangle_{\Yu^\star,\Yu}=\langle f, \mathcal{I}^\prime \vu\rangle_{\Y^\star,\Y}$.
Note that $\EQ\LQ y_\w\RQ=\sum_{i=1}^N \zeta_i y_{\w_i}= \mathcal{I}^\prime \mathcal{Z} \underline{y}$, where $\mathcal{Z}=\text{diag}(\zeta_1 I,\dots,\zeta_N I)$ is a diagonal matrix containing the quadrature weights. Finally, the operator $\underline{\Lambda}_{L^2}$ is defined as $\underline{\Lambda}_{L^2}\vu=(\Lambda_{L^2}v_{\w_1},\dots,\Lambda_{L^2}v_{\w_N})^\top$.
The semi-discrete matrix formulation of \eqref{eq:optimality_system}\footnote{The same semi-discrete optimality system can be derived using a discrete-then-optimize paradigm in probability, that is by replacing $\E\LQ\cdot\RQ$ with $\EQ\LQ\cdot\RQ$ into \eqref{eq:OCP}, and then by calculating the directional derivatives.}, written as an equality in dual spaces, is
\begin{equation}\label{eq:KKT_matrix}
\begin{pmatrix}
\underline{\Lambda}_{L^2}\left((1+\gamma)\mathcal{Z}-\gamma\mathcal{Z}\mathcal{I}\mathcal{I}^\prime \mathcal{Z}\right) & 0 & \widehat{\mathcal{A}}\\
0 & \beta \Lambda_U & -\Lambda_U \mathcal{I}^\prime \mathcal{Z}\\
\widehat{\mathcal{A}} & -\mathcal{Z}\mathcal{I}\Lambda_U & 0
\end{pmatrix}\begin{pmatrix}
\underline{y} \\u \\ \pu
\end{pmatrix}=
\begin{pmatrix}
\mathcal{Z}\mathcal{I} \Lambda_{L^2}y_d \\ 0 \\ \mathcal{Z}\mathcal{I}f
\end{pmatrix}.
\end{equation}
which corresponds to the set of equations
\begin{align}\label{eq:KKT}
 \mathcal{A}_{\w_i} p_{\w_i} + (1+\gamma)\Lambda_{L^2}y_{\w_i}-\gamma \EQ\LQ 
y_\w\RQ &=\Lambda_{L^2} y_d ,& i=1,\dots,N,\nonumber\\
 \beta\Lambda_U u-\Lambda_U\EQ\LQ  p_\w\RQ&=0,&\\
  \mathcal{A}_{\w_i} y_{\w_i}-\Lambda_U u &=f,& i=1,\dots,N.\nonumber
\end{align}

\subsection{Discretization in space}

Let us denote by $\left\{\mathcal{T}_h\right\}_{h>0}$ a family of regular triangulations of $D$. $Y^h$ denotes the space of continuous piece-wise polynomial functions of degree $r$ over $\mathcal{T}_h$ that vanish on $\partial D$, that is $Y^h:=\left\{v_h \in C^0(\overline{D}): v_h|_{K} \in \mathbb{P}_r(K),\quad \forall K\in \mathcal{T}_h,\text{ }y|_{\partial D}=0\right\}\subset \Y$. $N_h$ is the number of degrees of freedom associated to the space $Y^h$. 
We consider a finite element discretization of system \eqref{eq:KKT_matrix}. The vectors $\yb=(\yb_1,\dots,\yb_N)\in \setR^{N\cdot N_h}$ and $\pb=(\pb_1,\dots,\pb_N)\in \setR^{N\cdot N_h}$ are the discretization of the vector functions $\underline{y}$ and $\pu$. To discretize the control $u$, we use the same finite element space $Y^h$, see \cite[Remark 3.1]{martin2018analysis}. Further, the matrices $A_{\w_i}\in \setR^{N_h\times N_h}$ are the stiffness matrices corresponding to the elliptic operators $\mathcal{A}_{\w_i}$, and $A_0:=\sum_{i=1}^N \zeta_i A_{\w_i}$ is the empirical mean. $M_s\in \setR^{N_h\times N_h}$ is the standard mass matrix. The identity matrices are $I_s\in \setR^{N_h\times N_h}$ and $I \in \setR^{N\cdot N_h\times N\cdot N_h}$.
According to the choice of the control space, that is $U=L^2(D)$ or $U=Y^\star$, the representation of the Riesz operator $\Lambda_U$ is either $\Lambda_U=M_s$ or $\Lambda_U=K$, where $K$ is the stiffness matrix associated to the standard scalar product in $\Y$. In the following, we will suppose the control $u$ lies in $L^2(D)$, i.e. $U=L^2(D)$.

At the fully discrete level, system \eqref{eq:KKT_matrix} reads $\mathscr{S}\mathbf{x}=\mathbf{b}$,
\begin{equation}\label{eq:KKT_matrix_sym_disc}
\mathscr{S}=\begin{pmatrix}
M\left((1+\gamma)Z-\gamma Z\unob\unob^\top Z\right) & 0 & A\\
0 & \beta M_s & -M_s \unob^\top Z\\
A & -Z\unob M_s & 0
\end{pmatrix},\quad 
\mathbf{x}=\begin{pmatrix}
\yb \\ \mathbf{u} \\ \pb
\end{pmatrix},\quad 
\mathbf{b}=
\begin{pmatrix}
Z\unob M_s \yb_d \\ 0 \\ Z\unob\mathbf{f}
\end{pmatrix},
\end{equation}
where $A\in \setR^{N\cdot N_h\times N\cdot N_h}$, $M\in \setR^{N\cdot N_h\times N\cdot N_h}$, $\unob\in \setR^{N\cdot N_h\times N_h}$ are defined as 
\begin{equation*}
\begin{aligned}
A:=\begin{pmatrix}
\zeta_1A_{\w_1}& \\
& \zeta_2A_{\w_2} \\
& & \ddots\\
& & & \zeta_NA_{\w_N}
\end{pmatrix},\quad
 M:=\begin{pmatrix}
M_s \\
& M_s \\
& & \ddots\\
& & & M_s
\end{pmatrix},\quad
\unob:=\begin{pmatrix}
I_s\\ I_s\\ \vdots\\ I_s
\end{pmatrix},
\end{aligned}
\end{equation*}
and $Z=\text{diag}(\zeta_1 I_s,\dots,\zeta_N I_s)\in \setR^{N\cdot N_h,N\cdot N_h}$ is the discretization of $\mathcal{Z}$.
Since $M$ has constant diagonal blocks the following equalities hold true and will be extensively used,
\begin{equation}\label{eq:simplifing_equations}
M\unob\unob^\top=\unob M_s\unob^\top=\unob\unob^\top M\quad \text{and}\quad MZ=ZM.
\end{equation}

Introducing the matrices,
\begin{equation*}\label{eq:definition_C_B}
C:=\begin{pmatrix}
M_{\gamma} & 0 \\
0 & \beta M_s  
\end{pmatrix},\quad \text{and}\quad B:=\begin{pmatrix}
A & -Z\unob M_s
\end{pmatrix},
\end{equation*}
where $M_{\gamma}:=M\left((1+\gamma)Z-\gamma Z	\unob\unob^\top Z\right)$, the discrete version of the saddle point system \eqref{eq:KKT_matrix} is
\begin{equation}\label{eq:saddle_point}
\mathscr{S}=\begin{pmatrix}
C & B^\top\\
B & 0
\end{pmatrix}.
\end{equation}

The matrix $M_\gamma$ plays a key role in the following, thus we discuss its properties in the next Lemma.
\begin{lemma}\label{thm:M_gamma}
The matrix $M_\gamma=M\left((1+\gamma)Z-\gamma Z\unob\unob^\top Z\right)$ is symmetric and positive definite for any $\gamma\geq 0$. Its inverse is equal to
\[M_\gamma^{-1}=\left(M\left((1+\gamma)Z-\gamma Z\unob\unob^\top Z\right)\right)^{-1}=\left(\frac{1}{1+\gamma}Z^{-1}+\frac{\gamma}{1+\gamma}\unob\unob^\top\right)M^{-1}.\]
\end{lemma}
\begin{proof} 
A straightforward calculation shows that
\begin{equation*}\label{eq:symmetry}
\begin{aligned}
&(I-\unob\unob^\top Z)^\top MZ (I-\unob\unob^\top Z)= MZ- MZ\unob\unob^\top Z + Z\unob\unob^\top MZ\unob\unob^\top Z - Z\unob\unob^\top MZ\\
&=MZ -Z M\unob\unob^\top Z + Z\unob M_s\unob^\top Z - Z\unob\unob^\top MZ=MZ -Z\unob M_s\unob^\top Z,
\end{aligned}
\end{equation*}
where we used \eqref{eq:simplifing_equations} and $\unob^\top MZ\unob=\sum_{j=1}^N \zeta_j M_s= M_s$. Hence, $M_\gamma$ is symmetric since it can be written as
\[M_\gamma=MZ +\gamma MZ\left(I-\mathbbm{1}\unob^\top Z\right)=MZ+\gamma (I-\unob\unob^\top Z)^\top MZ(I-\unob\unob^\top Z).\]
The positive definiteness follows from the positiveness of the weights of the quadrature formulae. 
Finally using Woodbury identity the claim follows,
\begin{equation*}
\begin{aligned}
&\left((1+\gamma)Z-\gamma Z\unob\unob^\top Z\right)^{-1}= \frac{1}{1+\gamma}Z^{-1} -\left(\frac{1}{1+\gamma}\right)^2\unob\left(- \frac{1}{\gamma}I_s+\frac{1}{1+\gamma}\unob^\top Z\unob \right)^{-1}\unob^\top=\\
&=\frac{1}{1+\gamma}Z^{-1} -\left(\frac{1}{1+\gamma}\right)^2\unob\left(-\frac{1}{\gamma}+\frac{1}{1+\gamma}\right)^{-1}\unob^\top=\frac{1}{1+\gamma}Z^{-1}+\frac{\gamma}{1+\gamma}\unob\unob^\top .
\end{aligned}
\end{equation*} \text{ }
\end{proof}

\section{Algebraic Preconditioners}\label{sec:AlgebraicPreconditioners} In this section we study algebraic preconditioners for the saddle point matrix \eqref{eq:saddle_point} based on the seminal work \cite{murphy2000note}, where the authors showed that a saddle point matrix $\begin{pmatrix} C& B_2^\top\\ B_1 & 0\end{pmatrix}$ can be optimally preconditioned by $\mathcal{P}:=\text{diag}(C,S)$, where $S=B_1C^{-1}B_2^\top$ is the Schur complement.
However, since inverting $S$ is too expensive, one needs to use suitable
approximations. 
Let us consider the preconditioner $\widetilde{\mathcal{P}}:=\text{diag}(C,\widetilde{S}),$ obtained replacing the exact Schur complement $S$
with a symmetric positive definite approximation $\widetilde{S}$. Bounds on $\sigma(\widetilde{\mathcal{P}}^{-1}\mathscr{S})$ in terms of $\sigma(\widetilde{S}^{-1}S)$ are provided, for instance, in Proposition 2 of \cite{rees2010optimal}.
For our analysis, we will need the following, slightly more precise, result.
\begin{lemma}[Spectrum of $\widetilde{\mathcal{P}}^{-1}\mathscr{S}$]\label{thm:Pt}
The matrix $\widetilde{\mathcal{P}}^{-1} \mathscr{S}$ has eigenvalue $1$ with multiplicity $N_h$. The remaining $2N\cdot N_h$ eigenvalues are distinct and equal to 
\[\lambda_j=\frac{1+\sqrt{1+4\sigma_j}}{2},\quad \lambda_{j+N\cdot N_h}=\frac{1-\sqrt{1+4\sigma_j}}{2},\]
where $\sigma_j$ are the eigenvalues of $\widetilde{S}^{-1}S$, and $j=1,\dots, N\cdot N_h$.
\end{lemma}
\begin{proof}
An eigenpair $((\mathbf{x},\mathbf{y}),\lambda)$ of the preconditioned system $\widetilde{\mathcal{P}}^{-1}\mathscr{S}$ satisfies
\begin{equation*}
\begin{array}{rll}
\mathbf{x} &+ C^{-1}B^\top \mathbf{y}&=\lambda \mathbf{x},\\
\widetilde{S}^{-1} B \mathbf{x}& &=\lambda \mathbf{y},
\end{array}
\end{equation*}
where $\lambda \in\setR$, $\mathbf{x}\in \setR^{(N+1)N_h}$ and $\mathbf{y}\in \setR^{N\cdot N_h}$. As shown in \cite[Proposition 2]{rees2010optimal}, $((\mathbf{x},\mathbf{0}),1)$ is trivially an eigenpair for every $\mathbf{x}\in \text{Ker}B$. Thus, $\lambda=1$ is an eigenvalue with multiplicity $\text{dim}(\text{Ker}B)=N_h$.
Next, given the eigenpair $(\mathbf{v}_i,\sigma_i)$ of $\widetilde{S}^{-1}S$ for $i=1,\dots, N\cdot N_h$, it is immediate to verify  that $((\mathbf{x}_i,\mathbf{y}_i),\lambda_i)=((\frac{C^{-1}B^\top \mathbf{v}_i}{\lambda_i-1},\mathbf{v}_i),\frac{1+\sqrt{1+4\sigma_i}}{2})$ is an eigenpar of $\widetilde{\mathcal{P}}^{-1}\mathscr{S}$ as,
\begin{equation*}
\begin{aligned}
\mathbf{x}_i+C^{-1}B^\top \mathbf{y}_i&= \frac{C^{-1}B^\top \mathbf{v}_i}{\lambda_i-1}+C^{-1}B^\top \mathbf{v}_i= \frac{\lambda_i}{\lambda_i-1}C^{-1}B^\top v_i=\lambda_i \mathbf{x}_i,\\
\widetilde{S}^{-1}B\mathbf{x}_i&=\frac{1}{\lambda_i-1}\widetilde{S}^{-1}S\mathbf{x}_i=\frac{\sigma_i}{\lambda_i-1} \mathbf{v}_i= \lambda_i \mathbf{y}_i, 
\end{aligned}
\end{equation*}
where we used that $\lambda_i(\lambda_i-1)=\sigma_i.$ Similarly one verifies that $(\mathbf{x}_{N\cdot N_h+i},\mathbf{y}_{N\cdot N_h+i},\lambda_{N\cdot N_h+i})=(\frac{C^{-1}B^\top \mathbf{v}_i}{\lambda_{N\cdot N_h+i}-1},\mathbf{v}_i,\frac{1-\sqrt{1+4\sigma_i}}{2})$ is an eigenpair for $i=1,\dots,N\cdot N_h$, and the claim follows.
\end{proof}
Lemma \ref{thm:Pt} reduces the problem of estimating $\sigma(\widetilde{\mathcal{P}}^{-1}\mathscr{S})$ to the problem of estimating $\sigma(\widetilde{S}^{-1}S)$.

For the simpler deterministic OCP, the exact Schur complement is $S=A_sM_s^{-1}A_s+\frac{1}{\beta} M_s$, where $A_s$ is the stiffness matrix, $M_s$ the mass matrix and $\beta$ is the control regularization parameter. In \cite{rees2010all,rees2010optimal}, the authors approximate $S$ with $A_sM_s^{-1}A_s$, and obtain eigenvalues estimates for the preconditioned system which clearly show a $\beta$ dependence. A similar approximation in the context of robust OCPUU has been first proposed in \cite{Kouri2018}, though without a theoretical analysis. In subsection \ref{sec:firstschur}, we provide a full characterization of the spectrum of the preconditioned system which highlights both the dependence on $\beta$ and on the random field extremal values.

In subsection \ref{sec:matchschur}, we propose instead a $\beta$-robust preconditioner based on a more involved factorization of the Schur complement of \eqref{eq:saddle_point}, inspired by works on deterministic OCPs in \cite{pearson2012new,rees2010all,rees2010optimal,kourikkt,liu2020parameter,pearson2018matching} 

\subsection{A first Schur complement approximation}\label{sec:firstschur}
The exact Schur complement of the saddle point matrix $\mathscr{S}$ in \eqref{eq:saddle_point} is
\begin{equation}\label{eq:Schurcomplement_OCPUU}
S:=BC^{-1}B^\top= AM_\gamma^{-1}A+\frac{1}{\beta }Z\unob M_s\unob^\top Z.
\end{equation}
The term $AM_\gamma^{-1}A$ is block diagonal if and only if $\gamma=0$, in which case it is the direct generalization of the matrix which appears in a deterministic OCP, except that the diagonal blocks are multiplied by the weights of quadrature formula. On the other hand, the term $\frac{1}{\beta}Z\unob M_s\unob^\top Z$ is difficult to handle as it has significantly different properties from the corresponding $\frac{1}{\beta}M_s$ term of deterministic OCPs. First, $\frac{1}{\beta}Z\unob M_s\unob^\top Z$ is a dense-block matrix, where each block is given by a mass matrix. Second it is a relatively low-rank term. Its effect is to couple all the equations, increasing the difficulties to construct $\beta$ robust preconditioners. We remark that a similar low-rank perturbation appears in deterministic OCP with a control acting on a subset of the boundary, see \cite{heidel2019preconditioning}.

The first approximation $\widetilde{S}$, and corresponding preconditioner $\widetilde{P}$, we consider is obtained dropping the $\beta$-dependent low-rank term,
\begin{equation}\label{eq:definitionSt}
\widetilde{S}:=A M_\gamma^{-1} A\approx S,\quad \widetilde{P}:=\begin{pmatrix}
C & 0\\
0 &\widetilde{S}
\end{pmatrix}.
\end{equation}
Computing $\widetilde{S}^{-1}S$ we obtain
\begin{equation*}
\begin{aligned}
&\widetilde{S}^{-1}S=(AM_\gamma^{-1}A)^{-1}\left(AM_\gamma^{-1}A+\frac{1}{\beta}Z \unob M_s\unob^\top Z\right)=I+\frac{1}{\beta}A^{-1}M_\gamma A^{-1} Z \unob M_s\unob^\top Z=:I+\frac{1}{\beta}\widetilde{H},
\end{aligned}
\end{equation*}
that is, $\widetilde{S}^{-1}S$ is the identity plus a $\beta$-dependent low-rank term. Hence, $\widetilde{S}^{-1}S$ will have at most $N_h$ eigenvalues different from one since $\text{rank}\left(\unob \unob^\top Z\right)=N_h$. To study the spectrum of $\widetilde{H}$, we consider the similar matrix 
\[H:=Z\widetilde{H}Z^{-1}= ZA^{-1}M_\gamma A^{-1}Z M\unob\unob^\top=(1+\gamma)ZA^{-1} MZ A^{-1}ZM\unob\unob^\top -\gamma ZA^{-1}MZ\unob\unob^\top Z A^{-1}ZM\unob\unob^\top.\]
A characterization of the spectrum of $\widetilde{S}^{-1}S$ is provided in the following Lemma.

\begin{lemma}[Spectrum of $\widetilde{S}^{-1}S$]\label{thm:spectrumStildeinvS}
The spectrum of $\widetilde{S}^{-1}S$ satisfies $\sigma(\widetilde{S}^{-1}S)=\{1\}\cup \left\{1+\frac{1}{\beta}\mu_j\right\}_{j=1}^{N_h},$
where $\mu_j$ are the eigenvalues of $\EQ\LQ K^2_\w\RQ +\gamma \left( \EQ \LQ K^2_\w \RQ - \EQ \LQ K_\w\RQ ^2\right),$ with $K_\w:=A_\w^{-1}M_s$.
\end{lemma}
\begin{proof}

As $ZA^{-1}=\text{diag}(A^{-1}_{\w_1},A^{-1}_{\w_2},\dots,A^{-1}_{\w_N})$, direct calculations show that
\[ZA^{-1}MZA^{-1}ZM \unob\unob^\top=\begin{pmatrix}
\zeta_1 K^2_{\w_1} & \zeta_1 K^2_{\w_1} &\cdots &\zeta_1 K^2_{\w_1}\\
\zeta_2 K^2_{\w_2} & \zeta_2 K^2_{\w_2} &\cdots &\zeta_2 K^2_{\w_2}\\
\vdots & \vdots & \vdots &\vdots\\
\zeta_N K^2_{\w_N} & \zeta_N K^2_{\w_N} &\cdots &\zeta_N K^2_{\w_N}\\\end{pmatrix},
\]
and
\begin{equation*}
\begin{aligned}
ZA^{-1}MZ\unob\unob^\top ZA^{-1}ZM\unob\unob^\top&=  
\begin{pmatrix}
\zeta_1 K_{\w_1} & \cdots &\zeta_1 K_{\w_1}\\
\zeta_2 K_{\w_2} & \cdots &\zeta_2 K_{\w_2}\\
\vdots & \vdots & \vdots &\vdots\\
\zeta_N K_{\w_N} &\cdots &\zeta_N K_{\w_N}\\
\end{pmatrix}
\begin{pmatrix}
\zeta_1 K_{\w_1} & \cdots &\zeta_1 K_{\w_1}\\
\zeta_2 K_{\w_2} & \cdots &\zeta_2 K_{\w_2}\\
\vdots & \vdots  &\vdots\\
\zeta_N K_{\w_N} &\cdots &\zeta_N K_{\w_N}\\\end{pmatrix}\\
&=\begin{pmatrix}
\zeta_1 K_{\w_1} \left(\sum_{i=1}^N \zeta_i K_{\w_i}\right) &\cdots & \zeta_1 K_{\w_1} \left(\sum_{i=1}^N \zeta_i K_{\w_i}\right)\\
\zeta_2 K_{\w_2} \left(\sum_{i=1}^N \zeta_i K_{\w_i}\right) & \cdots & \zeta_2 K_{\w_2} \left(\sum_{i=1}^N \zeta_i K_{\w_i}\right)\\
\vdots & \vdots & \vdots\\
\zeta_N K_{\w_N} \left(\sum_{i=1}^N \zeta_i K_{\w_i}\right)  &\cdots & \zeta_N K_{\w_N} \left(\sum_{i=1}^N \zeta_i K_{\w_i}\right)
\end{pmatrix}
\end{aligned}
\end{equation*}
The matrix $H$ is then equal to
\begin{equation*}
H=\begin{pmatrix}
\zeta_1 K_{\w_1}^2 +\gamma\left(\zeta_1 K_{\w_1}^2-\zeta_1K_{\w_1}\left(\sum_{i=1}^N \zeta_i K_{\w_i}\right)\right) & \dots & \zeta_1 K_{\w_1}^2 +\gamma\left(\zeta_1 K_{\w_1}^2-\zeta_1K_{\w_1}\left(\sum_{i=1}^N \zeta_i K_{\w_i}\right)\right)\\
\zeta_2 K_{\w_2}^2 +\gamma\left(\zeta_2 K_{\w_2}^2-\zeta_2K_{\w_2}\left(\sum_{i=1}^N \zeta_i K_{\w_i}\right)\right) & \dots & \zeta_2 K_{\w_2}^2 +\gamma\left(\zeta_2 K_{\w_2}^2-\zeta_2K_{\w_2}\left(\sum_{i=1}^N \zeta_i K_{\w_i}\right)\right)\\
\vdots & \vdots &\vdots\\
\zeta_N K_{\w_N}^2 +\gamma\left(\zeta_N K_{\w_N}^2-\zeta_NK_{\w_N}\left(\sum_{i=1}^N \zeta_i K_{\w_i}\right)\right) & \dots & \zeta_N K_{\w_N}^2 +\gamma\left(\zeta_N K_{\w_N}^2-\zeta_NK_{\w_N}\left(\sum_{i=1}^N \zeta_i K_{\w_i}\right)\right)\\
\end{pmatrix},
\end{equation*}
and it is clearly not full rank, as only one block column is linearly independent. The rank of $H$ is $N_h$, being $N_h$ the size of $K_{\w_i}$, i.e. the number of degrees of freedom in the finite element discretization.

We then look for an eigenpair $(\lambda,\underline{\vb})$, where $\underline{\vb}=(\vb_1,\dots,\vb_N)\in \setR^{N\cdot N_h}$. 
Considering the eigenvalue equation $H\underline{\vb}=\lambda \underline{\vb}$ and denoting with $H_i$ the blocks of the matrix $H$, we get
\begin{equation}\begin{pmatrix}\label{eq:eigenvalue_system}
H_1 & \dots &H_1\\
H_2& \dots &H_2\\
\vdots & \vdots & \vdots\\
H_N & \dots &H_N\\
\end{pmatrix}\begin{pmatrix}
\vb_1\\\vb_2 \\ \vdots\\\vb_N
\end{pmatrix}=\lambda \begin{pmatrix}
\vb_1\\\vb_2 \\ \vdots\\\vb_N
\end{pmatrix}.
\end{equation}
Summing up all the equations one gets $(\sum_{i=1}^N H_i)(\sum_{i=1}^N \vb_i)=\lambda(\sum_{i=1}^N \vb_i)$. That is, if $(\lambda,\underline{\vb})$ is an eigenpair of $H$, then $(\lambda,\mathbf{w}:=\sum_{j=1}^N \vb_j)$ is an eigenpair of the reduced matrix $\sum_{i=1}^N H_i$. Thus, we can first compute the eigenpair $(\lambda,\mathbf{w})$ of $(\sum_{i=1}^N H_i)$, and then recover the eigenpair $(\lambda,\vb)$ of $H$ from the eigenvalue system \eqref{eq:eigenvalue_system}, $\vb=\frac{1}{\lambda}\text{diag}(H_1,\dots,H_N)\unob\mathbf{w}$.
Calculating explicitly $(\sum_{i=1}^N H_i)$ we obtain,
\begin{equation*}
\begin{aligned}
\sum_{i=1}^N H_i=\sum_{i=1}^N \zeta_i K_{\w_i}^2 +\gamma\LQ\sum_{i=1}^N \zeta_i K_{\w_i}^2 -\left(\sum_{j=1}^N \zeta_j K_{\w_j}\right)^2\RQ
=\EQ\LQ K^2_\w\RQ +\gamma \left( \EQ \LQ K^2_\w \RQ - \EQ \LQ K_\w\RQ ^2\right),
\end{aligned}
\end{equation*}
and the claim follows.
\end{proof}

Lemma \ref{thm:Pt} and \ref{thm:spectrumStildeinvS} guarantee that the spectrum of $\widetilde{\mathcal{P}}^{-1}\mathscr{S}$ is well clustered around 1 and $\frac{1\pm\sqrt{5}}{2}$, except for $2N_h$ eigenvalues which depend on $\beta$ and on the spectrum of $\EQ\LQ K^2_\w\RQ  +\gamma \left( \EQ \LQ K^2_\w \RQ - \EQ \LQ K_\w\RQ ^2\right)$.
To further characterize this spectrum, we study the matrices $L_{\w_i}:=A_{\w_i}^{-1}\Ms A_{\w_i}^{-1}$ and $\E\LQ L_\w\RQ$. We briefly recall some useful results, that is an inverse inequality \cite[Proposition 6.3.2]{quarteroni2009numerical},
an equivalence of norms \cite[Proposition 6.3.1 ]{quarteroni2009numerical}, and a characterization of the spectrum of the mass matrix \cite[Proposition 1.29]{elman2014finite}.

\begin{lemma}\label{thm:usefulresults}
If the triangulation $\mathcal{T}_h$ is quasi-uniform,
\begin{align}
\exists C_I>0\quad &\|\nabla v_h\|_{L^2(D)} \leq C_I h^{-1} \|v_h\|_{L^2(D)},\quad \forall v_h\in Y_h,\label{eq:inverseinequality}\\
\exists C_1,C_2>0\quad &C_1 h^d|\vb_h|^2\leq \|v_h\|^2_{L^2(D)} \leq C_2 h^d |\vb_h|^2,\quad  \forall v_h\in Y_h,\label{eq:equivalencenorm}\\
&\sigma(\Ms)\subset [C_1 h^d,C_2 h^d ]\label{eq:spectrumM_s}
\end{align}
where $\vb_h$ is the vector collecting the nodal degrees of freedom of $v_h$, $|\cdot |$ is the vector euclidean norm and $d$ is the spatial dimension.
\end{lemma}
Using the results of Lemma \eqref{thm:usefulresults}, the following Lemma can be easily proved.
\begin{lemma}\label{thm:eigen_A}
Let $\lambda(\w)$ be an eigenvalue of $A_\w$. Then, for a.e. $\w\in\Omega$ 
\[\df_{\min}(\w)\frac{C_1 h^d}{C_P^2} \leq \lambda(\w) \leq \df_{\max}(\w) C_I^2C_2 h^{d-2}, \]
where $C_P$ is the Poincar\'{e} constant.
\end{lemma}
\begin{lemma}\label{thm:eigen_L}
Defining $c_L:= \frac{C_1}{C_I^4C^2_2}$ and $C_L:=\frac{C_2C_P^4}{C_1^2}$, the following inclusions hold:
\begin{equation*}
\begin{array}{r l c}
\sigma\left(L_\w\right) &\subset\LQ \frac{c_L h^{4-d}}{\df_{\max}^2(\w)}, \frac{h^{-d}C_L }{\df_{\min}^2(\w)}\RQ,  & \text{for a.e. }\w \in \Omega,\\
\sigma\left(\EQ\LQ L_\w\RQ\right)&\subset \LQ c_L h^{4-d}\EQ\LQ \frac{1}{\df_{\max}^2(\w)}\RQ,   C_L h^{-d}\EQ\LQ \frac{1}{\df_{\min}^2(\w)}\RQ \RQ&
\end{array}
\end{equation*}
\end{lemma}
\begin{proof}
The matrix $L_\w$ is symmetric and positive definite. Its extrema eigenvalues  are characterized by the Raleigh quotients,
\begin{equation}\label{eq:spectrum_Lw}
\begin{array}{l l }
\lambda_{\max}(L_\w) &=\sup_{\vb\in \setR^{N_h}} \frac{\vb^\top A_\w^{-1}\Ms A_\w^{-1} \vb}{\vb^\top \vb}= \sup_{\mathbf{w}\in \setR^{N_h}} \frac{\mathbf{w}^\top \Ms  \mathbf{w}}{\mathbf{w}^\top A^2_\w \mathbf{w}} \leq \frac{C_2 h^d C_P^4}{\df_{\min}^2(\w) C^2_1 h^{2d}} =\frac{h^{-d}}{\df_{\min}^2(\w)} \frac{C_2 C_P^4}{C^2_1},\\
\lambda_{\min}(L_\w) &=\inf_{\vb\in \setR^{N_h}} \frac{\vb^\top A_\w^{-1}\Ms A_\w^{-1}\vb}{\vb^\top\vb}=
\inf_{\mathbf{w}\in \setR^{N_h}} \frac{\mathbf{w}^\top \Ms \mathbf{w}}{\mathbf{w}^\top A^2_\w \mathbf{w}}\geq \frac{C_1 h^d}{\df_{\max}^2(\w) C_I^4 C^2_2 h^{2d-4}} =\frac{h^{4-d}}{\df_{\max}^2(\w)}\frac{C_1}{C_I^4 C^2_2}. 
\end{array}
\end{equation}
The matrix $\EQ\LQ L_\w\RQ$ is positive definite as well, being the convex combination of positive definite matrices. Using \eqref{eq:spectrum_Lw}, its extrema eigenvalues are bounded by
\begin{equation*}
\begin{aligned}
\lambda_{\max}(\EQ\LQ L_\w\RQ) &= \sup_{\vb\in \setR^{N_h}} \frac{ \sum_{i=1}^N \zeta_i \vb^\top L_{\w_i}  \vb}{\vb^\top \vb} \leq \sum_{i=1}^N\zeta_j \sup_{\vb\in \setR^{N_h}} \frac{\vb^\top L_{\w_i} \vb}{\vb^\top\vb}\leq  h^{-d}C_L\EQ\LQ \frac{1}{\df_{\min}^2(\w)}\RQ,\\
\lambda_{\min}(\EQ\LQ L_\w\RQ) &= \inf_{\vb\in \setR^{N_h}} \frac{ \sum_{i=1}^N \zeta_i \vb^\top L_{\w_i}  \vb}{\vb^\top \vb} \geq \sum_{i=1}^N\zeta_j \inf_{\vb\in \setR^{N_h}} \frac{\vb^\top L_{\w_i} \vb}{\vb^\top\vb}\geq  h^{4-d}c_L\EQ\LQ \frac{1}{\df_{\max}^2(\w)}\RQ.
\end{aligned}
\end{equation*}\text{ }
\end{proof}

\begin{theorem}[Characterization of the spectrum of $\widetilde{\mathcal{P}}^{-1}\mathscr{S}$]\label{thm:Ptbounds}
If the triangulation $\mathcal{T}_h$ is quasi-uniform,
then the eigenvalues of $\widetilde{\mathcal{P}}^{-1}\mathscr{S}$ satisfy either
\begin{equation}\label{thm:eq:Ptbounds}
\footnotesize{
\begin{aligned}
&\lambda=\left\{1,\frac{1+\sqrt{5}}{2},\frac{1-\sqrt{5}}{2}\right\},\text{ or }\\
\frac{1}{2}\left(1+\sqrt{5 +\frac{4d_L h^4}{\beta} \EQ\LQ \frac{1}{\df_{\max}^2(\w)}\RQ}\right)\leq &\lambda \leq \frac{1}{2}\left(1+\sqrt{5 +\frac{4D_L(1+\gamma)}{\beta}\EQ\LQ\frac{1}{\df_{\min}^2(\w)}\RQ} \right),\\
\frac{1}{2}\left(1-\sqrt{5 +\frac{4D_L(1+\gamma)}{\beta} \EQ\LQ\frac{1}{\df_{\min}^2(\w)}\RQ }\right)\leq &\lambda \leq \frac{1}{2}\left(1-\sqrt{5 +\frac{4d_L h^4}{\beta}\EQ\LQ \frac{1}{\df_{\max}^2(\w)}\RQ}\right),
\end{aligned}
}
\end{equation}
where $d_L=c_LC_1$, $D_L=C_LC_2$ are constants independent on $N,\beta,\gamma,h$ and on the random set of realizations $\left\{ \w_i \right\}_{i=1}^N$.
\end{theorem}
\begin{proof}
Due to Lemma \ref{thm:Pt} and \ref{thm:spectrumStildeinvS}, $\lambda=1$ has multiplicity $N_h$, while $\lambda=\frac{1+\sqrt{5}}{2}$ and $\lambda=\frac{1-\sqrt{5}}{2}$ have each multiplicity equal to $(N-1)\cdot N_h$. Hence, we are left to estimate the $N_h$, $\beta$-dependent, eigenvalues of $\widetilde{S}^{-1}S$, which will correspond to $2N_h$ eigenvalues of $\widetilde{\mathcal{P}}^{-1}\mathscr{S}$ . To do so, we study the eigenvalues of $\EQ\LQ K_\w^2\RQ+\gamma \left( \EQ \LQ K^2_\w \RQ - \EQ \LQ K_\w\RQ ^2\right)$. We first suppose $\gamma=0$ and remark that $\EQ\LQ K_{\w}^2\RQ= \EQ\LQ L_\w \Ms\RQ= \EQ\LQ L_\w \RQ \Ms$, which is similar to $\Ms^{-1/2}\EQ\LQ L_\w \RQ M_s^{-1/2} $. 
Using Lemma \ref{thm:eigen_L}, the extrema eigenvalues of $\EQ\LQ K_{\w}^2\RQ$ are given by
\begin{equation*}\label{eq:spectrum_gamma0}
\begin{aligned}
&\lambda_{\max} (\EQ \LQ K_{\w}^2\RQ )= \sup_{\vb\in \setR^{N_h}} \frac{\vb^\top \EQ \LQ L_{\w}\RQ \vb}{\vb^\top \Ms^{-1}\vb}\leq  C_LC_2\EQ\LQ \frac{1}{\df_{\min}^2(\w)}\RQ.  \\
&\lambda_{\min} (\EQ \LQ K_{\w}^2\RQ )=\inf_{\vb\in \setR^{N_h}} \frac{\vb^\top \EQ \LQ L_{\w}\RQ \vb}{\vb^\top \Ms^{-1}\vb} \geq c_LC_1 h^4\EQ \LQ \frac{1}{\df_{\max}^2(\w)}\RQ,
\end{aligned}
\end{equation*}
thus $\sigma\left(\EQ \LQ K_{\w}^2\RQ\right)\subset \LQ d_L h^4\EQ \LQ \frac{1}{\df_{\max}^2(\w)}\RQ,D_L \EQ\LQ \frac{1}{\df_{\min}^2(\w)}\RQ \RQ$. From Lemma \ref{thm:spectrumStildeinvS} and \ref{thm:Pt}, the claim \eqref{thm:eq:Ptbounds} follows.

Next, we consider the case $\gamma\neq 0$ and observe that
\begin{equation*}
\begin{aligned}
&\EQ \LQ K^2_\w \RQ - \EQ \LQ K_\w\RQ ^2=\EQ\LQ A_\w^{-1}\Ms A_\w^{-1}M_s\RQ  -\EQ\LQ A_\w^{-1}\Ms \RQ\EQ\LQ A_\w^{-1}\Ms \RQ=\left(\EQ\LQ A_\w^{-1}\Ms A_\w^{-1}\RQ   -\EQ\LQ A_\w^{-1}\RQ \Ms \EQ\LQ A_\w^{-1}\RQ \right)\Ms\\
&= \EQ \LQ \left(A_\w^{-1} \Ms^{\frac{1}{2}}-\EQ \LQ A_\w^{-1}\Ms^{\frac{1}{2}}\RQ\right)\left(A_\w^{-1} \Ms^{\frac{1}{2}}-\EQ \LQ A_\w^{-1}\Ms^{\frac{1}{2}}\RQ\right)^\top\RQ M_s.
\end{aligned}
\end{equation*}
Thus, $\EQ \LQ K^2_\w \RQ - \EQ \LQ K_\w\RQ ^2$ can be written as the product between an expectation of a semi positive definite matrix and $M_s$, hence its eigenvalues are real and non-negative. Sharp estimates of the eigenvalues of $\EQ \LQ K^2_\w \RQ - \EQ \LQ K_\w\RQ ^2$ rely on bounds of the spectrum of $A_\w^{-1}-\EQ\LQ A_\w^{-1}\RQ$, which however are not available in terms of $\df_{\min}(\w)$ and $\df_{\max}(\w)$. To obtain an upper bound, we rely on the following estimates,
\begin{equation}\label{eq:upperbound_gamma}
\begin{array}{l l}
\lambda_{\max} \left(\EQ\LQ K_\w^2\RQ+\gamma \EQ \LQ K^2_\w \RQ - \gamma\EQ \LQ K_\w\RQ ^2\right)& \leq D_L(1+\gamma) \EQ \LQ \frac{1}{\df_{\min}^2(\w)}\RQ.\\
\end{array}
\end{equation}
To obtain a lower bound, we simply ignore the $\gamma$-dependent term,
\begin{equation}\label{eq:lowerbound_gamma}
\lambda_{\min}\left(\EQ\LQ K_\w^2\RQ+\gamma \left( \EQ \LQ K^2_\w \RQ - \EQ \LQ K_\w\RQ ^2\right)\right)\geq\lambda_{\min}\left(\EQ\LQ K_\w^2\RQ\right)=d_Lh^4\EQ \LQ \frac{1}{\df_{\max}^2(\w)}\RQ.
\end{equation}
Due to \eqref{eq:upperbound_gamma} and \eqref{eq:lowerbound_gamma}, \eqref{thm:eq:Ptbounds} follows using Lemma \ref{thm:Pt} and \ref{thm:spectrumStildeinvS}.
\end{proof}

Few comments are in order about Theorem \ref{thm:Ptbounds}. First, $\widetilde{P}$ will not lead to a $\beta$-robust convergence, as the spectrum clearly spreads as $\beta\rightarrow 0$. Second, as we increase the number of collocation points, $\EQ\LQ \frac{1}{\df_{\min}^2(\w)} \RQ$ and $\EQ\LQ \frac{1}{\df_{\max}^2(\w)} \RQ$ converge to the continuous expectations $\E\LQ \frac{1}{\df_{\min}^2(\w)} \RQ$ and $\E\LQ \frac{1}{\df_{\max}^2(\w)}\RQ$ which are finite quantities, since we assumed that Assumption \ref{ass:diff} holds for $p\geq 2$ for the well-posedness of the OCP.
Third, $\EQ\LQ \frac{1}{\df_{\min}^2(\w)} \RQ$ and $\EQ\LQ \frac{1}{\df_{\max}^2(\w)}\RQ$
represent the dependence of $\sigma(\widetilde{\mathcal{P}}^{-1}\mathscr{S})$ on the random field, and in particular the spectrum spreads as $\EQ\LQ \frac{1}{\df_{\min}^2(\w)}\RQ$ tends to zero.
Estimates on the moments of $\df_{\max}(\w)$ and $\frac{1}{\df_{\min}(\w)}$ are available, e.g., in \cite{Charrier,bonizzoni2013analysis} for the log-normal random field, 
\begin{equation}\label{eq:log-normal_field}
\df_L(x,\w)=e^{g(x,\w)}=e^{\sigma \sum_{j=1}^\infty\sqrt{\lambda_j}b_j(x)N_j(\w)},
\end{equation}
where $g(x,\w)$ is a Gaussian field with covariance function $\text{Cov}_g(x,y)$, $(b_j(x),\sigma^2\lambda_j)$ are the eigenpairs of $\mathcal{T}:L^2(D) \rightarrow L^2(D)$, $(\mathcal{T}f)(x):=\int_D \text{Cov}_g(x,y) f(y)dy$, and $N_j(\w)\sim \mathcal{N}(0,1)$.
Assuming that $b_j(x)$ are H\"{o}lder continuous with exponent $0<\alpha\leq 1$ $\forall j\geq 1$, and that $R_\alpha:=\sum_{j=1}^N \lambda_j \|b_j\|_{C^{0,\alpha}(\overline{D})}<\infty$, it holds
\[\|g\|_{L^p(\Omega,C^{0,\alpha}(\overline{D}))}\leq \sqrt{R_\alpha}\sigma ((p-1)!!)^\frac{1}{p},\]
where $(p-1)!!$ is the bi-factorial. Further, using Fernique theorem, one can show \cite[Proposition 3.10]{Charrier} that
\[\left\|\frac{1}{\df_{\min}}\right\|_{L^p(\Omega)}\leq D e^{C p \sigma^2}=:B_p, \quad\text{and}\quad \|\df_{\max}\|_{L^p(\Omega)}\leq B_p,\]
where $D$ and $C$ are constants independent on $p$ and $\sigma$. The exponential dependence over $\sigma^2$ is not dramatic, as in physical situations $\sigma^2$ is usually small: for instance setting $\sigma^2=2$, one can already model random fields which vary up to four orders of magnitude inside the domain.

To better understand the behaviour of the $N_h$, $\beta$-dependent, eigenvalues of $\widetilde{S}^{-1}S$, we consider two different random models and corresponding OCP \eqref{eq:OCP}. The first one is a log-normal random diffusion field with $\text{Cov}_g(x,y)=\sigma^2e^{\frac{-\|x-y\|_2^2}{L^2}}$, where $L$ is the correlation length. With $L=0.5$, retaining the first $M=3$ components in \eqref{eq:log-normal_field} is enough to preserve $99\%$ of the variance. 
The second random field is defined as
\begin{equation}\label{eq:bounded_random_field}
\df_B(x,y,\mathbf{\xi}):=1+\text{exp}(\sigma^2(\xi_1\cos(1.1\pi x)+\xi_2 \cos(1.2\pi x)+\xi_3\sin(1.3\pi y)+\xi_4\sin(1.4\pi y))),
\end{equation}
where $\mathbf{\xi}_i(\w) \sim \mathcal{U}([-1,1])$, $i=1,\dots,4$, and independent. We remark that $1\leq \df_B(x,y,\mathbf{\xi}(\w))\leq 1+\text{exp}(4\sigma^2)$ for all $\w \in\Omega$, thus \eqref{eq:bounded_random_field} is a uniformly bounded random field. 
We discretize the probability space using SCM with a tensorized Gauss-Hermite quadrature, for the log-normal field, and a tensorized Gauss-Legendre quadrature for the bounded random field. The number of nodes for each component is denoted with $m$.

Table \ref{Tab:Eig_wathen} shows the behaviour of smallest and largest real eigenvalues of $\widetilde{S}^{-1}S$ for different values of $\beta$, $\sigma^2$ and $m$.
As Theorem \ref{thm:Ptbounds} predicts, on the one hand, $\sigma(\widetilde{S}^{-1}S)$ is definitely spread for small values of $\beta$, on the other hand, the spectrum is well clustered for $\beta$ large. The random field \eqref{eq:bounded_random_field} is bounded from below by one, thus the constant $\EQ\LQ\frac{1}{\df_{\min}^2(\w)}\RQ$ does not deteriorate as $\sigma^2$ increases, and this results in a $\sigma(\widetilde{S}^{-1}S)$ which is bounded uniformly with respect to $\sigma^2$. The log-normal field shows instead a weak dependence on $\sigma^2$ as $\EQ\LQ\frac{1}{\df_{\min}^2(\w)}\RQ$ gets larger as $\sigma^2$ increases. The third subtable shows that the preconditioner is robust against the number of discretization points in the probability space, as expected, since 
the estimates of Theorem \ref{thm:Ptbounds} do not involve pointwise quantities such as, e.g., $\min_{\w}\df_{\min}(\w)$, but rely on empirical expectations which converge to finite quantities as $m$ increases.

Finally, we remark that $\widetilde{S}$ has favourable properties from the implementational point of view.
The major cost when applying $\widetilde{S}^{-1}$ is the matrix-vector multiplication between $A^{-1}$ and a vector $v$, which is commonly approximated using a spectrally equivalent preconditioner $\widehat{A}^{-1}$.
Due to its block diagonal structure, one can compute $\widehat{A}^{-1}v$ embarrassingly in parallel and in distributed way on a cluster. Each single node only needs to store an optimal preconditioner for one/few random matrices $A_{\w_i}$ which are readily available from the domain decomposition and multigrid literature. Further, the multiplication with the matrix $M_\gamma$ can be similarly performed in parallel if $\gamma=0$, and thus the action of the whole preconditioner $\widetilde{P}^{-1}$ is fully parallelizable. In contrast, if $\gamma\neq 0$, all nodes must communicate once at each application of $\widetilde{S}^{-1}$, as one needs to compute the expectation of $\widehat{A}^{-1}v$. However, the major cost of $\widetilde{S}^{-1}$, that is the application of $\widehat{A}^{-1}$ onto a vector, can still be performed in parallel.
\begin{table}[h]\caption{Smallest and largest real and positive eigenvalues $\lambda_{\min},\lambda_{\max}$ of $\widetilde{S}^{-1}S$ for several values of $\beta,\sigma^2$ and $m$.}\label{Tab:Eig_wathen}
\small
\centering
\begin{tabular}{| c | c | c | c | c|}
\hline
$\beta$ &  $ 10^{-2}$ & $10^{-4}$ &  $10^{-6}$ & $10^{-8}$\\ \hline
$\df_B(x,\w)$ & 1 - 1.06 &  1 - 7.12 &  1 - 613 & 1 - 61263  \\
$\df_L(x,\w)$ & 1 - 1.46  &  1 - 47.64 & 1 - 4.6e3 & 1 - 4.66e5  \\
\hline
\end{tabular}\\
\centering
$N_h=225$, $m=3$, $\sigma^2=0.5$, $\gamma=0.1$, $L^2=0.5$.\\\vspace{0.3cm}
\centering
\begin{tabular}{| c | c | c | c | c|}
\hline
$\sigma^2$ &  $ 0.1$ & $0.5$ &  $1$ & $1.5$\\ \hline
$\df_B(x,\w)$ & 1 - 1.06 &  1 - 1.06 &  1 - 1.05 & 1 - 1.05  \\
$\df_L(x,\w)$ & 1 - 1.28  &  1 - 1.46 & 1 - 1.83& 1 - 2.44  \\
\hline
\end{tabular}\\
\centering
$N_h=225$, $m=3$, $\beta=10^{-2}$, $\gamma=0.1$, $L^2=0.5$.\\\vspace{0.3cm}
\centering
\begin{tabular}{| c | c | c | c | c|}
\hline
$m$ &  $ 2$ & $3$ &  $4$ & $5$\\ \hline
$\df_B(x,\w)$ & 1 - 1.06 &  1 - 1.06 &  1 - 1.06 & 1 - 1.06  \\
$\df_L(x,\w)$ & 1 - 1.42  &  1 - 1.46 & 1 - 1.47 & 1 - 1.47  \\
\hline
\end{tabular}\\
\centering
$N_h=225$, $\sigma^2=0.5$, $\beta=10^{-2}$, $\gamma=0.1$, $L^2=0.5$.
\end{table}

\begin{remark}
As $\widetilde{\mathcal{P}}^{-1}\mathscr{S}$ has ``only'' $2N_h\lll (2N+1) N_h$ $\beta$-dependent eigenvalues, one could consider a preconditioned deflated-augmented Krylov method, see e.g. \cite{chapman1997deflated,olshanskii2010acquired}, to still obtain a $\beta$-robust convergence using $\widetilde{\mathcal{P}}$ as a preconditioner.
The key idea of a deflated-augmented Krylov method is to build a subspace $W$ which approximates the eigenspace of $\widetilde{\mathcal{P}}^{-1}\mathscr{S}$ associated to those eigenvalues responsible for the slow convergence of the standard Krylov method. Then, the solution is computed in an enhanced subspace consisting in the span of $W$ and in the Krylov subspace generated by the deflated matrix $\mathscr{S}-\mathscr{S}W\left(W^\top\mathscr{S}W\right)^{-1}W^\top\mathscr{S}$.
Due to Lemma \ref{thm:Pt}, the $2N_h$, $\beta$-dependent, eigenvalues of $\widetilde{\mathcal{P}}^{-1}\mathscr{S}$ are related to $N_h$ bad eigenvalues of $\widetilde{S}^{-1}S$. Hence, one could consider a deterministic or randomized subspace method to approximate the bad eigenspace $V$ of $\widetilde{S}^{-1}S$, and then use Lemma \ref{thm:Pt} to recover $W$.
However, it is usually suggested to store $W$ and $W^\top \mathscr{S} W$ for efficiency purposes, but in our setting the storage of the matrix $W$ is prohibitive, as it is a fully dense matrix of size $N_h(2N+1)\times 2\ell$, where $\ell\leq N_h$ is the number of eigenvectors of $\widetilde{S}^{-1}S$ to approximate. 
On the other hand, the action of $W$ can still be computed implicitly, if one stores the much smaller matrix $V\in \setR^{N_h\times \ell}$, containing the discretization of the first $\ell$ eigenvectors of $\widetilde{S}^{-1}S$. Each application of $W$ or $W^\top$ would then require to invert $C$.
A further bottleneck is the solution of the linear system with matrix $W^\top\mathscr{S}W\in \setR^{2\ell \times 2\ell}$, as Krylov methods needs ad-hoc preconditioner to solve this small linear system. 

Hence, the Deflated-Augmented Krylov approach space seems attractive only if one needs to approximate a small number $\ell$ of eigenvectors of $\widetilde{S}^{-1}S$, so that the storage of $W$ and of $W^\top\mathcal{S} W$ is feasible. Lemma \ref{thm:spectrumStildeinvS} shows that $N_h$ eigenvalues are $\beta$-dependent and numerically we remarked that all $N_h$ eigenvectors should be well approximated for $\beta \ll 1$. As the method does not scale well with $N_h$, we do not consider it further. 
\end{remark}
\subsection{Matching Schur complement technique}\label{sec:matchschur}
To get robust convergence with respect to $\beta$, we consider the matching Schur complement technique, proposed in \cite{pearson2012new,pearson2018matching} for deterministic OCP, which consists in looking for a preconditioner $\widehat{S}$ of the Schur complement factorized as
\begin{equation*}\label{eq:matching}
\begin{aligned}
&\widehat{S}=(A+\alpha \widehat{X})M_\gamma^{-1}(A+\alpha \widehat{X}^\top)=AM_\gamma^{-1}A  +\alpha^2 \widehat{X} M_{\gamma}^{-1}\widehat{X}^\top +\alpha \widehat{X} M_\gamma^{-1} A +\alpha A M_\gamma^{-1} \widehat{X}^\top,
\end{aligned}
\end{equation*} 
where $\alpha\in \setR$, $\widehat{X}\in \setR^{N\cdot N_h\times N\cdot N_h}$ are chosen such that $\alpha^2 X M_\gamma^{-1} \widehat{X}^\top=\frac{1}{\beta} Z \unob M_s\unob^\top Z.$ In other words, $\widehat{S}$ is equal to $S$, once the cross terms $\alpha \widehat{X} M_\gamma^{-1} A $ and $\alpha A M_\gamma^{-1} \widehat{X}^\top$ are neglected. For few simple deterministic OCP, it has been proven that this approximation is sufficient to obtain $\beta$ robustness \cite{pearson2012new,pearson2018matching}, without essentially increasing the computational cost compared to the approximation $S\approx A_sM^{-1}_sA_s$. Nevertheless, theoretical results are not available for several problems, see e.g. \cite{liu2020parameter,pearson2017fast}, despite the improved $\beta$ robustness has been confirmed by numerical examples. In this subsection, we apply this technique to our model problem and we partially characterize the spectrum of the preconditioned Schur complement. Finally we present numerical experiments confirming the improved $\beta$ robustness, and discuss the additional computational costs compared to $\widetilde{S}$.

Defining $\alpha:=\frac{1}{\sqrt{\beta}}$ and $\widehat{X}:=Z \unob M_s \unob^\top Z$, a direct calculation shows
\begin{equation*}\label{eq:plr}
\begin{aligned}
&\alpha^2\widehat{X}M_\gamma^{-1}\widehat{X}^\top=\alpha^2 Z  \unob M_s\unob^\top Z\left(\frac{Z^{-1}}{1+\gamma}+\frac{\gamma}{1+\gamma}\unob\unob^\top\right)M^{-1}Z\unob M_s\unob^\top Z=\\
&\alpha^2 Z\unob M_s\unob^\top Z\left(\frac{Z^{-1}}{1+\gamma}+\frac{\gamma}{1+\gamma}\unob\unob^\top\right)Z\unob\unob^\top Z= \alpha^2 Z \unob M_s \left(\frac{\unob^\top Z\unob}{1+\gamma}+\frac{\gamma(\unob^\top Z\unob)^2}{1+\gamma}\right)\unob^\top Z\\
&= \frac{1}{\beta} Z \unob M_s\unob^\top Z.  
\end{aligned}
\end{equation*} 
where we used $M^{-1}Z\unob M_s \unob^\top =Z\unob\unob^\top$ as $M$ has constant diagonal blocks and $M^{-1}Z=ZM^{-1}$, and $\unob^\top Z\unob=1$. Note further that $\widehat{X}^\top=\widehat{X}$. We thus study the preconditioner Schur complement $\SLR$ and the associated preconditioner $\PLR$, 
\begin{equation}\label{eq:SLR}
\SLR:=\left(A+\frac{1}{\sqrt{\beta}} Z \unob M_s \unob^\top Z \right)M_\gamma^{-1} \left(A+\frac{1}{\sqrt{\beta}}Z \unob M_s \unob^\top Z\right),\quad \PLR:=\begin{pmatrix}
C & 0\\
0 & \SLR\end{pmatrix}
\end{equation}
which is symmetric and positive definite. The subscript $LR$ stands for \textit{Low-Rank}, as both parenthesis \eqref{eq:SLR} involve a low-rank perturbation.
We partially characterize the spectrum of $\SLR^{-1}S$ in the following theorem.
\begin{theorem}[Spectrum of $\SLR^{-1}S$]\label{thm:spectrum_SLR}
The matrix $\SLR^{-1}S$ has the eigenvalue $\lambda=1$ with geometric multiplicity equal to $(N-2)N_h$.
\end{theorem}
\begin{proof}
To study the spectrum of $\SLR^{-1}S$ we consider the generalized eigenvalue problem $S\vb=\lambda\SLR \vb$, and we define the subspaces $\mathcal{H}:=\left\{\vb\in \setR^{N\cdot N_h}: \unob^\top Z \vb=\sum_{j=1}^N\zeta_j \vb_j=0\right\}$, and\\
 $\mathcal{K}:=\left\{\vb\in\setR^{N\cdot N_h}: \unob^\top  A \vb=\sum_{j=1}^N \zeta_j A_{\w_j}\vb_j= 0\right\}.$ Both $\mathcal{H}$ and $\mathcal{K}$ have dimension $(N-1)N_h$, and their intersection $\mathcal{H}\cap \mathcal{K}$ has dimension $(N-2)N_h$. We claim that any $\vb \in \mathcal{H}\cap \mathcal{K}$ satisfies $S\vb=1\cdot \SLR\vb$ and thus it is an eigenvector of $\SLR^{-1}S$ associated to $\lambda=1$. Indeed,
\begin{equation*}
\begin{aligned}
&S\vb=(AM_\gamma^{-1}A+\frac{1}{\beta }Z\unob M_s\unob^\top Z)\vb\\
&=(AM_\gamma^{-1}A+\frac{1}{\beta} Z\unob M_s\unob^\top Z +\frac{1}{\sqrt{\beta}}\left(Z\unob\unob^\top MZM_\gamma^{-1}A + AM_\gamma^{-1} ZM\unob\unob^\top Z ) \right)\vb\\
&=(AM_\gamma^{-1}A+\frac{1}{\beta} Z\unob M_s\unob^\top Z +\frac{1}{\sqrt{\beta}}\left(Z\unob\unob^\top A + A\unob\unob^\top Z ) \right)\vb= \SLR \vb,
\end{aligned}
\end{equation*}
where we used the equality
\[\unob^\top MZ M_\gamma^{-1}A=\unob^\top MZ \left(\frac{1}{1+\gamma} Z^{-1}+\frac{\gamma}{1+\gamma}\unob\unob^\top\right)M^{-1}A=\unob^\top A.\]
\end{proof}
Theorem \ref{thm:spectrum_SLR} guarantees that $\SLR^{-1}S$ has $(N-2)N_h$ eigenvalues equal to 1, but does not provide estimates for the remaining $2N_h$ eigenvalues. To further analyse the spectrum of $\SLR^{-1}S$, let us define $\widetilde{X}:=M_\gamma^{-\frac{1}{2}}A$ and $\widetilde{Y}:=\frac{1}{\sqrt{\beta}}M_\gamma^{-\frac{1}{2}}Z\unob M_s \unob^\top Z$. $\widetilde{X}$ is an invertible matrix, while $\widetilde{Y}$ has rank $N_h$. Algebraic manipulations show that
\[S= \widetilde{X}^\top \widetilde{X} + \widetilde{Y}^\top \widetilde{Y}\quad \text{and}\quad \SLR=(\widetilde{X}+\widetilde{Y})^\top(\widetilde{X}+\widetilde{Y}).\]
To get a lower bound on $\sigma(\SLR^{-1}S)$, we rely on the following theorem.
\begin{theorem}[Theorem 1,\cite{pearson2017fast}]\label{thm:spectrum_SLR_below}
Let $K$ and $\widetilde{K}$ be generic invertible matrices satisfying
\[K= \widetilde{X}^\top \widetilde{X} + \widetilde{Y}^\top \widetilde{Y}\quad \text{and}\quad \widetilde{K}=(\widetilde{X}+\widetilde{Y})^\top(\widetilde{X}+\widetilde{Y}),\]
with real $\widetilde{X}$ and $\widetilde{Y}$. Then the eigenvalues of $\widetilde{K}^{-1}K$ are real and greater than $\frac{1}{2}$.
\end{theorem}
To estimate an upper bound for $\sigma(\SLR^{-1}S)$, one could consider the Raleigh quotient
\[R(\vb):=\frac{\vb^\top S \vb}{\vb^\top \SLR \vb}=\frac{\vb^\top \widetilde{X}^\top \widetilde{X}\vb+\vb^\top \widetilde{Y}^\top \widetilde{Y}\vb}{\vb^\top \widetilde{X}^\top \widetilde{X} \vb+\vb^\top \widetilde{Y}^\top \widetilde{Y} \vb+\vb^\top \widetilde{X}^\top \widetilde{Y}\vb +\vb^\top \widetilde{Y}^\top \widetilde{X}\vb}.\]
For $\vb\in \text{Ker} \widetilde{Y}$, we have $R(\vb)=1$. Taking $\vb\notin \text{Ker} \widetilde{Y}$ and dividing numerator and denominator by $\vb^\top \left(\widetilde{X}^\top \widetilde{X} + \widetilde{Y}^\top \widetilde{Y}\right)\vb\neq 0$, one obtains
\begin{equation}\label{eq:quotien_Raileigh}
R(\vb)=\frac{1}{1+\frac{\vb^\top \left( \widetilde{X}^\top \widetilde{Y} +\widetilde{Y}^\top \widetilde{X}\right)\vb}{\vb^\top \left(\widetilde{X}^\top \widetilde{X} +\widetilde{Y}^\top \widetilde{Y}\right) \vb}}.
\end{equation}
For a deterministic OCP \cite{pearson2012new}, $\widetilde{X}=M_s^{-\frac{1}{2}}A_s$ and $\widetilde{Y}= \frac{1}{\sqrt{\beta}}M_s^{\frac{1}{2}}$ so that  $\widetilde{X}^\top \widetilde{Y} +\widetilde{Y}^\top \widetilde{X}=\frac{2}{\sqrt{\beta}}A_s$, is positive definite, thus \eqref{eq:quotien_Raileigh} implies $R(\vb)\leq 1$, hence $\sigma(\widehat{S}^{-1}S)\subset \LQ\frac{1}{2},1\RQ$.
Unfortunately in our case, like in \cite{liu2020parameter,pearson2017fast}, $\widetilde{X}^\top \widetilde{Y} +\widetilde{Y}^\top \widetilde{X}=\frac{1}{\sqrt{\beta}}\left(Z\unob \unob^\top A + A\unob\unob^\top Z\right)$ is indefinite.
Numerically, we have observed that $\SLR^{-1}S$ has $N_h$ eigenvalues in the interval $[\frac{1}{2},1]$ and the remaining $N_h$ are slightly larger than 1. We refer to Tables \ref{Tab:Eig_matching} for a further discussion.

\subsubsection{Mean and Chebyshev semi-iterative approximations}
The application of $\SLR^{-1}$ requires to invert the symmetric and positive definite matrix $\left(A+\frac{1}{\sqrt{\beta}} Z\unob M_s \unob^\top Z\right)$, which consists in a full-rank matrix plus a low-rank perturbation. 
To do so, we use the Woodbury identity
\begin{equation}\label{eq:Woodbury}
\begin{aligned}
\left(A+\frac{1}{\sqrt{\beta}}Z \unob M_s \unob^\top Z\right)^{-1}&=\left(A+Z\unob \frac{1}{\sqrt{\beta}} M_s\unob^\top Z\right)^{-1}=A^{-1}\left(I-Z\unob\LQ \sqrt{\beta}M_s^{-1} +\unob^\top Z A^{-1} Z\unob\RQ^{-1}\unob^\top Z A^{-1}\right)\\
&=A^{-1}\left(I-Z\unob\LQ I +\frac{1}{\sqrt{\beta}}M_s\unob^\top Z A^{-1}Z\unob\RQ^{-1}\frac{1}{\sqrt{\beta}}M_s\unob^\top Z A^{-1}\right).
\end{aligned}
\end{equation}
Unfortunately, \eqref{eq:Woodbury} is of no practical use as it requires to solve a linear system with $L:=\LQ I +\frac{1}{\sqrt{\beta}}M_s\unob^\top Z A^{-1}Z\unob\RQ$, which involves $\unob^\top ZA^{-1}Z\unob=\sum_{i=1}^N \zeta_i A_{\w_i}^{-1}$.
To make the approach feasible, we propose two different approximations.

The first one is based on the mean approximation $\sum_{i=1}^N \zeta_i A_{\w_i}^{-1}\approx  A_0^{-1}$, that is we replace the sum of the inverses with the inverse of the mean matrix $A_0=\sum_{i=1}^N \zeta_i A_{\w_i}$.
Then,
\begin{equation}\label{eq:Woodbury_mean}
\begin{aligned}
\left(A+\frac{1}{\sqrt{\beta}} Z \unob M_s \unob^\top Z\right)^{-1}&\approx A^{-1}\left(I- Z\unob\LQ I +\frac{1}{\sqrt{\beta}}M_s A_0^{-1}\RQ^{-1}\frac{1}{\sqrt{\beta}}M_s\unob^\top Z A^{-1}\right)\\
&=A^{-1}\left(I- Z\unob A_0 \LQ A_0+\frac{1}{\sqrt{\beta}}M_s\RQ ^{-1} \frac{1}{\sqrt{\beta}}M_s\unob^\top Z A^{-1}\right).
\end{aligned}
\end{equation}
We will denote with $\SLRM$ the preconditioner \eqref{eq:SLR}, where the inverse of the parenthesis are approximated through \eqref{eq:Woodbury_mean}, and the associated preconditioner by 
\begin{equation}\label{eq:PLRM}
\PLRM:=\begin{pmatrix}
C& 0\\
0 & \SLRM
\end{pmatrix}.
\end{equation}
As for forward problems \cite{powell2009block}, this approximation is satisfactory if the variance of the random field is small, while it is definitely poor if the variance is large. 

As an alternative approximation in the large variance case, it would be tempting to use a Krylov method to approximate the inverse of $L$. However, any Krylov method is a non-linear map with respect to the right hand side and the initial vector, and thus it would lead to a non-linear preconditioner for the global saddle point system, see \cite{wathen2008chebyshev} for a detailed discussion. We propose here instead to approximate the solution of $L\vb=\mathbf{z}$ using $\Nit$ iterations of the damped preconditioned stationary iterative method that, starting from an initial guess $\vb^0$, computes
\[\vb^k=\vb^{k-1}+\alpha P_0^{-1}(\mathbf{z}-L\vb^{k-1}),\quad k=1,\dots,\Nit,\]
accelerated by the Chebyshev Semi-Iterative method, see e.g \cite[Section 10.1.5]{golub2013matrix}. The preconditioner Schur complement obtained by approximating the inverse of $L$ with such iterative procedure is denoted with $\SLRC$, and the associated preconditioner
\begin{equation}\label{eq:PLRC}
\PLRC:=\begin{pmatrix}
C & 0\\
0 & \SLRC
\end{pmatrix}.
\end{equation}
In our experiments, we set $P_0^{-1}=(I+\frac{1}{\sqrt{\beta}}M_sA_0^{-1})^{-1}= A_0\left(A_0+\frac{1}{\sqrt{\beta}}M_s\right)^{-1}$. 

The Chebyshev Semi-Iterative method requires two parameters $\underline{\lambda}$ and $\overline{\lambda}$ such that $-1<\underline{\lambda}\leq\lambda_1\leq\cdots\leq\lambda_N\leq\overline{\lambda}<1$, where $\lambda_j$ are the eigenvalues of $I-\alpha P_0^{-1}L$. Once given $\underline{\lambda}$ and $\overline{\lambda}$, the Chebyshev Semi-Iterative method computes a sequence of polynomials in the matrix which depends only on $\underline{\lambda}$ and $\overline{\lambda}$, thus the method is linear with respect to the initial guess and right hand side (if $\Nit$ is fixed) \cite{wathen2008chebyshev}. To estimate the spectrum of $(I-\alpha P_0^{-1}L)$, we rely on the following Lemma.

\begin{lemma}\label{lemma:spectrumPinvL}
The spectrum of $P_0^{-1}L$ is real and bounded from below by $1$.
\end{lemma}
\begin{proof}
Algebraic manipulations lead to
\begin{align*}
P_0^{-1}L&=\left(M_s^{-1}+\frac{1}{\sqrt{\beta}}A_0^{-1}\right)^{-1}\left(M_s^{-1}+\frac{1}{\sqrt{\beta}}\unob^\top Z A^{-1} Z\unob +\frac{1}{\sqrt{\beta}}A_0^{-1}-\frac{1}{\sqrt{\beta}}A_0^{-1}\right)\\
&=I+\left(M_s^{-1}+\frac{1}{\sqrt{\beta}}A_0^{-1}\right)^{-1}\frac{1}{\sqrt{\beta}}\left(\unob^\top Z A^{-1} Z\unob -A_0^{-1}\right).
\end{align*}
Hence, if $\sum_{j=1}^N \zeta_j A^{-1}_{\w_j} -A_0^{-1}$ is semi-positive definite, then $P_0^{-1}L$ has real eigenvalues and $\lambda_{\min}>1$.
To show this, take an arbitrary $0\neq\vb\in \setR^{N_h}$, and consider the map $\phi_{\vb}:S^n_{++}\to \setR$, where $S^n_{++}$ is the set of positive definite matrices in $\setR^{n\times n}$, defined as $\phi_{\vb}:= \vb^\top A^{-1}\vb$.
The map $\phi_{\vb}$ is convex \cite[Lemma 1]{Whittle}. Thus due to Jensen's inequality
\[\vb^\top \left(\sum_{j=1}^N \zeta_j A_{\w_j}^{-1}\right)\vb - \vb^\top \left(\sum_{j=1}^N \zeta_j A_{\w_j}\right)^{-1}\vb=\sum_{j=1}^N \zeta_j \phi_{\vb}(A_{\w_j})-\phi_\vb\left(\sum_{j=1}^N \zeta_j A_{\w_j}\right)\geq 0,\] 
hence, due to the arbitrariness of $\vb$, $\unob^\top Z A^{-1}Z\unob -A_0^{-1}$ is semi-positive definite. 
\end{proof}
Let $\lambda_{\min}$ and $\lambda_{\max}$ be the minimum and maximum eigenvalues of $P_0^{-1}L$. From Lemma \eqref{lemma:spectrumPinvL} follows that $\sigma(I-\alpha P_0^{-1}L)\subset [1-\alpha\lambda_{\max},1-\alpha\lambda_{\min}]$ as $P_0^{-1}L$ has real and positive spectrum. The parameter $\alpha$ is needed to guarantee the convergence of the stationary method, that is $\rho(I-\alpha P_0^{-1}L)<1$. The optimal alpha which minimizes $\rho(I-\alpha P_0^{-1}L)$ is $\alpha_\text{opt}=\frac{2}{\lambda_{\min}(P_0^{-1}L)+\lambda_{\max}(P_0^{-1}L)}$. However, $\alpha_\text{opt}$ leads to a spread spectrum, while the Chebyshev Semi-Iterative method takes advantage of clustered spectra, like Krylov methods \cite[Section 10.1.5]{golub2013matrix}. We therefore set $\alpha:=\frac{1}{1+\lambda_{\max}(P_0^{-1}L)}$, where $\lambda_{\max}(P_0^{-1}L)$ is approximated, once and for all, using few iterations of the power method. This choice guarantees the convergence of the iterative method since $\alpha\leq \frac{2}{\lambda_{\max}(P_0^{-1}L)}$.
Finally, in the Chebyshev Semi-Iterative method we take $\underline{\lambda}=1-\alpha\lambda_{\max}$ and $\overline{\lambda}=1-\alpha$. 

Table \ref{Tab:Eig_matching} shows the behaviour of the extrema eigenvalues of the Schur complement preconditioned by $\SLR$, $\SLRM$ and $\SLRC$ in different regimes.  
We stress that $\SLR$ has a very high computational cost and is of no practical use. It is included in Table \ref{Tab:Eig_matching} as a reference, in order to assess how well the approximated versions $\SLRM$ and $\SLRC$ perform, compared to $\SLR$.

From the first two tables, we observe that $\SLR$ shows a (very weak) dependence on $\beta$ and on $\sigma^2$, emphasized in the case of the log-normal field, but the spectrum still remains sufficiently clustered. 
Anyway, $\sigma(\SLR^{-1}S)$ is not contained in the interval $[\frac{1}{2},1]$, as in the deterministic case \cite{pearson2012new}.
The third table shows that $\SLR$ is robust with respect to finer discretizations of the probability space.

Let us now consider the approximations $\SLRM$ and $\SLRC$. On the one hand, the performance of $\SLRM$ is highly affected by the variance of the random field. $\SLRM$ is a valid choice for small values of the variance and for values of $\beta$ not too small. It definitely performs poorly in the remaining regimes.
On the other hand, $\SLRC^{-1}$ matches the performance of the exact preconditioner $\SLR$, both for the bounded and log-normal fields, with a small number $\Nit$ of Chebyshev semi-iterations. However, to obtain good performances, $\Nit$ has to increase as $\sigma^2$ increases, especially for the log-normal field, due to the poorer performance of $P_0$ as a preconditioner in the inner Chebyshev semi-iterations. More accurate preconditioners, which capture better the effective spectrum of $\mathbbm{1}^\top Z A^{-1}Z\mathbbm{1}$, are expected to reduce the number of inner iterations needed, or even to replace directly the mean approximation $A_0$ into \eqref{eq:Woodbury_mean}, leading to a modified $\SLRM$, and thus to reduce the overall computational time, see Section \ref{sec:num_lognormal} for further comments.

\begin{table}[h]\caption{\small{Smallest and largest real and positive eigenvalues $\lambda_{\min}-\lambda_{\max}$ of $\SLR^{-1}S$, $\SLRM^{-1}S$ and $\SLRC^{-1}S$}}\label{Tab:Eig_matching}
\centering
\small
\begin{tabular}{| c | c | c | c | c | c|}
\hline
$\beta$ &  &$ 10^{-2}$ & $10^{-4}$ &  $10^{-6}$ & $10^{-8}$\\ \hline
$\SLR^{-1}S$ &$\df_B(x,\w)$ & 0.68 - 1.00 &  0.50 - 1.02 &  0.50 - 1.17 & 0.50 - 1.30  \\
$\SLRM^{-1}S$ & $\df_B(x,\w)$ & 0.68 - 1.00 &  0.48 - 1.02 &  0.13 - 1.07 & 8.7e-3 - 30.59  \\
$\SLRC^{-1}S$ &$\df_B(x,\w)$ & 0.68 - 1.00 &  0.50 - 1.02 &  0.50 - 1.17 & 0.50 - 1.30  \\
$\SLR^{-1}S$ & $\df_L(x,\w)$ & 0.52 - 1.11  &  0.50 - 1.74 & 0.50 - 2.39& 0.52 - 2.61  \\
$\SLRM^{-1}S$ & $\df_L(x,\w)$ & 0.45 - 1.12  &  0.02 - 2.13 & 1e-4 - 7.73e2& 1.3e-5 - 8.98e4\\
$\SLRC^{-1}S$ & $\df_L(x,\w)$ & 0.52 - 1.11  &  0.50 -1.74 & 0.50 - 2.39& 0.52 - 2.61  \\
\hline
\end{tabular}
\\\centering
$N_h=225$, $m=3$, $\sigma^2=0.5$, $\gamma=0.1$, $L^2=0.5$. $\Nit=2$ for $\df_B(x,\w)$ and $\Nit=4$ for $\df_L(x,\w)$.\\\vspace{0.3cm}
\begin{tabular}{| c | c | c | c | c | c|}
\hline
$\sigma^2$ &  & $ 0.1$ & $0.5$ &  $1$ & $1.5$\\ \hline
$\SLR^{-1}S$ &$\df_B(x,\w)$ & 0.50 - 1.04 &  0.50 - 1.30 &  0.50 - 1.66 & 0.50 - 1.98\\
$\SLRM^{-1}S$ & $\df_B(x,\w)$ & 0.49 - 1.01 &  8.7e-3 - 30.59 &  3.7e-8 - 1.00e3 & 4.4e-5 - 1.03e4  \\
$\SLRC^{-1}S$ &$\df_B(x,\w)$ & 0.50 - 1.04  &  0.50 - 1.30 & 0.49 - 1.67 & 0.09 - 1.97 \\
$\SLR^{-1}S$ & $\df_L(x,\w)$ &  0.52 - 1.43  &  0.52 - 2.61 & 0.52 - 4.35&  0.52 - 6.54  \\
$\SLRM^{-1}S$ & $\df_L(x,\w)$ & 5.9e-4 - 1.52e3  &  1.3e-5 - 8.98e4 & 0.23 - 9.82e5& 0.70 - 5.87e6\\
$\SLRC^{-1}S$ & $\df_L(x
,\w)$ & 0.52 - 1.43  &  0.52 - 2.61 & 0.52 - 4.34 & 0.51 - 6.54  \\
\hline
\end{tabular}
\\
\centering
$N_h=225$, $m=3$, $\beta=10^{-8}$, $\gamma=0.1$, $L^2=0.5$. $\Nit$ is equal to $2$ for $\df_B(x,\w)$ and equal to $2,4,6,8$ for $\sigma^2=0.1,0.5,1,1.5$ respectively for $\df_L(x,\w)$.\\\vspace*{0.3cm}
\begin{tabular}{| c | c | c | c | c | c|}
\hline
$m$ &  & $ 2$ & $3$ &  $4$ & $5$\\ \hline
$\SLR^{-1}S$ &$\df_B(x,\w)$ & 0.50 - 1.17 &  0.50 - 1.17 &  0.50 - 1.17 & 0.50 - 1.17\\
$\SLRM^{-1}S$ & $\df_B(x,\w)$ & 0.13 - 1.07 &  0.13 - 1.07 &  0.13 - 1.07 & 0.13 - 1.07\\
$\SLRC^{-1}S$ &$\df_B(x,\w)$ & 0.50 - 1.17 &  0.50 - 1.17 &  0.50 - 1.17 & 0.50 - 1.17\\
$\SLR^{-1}S$ & $\df_L(x,\w)$ & 0.50 - 2.13  &  0.50 - 2.39 &   0.50 - 2.43&  0.50 - 2.43 \\
$\SLRM^{-1}S$ & $\df_L(x,\w)$ & 0.0025 - 634  &  1e-6 - 773 & 1e-4 - 784 & 2.9e-3 - 785\\
$\SLRC^{-1}S$ & $\df_L(x,\w)$ & 0.50 - 2.13 &  0.50 - 2.39 & 0.5 - 2.43& 0.5 - 2.43  \\
\hline
\end{tabular}
\\\centering
$N_h=225$, $\sigma^2=0.5$, $\beta=10^{-6}$, $\gamma=0.1$, $L^2=0.5$, $\Nit=2$ for $\df_B$ and $\Nit=4$ for $\df_L$.
\end{table}
We finally remark that both $\SLRM$ and $\SLRC$ require to invert four times (approximately and possibly in parallel) the matrix $A$ at each outer Krylov iteration, in constrast with $\widetilde{S}$ which requires to invert (approximately) $A$ only twice per iteration. There is further a synchronization step where the reduced size system involving the matrix $L$, or its mean approximation, is approximately solved.

\section{Preconditioning in a Hilbert setting}\label{sec:OP}
Another technique to develop preconditioners for parameter-dependent saddle point problems is often called ``operator preconditioning'', which has its foundation in the analysis of iterative methods in Hilbert spaces \cite{malek2014preconditioning,mardal2011preconditioning,functional_iterative}. In a nutshell, let $\mathscr{S}$ be a self-adjoint operator from $\mathcal{V}\rightarrow \mathcal{V}^\prime$, and suppose to solve the linear equation $\mathscr{S}x=f$ in $\mathcal{V}^\prime$. As $\mathscr{S}$ is a map between two different spaces, iterative methods cannot be applied, unless one identifies a isomorphism $\mathcal{R}:\mathcal{V}^\prime \rightarrow \mathcal{V}$, and consider the equivalent problem $\mathcal{R}\mathscr{S}x=\mathcal{R}f$ in $\mathcal{V}$. It is natural to choose $\mathcal{R}$ self-adjoint and positive definite, so that $\mathcal{R}^{-1}$ defines the scalar product $(x,y)_\mathcal{V}=\langle \mathcal{R}^{-1} x,y\rangle$, and $\mathcal{R}\mathscr{S}$ is still self-adjoint with respect to $(\cdot,\cdot)_\mathcal{V}$. Instead of coming up with an operator $\mathcal{R}$, one can define first a scalar product on $\mathcal{V}$, and then choose $\mathcal{R}$ equal the Riesz isomorphism such that $\langle \mathscr{S} x,y\rangle = (\mathcal{R} \mathscr{S} x, y)_\mathcal{V}$, $\|\mathcal{R} \mathscr{S}x\|_{\mathcal{V}}=\|\mathscr{S} x\|_{\mathcal{V}^\prime}$, so that
\begin{equation}\label{eq:theory_op}
\begin{aligned}
\|\mathcal{R}\mathscr{S}\|_{\mathcal{L}(\mathcal{V},\mathcal{V})}&=\sup_{0\neq x\in\mathcal{V}} \frac{\|\mathcal{R}\mathscr{S}x\|_\mathcal{V}}{\|x\|_\mathcal{V}}=\sup_{0\neq x\in\mathcal{V}} \frac{\|\mathscr{S}x\|_{\mathcal{V}^\prime}}{\|x\|_\mathcal{V}}=\|\mathscr{S}\|_{\mathcal{L}(\mathcal{V},\mathcal{V}^\prime)},\\
\|\left(\mathcal{R}\mathscr{S}\right)^{-1}\|_{\mathcal{L}(\mathcal{V},\mathcal{V})}&=\sup_{0\neq x\in \mathcal{V}} \frac{\|\left(\mathcal{R}\mathscr{S}\right)^{-1}x\|_\mathcal{V}}{\|x\|_\mathcal{V}}=\left( \inf_{0\neq x\in\mathcal{V}} \frac{\|\mathcal{R}\mathscr{S}x\|_\mathcal{V}}{\|x\|_\mathcal{V}}\right)^{-1}=\left( \inf_{0\neq x\in\mathcal{V}} \frac{\|\mathscr{S}x\|_\mathcal{V^\prime}}{\|x\|_\mathcal{V}}\right)^{-1}=\|\mathscr{S}^{-1}\|_{\mathcal{L}(\mathcal{V},\mathcal{V}^\prime)}.
\end{aligned}
\end{equation}
Hence, if one finds an appropriate, parameter-dependent, scalar product $(\cdot,\cdot)_\mathcal{V}$ (hence, a norm on $\mathcal{V}$), so that $\|\mathscr{S}\|_{\mathcal{L}(\mathcal{V},\mathcal{V}^\prime)}\leq C$ and $\|\mathscr{S}^{-1}\|_{\mathcal{L}(\mathcal{V}^\prime,\mathcal{V})}\leq \alpha$, with $C$ and $\alpha$ parameter-independent, then considering the Riesz isomorphism $\mathcal{R}$ associated to $(\cdot,\cdot)_\mathcal{V}$, one obtains using \eqref{eq:theory_op}, $\kappa (\mathcal{R}\mathscr{S})=\|\mathcal{R}\mathscr{S}\|_{\mathcal{L}(\mathcal{V},\mathcal{V})}\|\left(\mathcal{R}\mathscr{S}\right)^{-1}\|_{\mathcal{L}(\mathcal{V},\mathcal{V})}\leq C \alpha$, that is, the condition number of the the preconditioned system $\mathcal{R}\mathscr{S}$ is independent on the parameters. 
In this section, we apply this operator preconditioning paradigm to the robust optimal control \eqref{eq:OCP}, and we set the functional space of control functions $U$ 
equal to $\Y$. 

Let us consider the optimality conditions in \eqref{eq:optimality_system}.
We introduce the space $\Xhat:=\ES \times U$ and the bilinear forms
\begin{equation*}\label{eq:bilinear_forms_saddle}
\begin{aligned}
\I&:\Y^\star \rightarrow \ES^\star\quad \text{such that}\quad  \langle \I f,v\rangle :=\int_\Omega \langle  f,v(\cdot,\w)\rangle d\PP(\w)=\E\LQ \langle f,v_\w\rangle\RQ,\quad \forall v\in \ES,\\ 
\C&:\Xhat\times\Xhat\rightarrow \setR\quad \text{such that}\quad \C\left((y,u),(w,v)\right):=   \E \LQ \langle \Lambda_{L^2}((1+\gamma)y_\w-\gamma \E\LQ y_\w\RQ),w_\w\rangle\RQ + \beta \langle \Lambda_Y u,v\rangle,\\
\B&:\Xhat\times \ES \rightarrow \setR\quad \text{such that}\quad \B\left((y,u),p)\right):= \E\LQ \langle \mathcal{A}_\w y_\w, p_\w\rangle -\langle \Lambda_Y u,p_\w\rangle\RQ,
\end{aligned}
\end{equation*}
The optimality conditions can be formulated as: 
\begin{equation}\label{eq:saddle_point_cont}
\begin{array}{r l l}
\text{Find $(\xu,p)\in \Xhat\times \ES$ such that}&
\C(\xu,\ru) + \B(\ru,p)&=\langle \mathcal{F},\ru\rangle,\quad \forall \ru=(w,v)\in \Xhat,\\
&\B(\xu,q)&=\langle \mathcal{G},q\rangle,\quad \forall q\in \ES,
\end{array}
\end{equation}
where $\mathcal{F} \in \Xhat^\star:$ $\langle \mathcal{F},\ru\rangle = \E\LQ\langle  \Lambda_{L^2} y_d,w_\w\rangle\RQ$, $\forall\ru=(w,v)\in \Xhat$, and $\mathcal{G}\in \ES^\star:$ $\langle \mathcal{G},q\rangle =\E \LQ \langle f,q_\w\rangle \RQ$, $\forall q\in \ES$.
The bilinear form $\C(\cdot,\cdot)$ is  symmetric, as a direct generalization of Lemma \ref{thm:M_gamma} shows.

To obtain parameters-robust estimates, we have to consider a slightly modified formulation of \eqref{eq:saddle_point_cont}. Let us define the subspace in $\ES$ of functions with zero average, $G:=\left\{v \in \ES : \E\LQ v(\cdot,\w) \RQ=0 \right\}$ and its polar space $G^0:=\left\{ \mathcal{F}\in \ES^\star: \mathcal{F}(v) =0,\text{ }\forall v \in G.\right\}$.
We can prove the following Lemma.

\begin{lemma}[Isomorphism between $Y^\star$ and $G^0$]\label{lemma:G0}
$\I$ is an isomorphism between $\Y^\star$ and $G^0$.
\end{lemma}
\begin{proof} To prove that $\I$ is injective, we show that for any $\widetilde{u}$,$\widetilde{w} \in \Y^\star$, $\I\widetilde{u}=\I\widetilde{w}$ in $\ES^\star$ implies $\widetilde{u}=\widetilde{w}$. A direct calculation leads to
\begin{equation}\label{eq:settozero}
\langle \I\widetilde{u}-\I\widetilde{w},v\rangle =\E\LQ \langle \widetilde{u}-\widetilde{w},v_\w\rangle \RQ=0,\quad \forall v \in \ES.
\end{equation}
Consider now the sets $\Gamma^n:=\left\{\w \in \Omega: \max\left\{\df_{\max}(\w),1/\df_{\min}(\w)\right\}<n \right\}$. The sets $\Gamma^n$ are measurable, and $|\Gamma^n|>0$ for a sufficiently large $n$.
Taking $v(x,\w)=\mathbbm{1}_{\Gamma^n}(\w)\phi(x)$, where $\phi\in Y$ is arbitrary and $n$ is large enough, we have that $v
\in\ES$ and \eqref{eq:settozero} implies $\widetilde{u}=\widetilde{w}$ in $\Y^\star$.

For the surjectivity, first note that $\text{Im}\I\subset G^0$ since
\[\langle \I \widetilde{u},v\rangle=\E\LQ\langle \widetilde{u},v_\w\rangle\RQ= \langle \widetilde{u},\E \LQ v_\w\RQ \rangle=0,\quad \forall v\in G,\]
where one can exchange the duality pair between $Y$ and $Y^\star$ and the expectation operator due to the property of the Bochner integral \cite[E. 11]{cohn2013measure}. We now prove that $G^0\subset \text{Im}\I$. Take any $F\in G^0\subset \ES^\star$. Due to Riesz theorem, there exists a $\widetilde{f}\in \ES$ such that $F(v)=a(\widetilde{f},v)$, $\forall v\in \ES$. Restricting to $v\in G$,
\begin{equation}\label{eq:prop_isomorphism}
0=F(v)=a(\widetilde{f},v)=\E\LQ \langle \mathcal{A}_\w \widetilde{f}_\w,v_\w\rangle \RQ.
\end{equation}
Consider now the set $\Gamma^\infty :=\cup_{n\in \mathbb{N}} \Gamma^n$ which has full measure, i.e. $\PP(\Gamma^\infty)=1$ \footnote{If $\PP\left(\left(\Gamma^\infty\right)^c\right)>0$ then either $\frac{1}{\df_{\min}(\w)}$ or $\df_{\max}(\w)$ would not lie in $L^1(\Omega)$, contradicting Assumption \ref{ass:diff}.} and the restricted sigma algebra $\mathcal{M}:=\left\{E \cap \Gamma^\infty: E \in \mathcal{F}\right\}$ on $\Gamma^\infty$. Let us define $v=\psi(\w)\phi(x)$ where $\phi(x)\in Y$, and $\psi(\w)=\unob_E(\w)-\overline{\unob_E}$ for an arbitrary $E\in\mathcal{M}$, with $\overline{\unob_E}:=\E\LQ\unob_E(\w)\RQ$, so that $v\in G\subset \ES$. Replacing the expression of $v$ into \eqref{eq:prop_isomorphism}, we obtain
\[\E\LQ \unob_E\langle \mathcal{A}_\w \widetilde{f}_\w,\phi\rangle \RQ=\E\LQ \overline{\unob_E}\langle \mathcal{A}_\w \widetilde{f}_\w,\phi\rangle \RQ=\overline{\unob_E}\langle \E\LQ \mathcal{A}_\w \widetilde{f}_\w\RQ,\phi\rangle,\]
and denoting $f:= \E\LQ \mathcal{A}_\w \widetilde{f}_\w\RQ$ we have,
\begin{equation}\label{eq:almostsurely}
\E\LQ \unob_E \langle \mathcal{A}_\w\widetilde{f}_\w -f,\phi\rangle \RQ=\int_E \langle \mathcal{A}_\w \widetilde{f}_\w-f,\phi\rangle=0,\quad \forall E\in \mathcal{M},\forall \phi \in Y.
\end{equation}
Due to the arbitrariness of $E$ and the full measure of $\Gamma^\infty$, \eqref{eq:almostsurely} implies $\langle \mathcal{A}_\w \widetilde{f}_\w-f,\phi\rangle=0$ $\PP$-a.s., $\forall \phi \in Y$, hence $\mathcal{A}_\w \widetilde{f}_\w=f\in Y^\star$ $\PP$-a.s.

Thus, we conclude
\[\mathcal{F}(v)=a(\widetilde{f},v)=\langle \mathcal{A}f,v\rangle=\E\LQ \langle\mathcal{A}_\w \widetilde{f}_\w,v_\w\rangle \RQ=\E\LQ \langle f,v_\w\rangle\RQ= \mathcal{I}f(v),\quad \forall v\in \ES,\]
that is, for every $\mathcal{F}\in G^0$, there exists a $f\in Y^\star$ such that $\mathcal{F}=\I f.$
\end{proof}

Considering the state equation, we remark that  
\[a(y,v)= \langle \I \left(\Lambda_{U} u+f\right),v\rangle=0 \quad \forall v \in G,\]
that is, if $y$ is a solution to the state equation, then $y$ is a-orthogonal to G, i.e. $y\in G^\perp:=\left\{ y \in \ES: a(y,v) =0,\text{ }\forall v \in G\right\}$. In other words, whatever control function $u$ we choose, we cannot obtain a generic state $y \in \ES$, but the state solution is constrained to lie in the subspace $G^\perp$ of $\ES$. 
\begin{remark}
A similar constraint on the state variable has been observed in \cite{elvetun2016pde} for deterministic OCP with a control function acting only on a subdomain $\widetilde{D}\subset D$. The parallelism between a robust OCPUU and a deterministic OCP with local control lies in the observation that, in both OCPs, one cannot generate the whole dual of the state space using only elements of the control space. For robust OCPUU one has $\Ima\mathcal{I} \subsetneq \ES^\star$, see Lemma \ref{lemma:G0}, and similarly for a deterministic OCP with local control one cannot generate $(H^1(D))^\star$ using only elements of $(H^1(\widetilde{D}))^\star$
\cite{elvetun2016pde}. From the algebraic point of view, this leads to low-rank perturbed Schur complements, where the rank of the perturbation is equal to the size of the finite element discretization of the control space (see \eqref{eq:Schurcomplement_OCPUU} and \cite{heidel2019preconditioning}).
\end{remark}

To get robust-parameters estimates at the continuous level, it is essential to use these properties of the continuous formulation of the saddle point system. We thus consider the OCP \eqref{eq:OCP} with the state space equal to $G^\perp$. As the residual of the state equation $\mathcal{A}y-\I(\Lambda_{Y}u +f)\in G^0=(G^\perp)^\star$, the adjoint variable $p$ belongs to $G^\perp$ as well. Computing the directional derivatives of the restricted Lagrangian $\widehat{\mathcal{L}}(y,u,p): G^\perp \times Y \times G^\perp \rightarrow  \setR$ with  $\widehat{\mathcal{L}}(y,u,p):=\mathcal{L}(y,u,p)$, the optimality system becomes: find $(y,u,p)\in G^\perp\times \Y\times G^\perp$ such that
\begin{equation}\label{eq:optimality_system_Gperp}
\begin{array}{r l r r}
\E\LQ \langle \mathcal{A}_\w p_\w,v_\w\rangle \RQ + \E\LQ \langle \Lambda_{L^2}\right(y_\w+\gamma (y_\w-\E\LQ y_\w\RQ)\left),v_\w\rangle \RQ &= \E\LQ \langle \Lambda_{L^2} y_d,v_\w\rangle \RQ,\quad &\forall v\in G^\perp,&\\
\langle \beta \Lambda_Y u- \Lambda_Y \E\LQ  p_w\RQ,v\rangle &=0,\quad &\forall v\in Y,&\\
\E\LQ \langle \mathcal{A}_\w y_\w,v_\w\rangle\RQ -\E\LQ \Lambda_Y u,v_\w\rangle \RQ & = \E\LQ \langle f,v_\w\rangle\RQ,\quad &\forall v \in G^\perp.&
\end{array}
\end{equation}
Defining the space $\mathcal{X}:=G^\perp\times \Y$ and using the bilinear and linear forms defined above, the optimality conditions read:
 Find $(\xu,p)\in \mathcal{X} \times G^\perp$ such that
\begin{equation*}\label{eq:saddle_point_cont_Gperp}
\begin{aligned}
\C(\xu,\vu) + \B(\vu,p)&=\langle \mathcal{F},\vu\rangle,\quad \forall \vu \in \mathcal{X},\\
\B(\xu,q)&=\langle \mathcal{G},q\rangle,\quad \forall q\in G^\perp.
\end{aligned}
\end{equation*}

We now prove an important result stating that $G^\perp$ is homeomorphic to $\Y$. We introduce the operator $\mathcal{E}:Y\rightarrow G^\perp$ defined as $\mathcal{E}y=\mathcal{A}^{-1}\mathcal{I}\Lambda_Y y$, and its inverse $\mathcal{E}^{-1}: G^\perp \rightarrow Y$ such that $\mathcal{E}^{-1}v=\Lambda_Y^{-1}\IG^{-1}\mathcal{A}v$, where $\IG^{-1}$ is the inverse of the map $\mathcal{I}:\Y^\star \rightarrow G^0$.
\begin{theorem}\label{thm:equivalence_norms}
Let us consider $G^\perp$ equipped with the norm $\|\cdot\|_{\mathcal{A}}^2:=\langle {\mathcal{A}}\cdot,\cdot\rangle$ and $\Y$ equipped with the norm $\|\cdot\|_\Y$. The map $\mathcal{E}^{-1}$ is a homeomorphism between $G^\perp$ and $Y$, and further it holds
\[\frac{1}{\sqrt{\E\LQ \frac{1}{\df_{\min}(\w)}\RQ}}\|y\|_{{\mathcal{A}}}\leq \|\mathcal{E}^{-1}y\|_Y\leq  \sqrt{\E\LQ \df_{\max}(\w)\RQ} \|y\|_{\mathcal{A}}.\]
\end{theorem}
\begin{proof}
Observe that for every $y\in G^\perp$, $\mathcal{E}^{-1}y$ is well defined because ${\mathcal{A}}y\in G^0$ and thus it can be written as ${\mathcal{A}}y=\I\widetilde{L}$ for a unique $\widetilde{L}\in \Y^\star$ due to Lemma \ref{lemma:G0}. Finally the Riesz's representation isomorphism on $Y$ returns the Riesz representative $L$, so that $L=\mathcal{E}^{-1}y$. 
Moreover, on the one hand
\begin{equation*}
\begin{split}
\|y\|^2_{{\mathcal{A}}}&=\langle {\mathcal{A}} y,y\rangle=
\E\LQ \langle \IG^{-1}\mathcal{A} y,y_\w\rangle\RQ \leq \|\Lambda_Y^{-1}\IG^{-1} \mathcal{A} y\|_{Y}\E\LQ \|y_\w\|_Y \RQ\leq \|\mathcal{E}^{-1}y\|_{Y} \E\LQ \frac{1}{\sqrt{\df_{\min}(\w)}} \|y_\w\|_{\mathcal{A}_{\w}}\RQ \  \\
&\leq \|\mathcal{E}^{-1}y\|_{\Y} \left(\E\LQ \frac{1}{\df_{\min}(\w)}\RQ\right)^{\frac{1}{2}}\left(\E\LQ \|y_\w\|^2_{\mathcal{A}_\w}\RQ \right)^{\frac{1}{2}} = \|\mathcal{E}^{-1}y\|_{\Y} \sqrt{\E\LQ \frac{1}{\df_{\min}(\w)}\RQ} \|y\|_{\mathcal{A}},
\end{split}
\end{equation*}
which implies $\|y\|_{{\mathcal{A}}}\leq \sqrt{\E\LQ \frac{1}{\df_{\min}(\w)}\RQ}\|\mathcal{E}^{-1}y\|_Y$.
On the other hand,
\begin{equation*}
\begin{aligned}
\|\mathcal{E}^{-1}y\|_Y^2&=(\Lambda_Y^{-1}\IG^{-1}\mathcal{A}y,\mathcal{E}^{-1}y)_{\Y}=\langle \IG^{-1}\mathcal{A}y,\mathcal{E}^{-1}y\rangle=\E\LQ \langle \IG^{-1}\mathcal{A}y,\mathcal{E}^{-1}y\rangle \RQ=\E \LQ\langle \mathcal{A}_{\w} y_\w,\mathcal{E}^{-1}y\rangle\RQ \\
&\leq \E\LQ \|y_\w\|_{\mathcal{A}_{\w}}\|\mathcal{E}^{-1}y\|_{\mathcal{A}_{\w}}\RQ
\leq \|\mathcal{E}^{-1}y\|_Y \E\LQ\|y_\w\|_{\mathcal{A}_{\w}}\sqrt{\df_{\max}(\w)}\RQ
\leq \|\mathcal{E}^{-1}y\|_Y \|y\|_{{\mathcal{A}}} \sqrt{\E\LQ \df_{\max}(\w)\RQ}.
\end{aligned}
\end{equation*}
which implies $\|\mathcal{E}^{-1}y\|_Y\leq \sqrt{\E\LQ \df_{\max}(\w)\RQ} \|y\|_{\mathcal{A}}$.
\end{proof}
Fig. \ref{Fig:spaces} provides a useful graphical overview of the relations between the spaces $\Y$, $\Y^\star$, $ G^0$ and $G^\perp$. Due to Theorem \eqref{thm:equivalence_norms} we can parametrize the space $G^\perp\subset \ES$. Any $y \in G^\perp$ is in a one-to-one correspondence with an element of $\Y$ through the operator $\mathcal{E}$, that is $y=\mathcal{E}v$, for a unique $v \in \Y$. This property will be essential to prove a $\beta$-independent inf-sup condition.
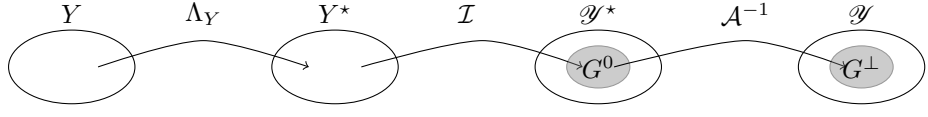
\begin{figure}
\centering
\begin{tikzpicture}[scale=0.7]
\draw (0,0) ellipse (1.2cm and 0.7cm);
\draw [black] (0,1) node {$\Y$};\quad
\draw [black,->] plot [smooth] coordinates {(0.5,0) (2.5,0.5) (4.5,0)};
\draw [black] (2.5,1) node {$\Lambda_{\Y}$};
\draw (5,0) ellipse (1.2cm and 0.7cm);
\draw[black] (5,1) node {$\Y^\star$};
\draw [black,->] plot [smooth] coordinates {(5.5,0) (7.8,0.5) (9.7,0)};
\draw [black] (7.5,1) node {$\mathcal{I}$};

\draw (10,0) ellipse (1.2cm and 0.7cm);
\draw [fill=gray,opacity=0.4] (10,0) ellipse (0.6cm and 0.4cm);
\draw[black] (10,1) node {$\ES^\star$};
\draw[black] (10,0) node {$G^0$};
\draw [black,->] plot [smooth] coordinates {(10.3,0) (12.8,0.5) (14.7,0)};
\draw [black] (12.8,1) node {${\mathcal{A}}^{-1}$};
\draw (15,0) ellipse (1.2cm and 0.7cm);
\draw[black] (15,1) node {$\ES$};
\draw [fill=gray,opacity=0.4] (15,0) ellipse (0.6cm and 0.4cm);
\draw[black] (15,0) node {$G^\perp$};
\end{tikzpicture}
\caption{Graphical representation of the maps between the different functional spaces.}\label{Fig:spaces}
\end{figure}

Let us now consider the following functional setting,
\[\mathcal{Y}:=\left(G^\perp,(\cdot,\cdot)_{\Ys}\right),\quad U=\left(\Y,(\cdot,\cdot)_U\right),\quad \mathcal{X}=\left(\Ys\times U,(\cdot,\cdot)_\mathcal{X}\right),   \quad \Ps:=\left(G^\perp,(\cdot,\cdot)_{\Ps}\right),\] where the scalar products define the weighted-norms
\begin{equation}\label{eq:weighted_norms}
\begin{aligned}
\|y\|^2_{\Ys}&:=(y,y)_{\Ys}=\E\LQ (y_\w,y_\w)_{L^2} +\gamma \left(y_\w-\E\LQ y_\w\RQ,y_\w-\E\LQ y_\w\RQ\right)_{L^2}\RQ +\beta \E\LQ \langle \mathcal{A}_{\w} y_\w,y_\w\rangle \RQ=\\
&= (y,y)_{L^2,\gamma} +\beta (y,y)_{{\mathcal{A}}},\\
\|u\|^2_{U}&:=(u,u)_U=\beta (u,u)_{\Y},\\
\|(y,u)\|^2_\mathcal{X}&:= \left((y,u),(y,u)\right)_\mathcal{X}=(y,y)_{\Ys} + (u,u)_U,\\
\|p\|^2_{\Ps}&:=\frac{1}{\beta}\E\LQ \langle \mathcal{A}_\w p_\w,p_\w\rangle\RQ=\frac{1}{\beta}(p,p)_{\mathcal{A}}.\\
\end{aligned}
\end{equation}
For the state variable $y$ we introduce the scalar product $(\cdot,\cdot)_{L^2,\gamma}$ which consists of two parts: the first one is the simple $L^2(\Omega,L^2(D))$ norm. The second part proportional to $\gamma$ consists in expectation of the $L^2(D)$ norm of the difference between $y_\w$ from its mean value.
\begin{remark}\label{remark:OP2}
We remark once more that the energy norm and the $L^2(\Omega,Y)$ norm are not equivalent, unless $\df(x,\w)$ is uniformly bounded. In the latter case, one could show the well-posedness of the saddle point system working exclusively with the energy norm (obtaining, though, $\beta$-dependent constants). In constrast, for a not uniformly bounded $\df(x,\w)$, one would fail to bound the bilinear form $\mathcal{C}(\cdot,\cdot)$ with only the energy norm, and thus one would have to rely on the framework of \cite{gittelson,schwab}, and introduce an energy norm with respect to an auxiliary measure to be able to bound the $L^2(\Omega,Y)$ norm with the new modified energy norm. In this manuscript, we are interested to study $\beta$-robust preconditioners, which are derived taking a weighted combination of both the $L^2(\Omega,Y)$ norm and the energy norm, see \cite{doi:10.1137/060660977,zulehner2011nonstandard,mardal2011preconditioning} for deterministic OCP, even for bounded random fields, and thus we can avoid the framework of \cite{gittelson,schwab}, since we do not need any relation between the two norms, as the next Theorem shows. 
\end{remark}
\begin{theorem}[Well-posedness of the saddle point problem]\label{thm:Operator_pre}\text{}\\
\begin{enumerate}
\item The bilinear form $\C$ is bounded: $\C(\xu,\vu)\leq \|\xu\|_{\mathcal{X}}\|\vu\|_{\mathcal{X}},\quad \forall \xu,\vu \in \mathcal{X}$.
\item The bilinear form $\C$ is coercive on the Kernel of $\B$: $\C(\xu,\xu)\geq C_1 \|\xu\|_{\mathcal{X}}^2\quad \forall \xu \in \text{Ker}\B,$
where $C_1:=\min\left\{\frac{1}{2},\frac{1}{2\E\LQ \frac{1}{\df_{\min}(\w)}\RQ}\right\}$.
\item The bilinear form $\B$ is bounded: $\sup_{0\neq \xu\in \mathcal{X}} \frac{\B(\xu,q)}{\|\xu\|_\mathcal{X}}\leq C_2\|q\|_{\Ps} \quad \forall q \in \Ps,$
where $C_2=\max\left\{1,\E\LQ \frac{1}{\df_{\min}(\w)}\RQ\right\}$.
\item The bilinear form $\B$ satisfies the inf-sup condition: $\sup_{0\neq \xu\in \mathcal{X}} \frac{\B(\xu,\qu)}{\|\xu\|_\mathcal{X}}\geq C_3\|\qu\|_{\Ps},\quad \forall \qu \in \Ps$, where $C_3=\frac{1}{\E\LQ \df_{\max}(\w)\RQ}$.
\end{enumerate}
\end{theorem}
\begin{proof}
Let us first show the continuity of $\C$. Being $\C(\cdot,\cdot)$ symmetric, it is sufficient to show that $\C(\xu,\xu)\leq \|\xu\|^2_{X}$ which is trivially true since, for $\xu=(y,u)$, 
\[\C(\xu,\xu)=(y,y)_{L^2,\gamma}+\beta (u,u)_{\Y}\leq (y,y)_{L^2,\gamma}+\beta (y,y)_{\mathcal{A}}+\beta (u,u)_{\Y}=(\xu,\xu)_\mathcal{X}=\|\xu\|^2_\mathcal{X}.\]
Next, we focus on the coercivity of $\C$ on Ker$\B$. If $\xu=(y,u)\in \text{Ker}\B$ then $\langle {\mathcal{A}} y,q\rangle=\langle \I \Lambda_{\Y} u,q\rangle =\E\LQ \langle \Lambda_Y u,q_\w\rangle \RQ$ which, choosing $q=y$, implies
\[(y,y)_{\mathcal{A}}\leq \|u\|_{\Y} \E\LQ \frac{1}{\sqrt{\df_{\min}(\w)}} \|y_\w\|_{\mathcal{A}_{\w}}\RQ \leq \|u\|_Y \sqrt{\E\LQ \frac{1}{\df_{\min}(\w)}\RQ} \|y\|_{\mathcal{A}}.\] 
Then,
\begin{equation*}
\begin{aligned}
\C((y,u),(y,u))&=(y,y)_{L^2,\gamma}+\beta(u,u)_{\Y}\geq (y,y)_{L^2,\gamma}+\frac{\beta}{2}(u,u)_{\Y}+\frac{\beta}{2\E\LQ \frac{1}{\df_{\min}(\w)}\RQ} (y,y)_{{\mathcal{A}}}\\
&\geq \min\left\{\frac{1}{2},\frac{1}{2\E\LQ \frac{1}{\df_{\min}(\w)}\RQ}\right\} ((y,u),(y,u))_\mathcal{X}.
\end{aligned}
\end{equation*}

To show the continuity of $\B$, we consider
\begin{equation*}
\sup_{(y,u)\in \mathcal{X}} \frac{\B^2((y,u),q)}{\|(y,u)\|^2_{\mathcal{X}}}=\sup_{(y,u)\in \mathcal{X}} \frac{\left((y,q)_{{\mathcal{A}}} -\langle \I\Lambda_{\Y} u,q\rangle \right)^2}{\|(y,u)\|^2_{\mathcal{X}}}=\sup_{y\in \Ys} \frac{(y,q)^2_{\mathcal{A}}}{\|y\|^2_{\Ys}} + \sup_{u\in U} \frac{\left(\E\LQ ( u,q_\w)_{\Y}\RQ\right)^2}{\|u\|^2_{U}},
\end{equation*}
where the last equality follows from \cite[Lemma 2.1]{zulehner2011nonstandard}. The second term simplifies to
\begin{equation}\label{eq:continuity_B1}
\sup_{u\in U} \frac{\left(\E\LQ (u,q_\w)\RQ\right)^2}{\|u\|^2_{U}}=\frac{1}{\beta}\sup_{u\in \Y} \frac{\left( (u,\E\LQ q_\w \RQ )_{\Y}\right)^2}{\|u\|^2_{\Y}}=\frac{1}{\beta}\left\|\E\LQ q_\w\RQ \right\|^2_{\Y}. 
\end{equation}
Considering the first term,
\begin{equation}\label{eq:continuity_B2}
\begin{aligned}
\sup_{y\in \Ys} \frac{(y,q)_{\mathcal{A}}^2}{\|y\|^2_{\Ys}}=\sup_{y\in \Ys} \frac{(y,q)_{\mathcal{A}}^2}{(y,y)_{L^2,\gamma}+\beta (y,y)_{\mathcal{A}}}\leq \frac{1}{\beta}\sup_{y\in \Ys} \frac{(y,q)_{\mathcal{A}}^2}{(y,y)_{\mathcal{A}}}=\frac{1}{\beta} (q,q)_{\mathcal{A}}.
\end{aligned}
\end{equation}
Putting together \eqref{eq:continuity_B1} and \eqref{eq:continuity_B2}, using the Cauchy-Schwarz inequality and equivalence between $\|\cdot\|_Y$ and $\|\cdot\|_{A_{\w}}$,
\begin{equation*}
\begin{aligned}
&\sup_{(y,u)\in \mathcal{X}} \frac{\B^2((y,u),q)}{\|(y,u)\|^2_{\mathcal{X}}}\leq \frac{1}{\beta}\left\|\E\LQ q_\w\RQ \right\|^2_{\Y}+\frac{1}{\beta}(q,q)_{\mathcal{A}}\leq \frac{1}{\beta}\left( \E\LQ \frac{1}{\df_{\min}(\w)}\RQ\|q\|^2_{\mathcal{A}}+ \|q\|^2_{\mathcal{A}}\right)\\
&\leq \max\left\{1,\E\LQ \frac{1}{\df_{\min}(\w)}\RQ\right\}  \|q\|^2_{\Ps}.
\end{aligned}
\end{equation*}
Finally, we deal with the inf-sup condition. Using again \cite[Lemma 2.1]{zulehner2011nonstandard} and choosing $(y,u)=(0,u)\in\mathcal{X}$, we simply obtain the estimate
\begin{equation*}\label{eq:inf_sup}
\begin{aligned}
\sup_{0\neq (y,u)\in \mathcal{X}} \frac{\B^2(x,q)}{\|x\|^2_\mathcal{X}}=\sup_{y\in \Ys} \frac{(y,q)^2_{\mathcal{A}}}{\|y\|^2_{\Ys}} + \sup_{u\in U} \frac{\left(\langle \I \Lambda_{\Y}u,q\rangle\right)^2}{\|u\|^2_{U}}\geq \frac{1}{\beta}\sup_{u\in \Y} \frac{\left(\langle \I \Lambda_{\Y}u,q\rangle\right)^2}{\|u\|^2_{\Y}}.
\end{aligned}
\end{equation*}
As $q\in G^\perp$, we set $u=\mathcal{E}^{-1}q=\Lambda^{-1}_{\Y}\IG^{-1}\mathcal{A}q$ and Theorem \ref{thm:equivalence_norms} guarantees that $\|u\|^2_{Y}\leq \E\LQ \df_{\max}(\w)\RQ \|q\|^2_{\mathcal{A}}$, so that
\[\sup_{0\neq (y,u)\in \mathcal{X}} \frac{\B^2(x,q)}{\|x\|^2_\mathcal{X}}\geq \frac{1}{\beta}\frac{\|q\|^4_{\mathcal{A}}}{\|u\|^2_{Y}}\geq \frac{1}{\E\LQ \df_{\max}(\w)\RQ} \left(\frac{1}{\beta}\|q\|^2_{\mathcal{A}}\right)= \frac{1}{\E\LQ \df_{\max}(\w)\RQ} \|q\|^2_{\Ps}. \]
\end{proof}

\subsection{Mean and Chebyshev semi-iterative approximations}
The optimality system \eqref{eq:optimality_system_Gperp} involves the non standard trial and test space $G^\perp$. To implement it efficiently, we can rely on the isomorphism $\mathcal{E}$ between $\Y$ and $G^\perp$, so that \eqref{eq:optimality_system_Gperp} is equivalent to: find $(y,u,p)\in \Y\times Y\times Y$ such that $\forall (v,w,r)\in Y\times Y\times Y$
\begin{equation}\label{eq:optimality_system_Gperp_E}
\begin{array}{r l r r}
\E\LQ \langle \mathcal{A}_\w (\Ex p)_\w,(\Ex v)_\w\rangle + \langle \Lambda_{L^2}\right((1+\gamma)(\Ex y)_\w-\gamma\E\LQ (\Ex y)_\w\RQ)\left),(\Ex v)_\w\rangle \RQ &= \E\LQ \langle \Lambda_{L^2} y_d,(\Ex v)_\w\rangle \RQ,\\
\langle \beta \Lambda_Y u- \Lambda_Y \E\LQ  (\Ex p)_\w\RQ,w\rangle &=0,\\
\E\LQ \langle \mathcal{A}_\w (\Ex y)_\w,(\Ex r)_\w\rangle\RQ -\E\LQ\langle \Lambda_Y u,(\Ex r)_\w\rangle \RQ & = \E\LQ \langle f,(\Ex r)_\w \rangle\RQ.
\end{array}
\end{equation}
A discretization of \eqref{eq:optimality_system_Gperp_E} leads to the discrete system $\SOP \mathbf{x}=E^\top \mathscr{S} E\mathbf{x}=\mathbf{f}$, where $\mathbf{f}=E^\top \mathbf{b}$, $\mathscr{S}$ and $\mathbf{b}$ are given by \eqref{eq:KKT_matrix_sym_disc} \footnote{Replacing the $L^2(D)$ Riesz operator $M_s$, with the $\Y$ Riesz operator $K$.}, while
\begin{equation*}
E:=\begin{pmatrix}
A^{-1}Z\unob K\\
& I_s\\
& & A^{-1}Z\unob K
\end{pmatrix},\text{ } \mathbf{x}=\begin{pmatrix}
\mathbf{y}\\ \mathbf{u}\\ \mathbf{p}
\end{pmatrix}\in\setR^{3 N_h},
\end{equation*}
The matrix $E$ is the discretization of the isomorphism $\Ex$.
The solution of $\SOP \mathbf{x}=\mathbf{f}$ using a Krylov method requires, on the one hand, to compute the matrix vector product with $E^\top \mathscr{S} E$, thus to invert the matrix $A^{-1}$ twice, and this must be computed exactly, or up to a very low tolerance. On the other hand, $E^\top \mathscr{S}E$ is a matrix of dimension $3N_h \lll (2 N+1)N_h$, which is the size of $\mathscr{S}$. Thus, a Krylov method is less prone to saturation of memory and instability due to orthogonalization.
According to the software, architecture and problem at hand, the pros could be larger than the cons, or viceversa.
As a preconditioner we use $\POP=E^\top B E$, where $B$ is the matrix representing the weighted norms defined in \eqref{eq:weighted_norms}, that is 
\begin{equation*}\label{eq:Preconditioner_OP}
B:=\begin{pmatrix}
M_\gamma +\beta A\\
& \beta K\\
& & \frac{1}{\beta} A
\end{pmatrix} \in \setR^{(2N+1)\cdot N_h\times (2N+1)\cdot N_h}.
\end{equation*}
A direct calculation leads to
\begin{equation}\label{eq:Preconditioner_OP_restricted}
\ \POP=\begin{pmatrix}
B_1\\
& B_2 \\
&& B_3
\end{pmatrix}:=\begin{pmatrix}
K\unob^\top \left( A^{-1} Z M_\gamma Z A^{-1} +\beta  Z A^{-1} Z\right)\unob K\\
& \beta K\\
& & \frac{1}{\beta} K \unob^\top Z A^{-1} Z\unob K.
\end{pmatrix}.
\end{equation}
Similarly to Section \ref{sec:matchschur}, we can approximate the inverse of $\POP$ using a mean approximation of the blocks $B_1$ and $B_3$. Replacing formally the matrix $A^{-1}$ with a matrix of equal size with $A^{-1}_0$ on the diagonal, we obtain
the mean preconditioner $\POPM:= E^\top B_M E$,
\begin{equation}\label{eq:Preconditioner_OP_restricted_mean}
\ \POPM^{-1}=\begin{pmatrix}
B^{-1}_{1,M} \\
& B^{-1}_{2,M}\\
& & B^{-1}_{3,M}
\end{pmatrix}:=\begin{pmatrix} K^{-1}A_0\left(M_s+\beta A_0\right)^{-1}A_0 K^{-1}\\
& \frac{1}{\beta} K^{-1}\\
& & \beta K^{-1}A_0K^{-1}
\end{pmatrix}.
\end{equation}
If the variance is large, we use $B^{-1}_{1,M}$ and $B^{-1}_{3,M}$ as preconditioners inside a Chebyshev Semi-Iterative method to invert $B_1$ and $B_3$. The two Chebyshev Semi-Iterative method can be executed separately and in parallel. To choose the parameters $\alpha,\underline{\lambda}$ and $\overline{\lambda}$, 
we rely on the following Lemma, obtained using the same argument of Lemma \ref{lemma:spectrumPinvL}.
\begin{lemma}\label{lemma:spectraB_1B_3}
The spectra of $B^{-1}_{1,M}B_1$ and of $B^{-1}_{3,M}B_3$ are real and bounded from below by $1$.
\end{lemma}

\begin{table}\caption{\small{Minimum and maximum in modulo eigenvalues $\lambda_{\min}-\lambda_{\max}$ of the preconditioned system}}\label{Tab:Eig_OP}
\small
\centering
\begin{tabular}{| c | c | c | c | c | c|}
\hline
$\beta$ &  &$ 10^{-2}$ & $10^{-4}$ &  $10^{-6}$ & $10^{-8}$\\ \hline
$\POP^{-1}\SOP$ & $\df_B(x,\w)$ & 0.54 - 1.55  &  0.36 - 1.50 & 0.36 - 1.37& 0.36 - 1.36\\
$\POPM^{-1}\SOP$ & $\df_B(x,\w)$ & 0.55 - 1.61  &  0.36 - 1.57 & 0.36 - 1.40 & 0.36 - 1.39\\
$\POPC^{-1}\SOP$ & $\df_B(x,\w)$ & 0.54 - 1.55  &  0.56 - 1.50 & 0.35 - 1.37& 0.35 - 1.36  \\
$\POP^{-1}\SOP$ &$\df_L(x,\w)$ & 0.41 - 1.90 &  0.68 - 1.80 &  0.68 - 1.73 & 0.68 - 1.73  \\
$\POPM^{-1}\SOP$ & $\df_L(x,\w)$ & 0.53 - 2.74 &  0.87 - 3.91 &  0.87 - 4.08 & 0.87 - 4.10  \\
$\POPC^{-1}\SOP$ &$\df_L(x,\w)$ & 0.41 - 1.91 &  0.68 - 1.79 &  0.68 - 1.73 & 0.68 - 1.73  \\
\hline
\end{tabular}
\\
\centering
$N_h=225$, $m=3$, $\sigma^2=0.5$, $\gamma=0.1$, $L^2=0.5$. $\Nit=2$ for both $\df_B$ and $\df_L$.\\\vspace*{0.3cm}
\begin{tabular}{| c | c | c | c | c | c|}
\hline
$\sigma^2$ &  & $ 0.1$ & $0.5$ &  $1$ & $1.5$\\ \hline
$\POP^{-1}\SOP$ &$\df_B(x,\w)$ & 0.36 - 1.36 &  0.35 - 1.36 &  0.34 - 1.36 & 0.32 - 1.36\\
$\POPM^{-1}\SOP$ & $\df_B(x,\w)$ & 0.36 - 1.36 &  0.36 - 1.39 &  0.35 - 2.56 & 0.35 - 8,48  \\
$\POPC^{-1}\SOP$ &$\df_B(x,\w)$ & 0.36 - 1.36 &  0.35 - 1.36 &  0.34 - 1.36 & 0.32 - 1.36 \\
$\POP^{-1}\SOP$ & $\df_L(x,\w)$ &  0.63 - 1.63  &  0.68 - 1.73 & 0.75 - 1.85&  0.83 - 2.00  \\
$\POPM^{-1}\SOP$ & $\df_L(x,\w)$ & 0.63 - 1.68  &  0.68 - 4.25 & 0.75 - 16.65& 0.80 - 59.98\\
$\POPC^{-1}\SOP$ & $\df_L(x,\w)$ & 0.63 - 1.63  &  0.68 - 1.73 & 0.75 - 1.85 & 0.80 - 2.00  \\
\hline
\end{tabular}
\\
\centering
$N_h=225$, $m=3$, $\beta=10^{-8}$, $\gamma=0.1$, $L^2=0.5$. $\Nit$ is equal to $2,2,4,4$ for $\df_B(x,\w)$ and equal to $2,4,6,8$ for $\df_L(x,\w)$.\\\vspace*{0.3cm}
\begin{tabular}{| c | c | c | c | c | c|}
\hline
  m  & & $ 2$ & $3$ &  $4$ & $5$\\ \hline
$\POP^{-1}\SOP$ & $\df_B(x,\w)$ & 0.36 - 1.37 &  0.36 - 1.37 & 0.36 - 1.37 &0.36 - 1.37\\
$\POPM^{-1}\SOP$ & $\df_B(x,\w)$ & 0.36 - 1.41 &  0.36 - 1.41 & 0.36 - 1.41& 0.36 - 1.41\\
$\POPC^{-1}\SOP$ & $\df_B(x,\w)$  & 0.36 - 1.37 &  0.36 - 1.37 & 0.36 - 1.37& 0.36 - 1.37  \\
$\POP^{-1}\SOP$ & $\df_L(x,\w)$  & 0.67 - 1.72 &  0.68 - 1.73 & 0.68 - 1.73 & 0.68 - 1.73\\
$\POPM^{-1}\SOP$ & $\df_L(x,\w)$  & 0.86 - 3.67 &  0.87 - 4.24 & 0.87 - 4.31 & 0.87 - 4.31\\
$\POPC^{-1}\SOP$ & $\df_L(x,\w)$  & 0.67 - 1.72 &  0.68 - 1.73 & 0.68 - 1.73& 0.68 - 1.73  \\
\hline
\end{tabular}
\\\centering
$N_h=225$, $\sigma^2=0.5$, $\beta=10^{-6}$, $\gamma=0.1$, $L^2=0.5$, $\Nit=2$ for both $\df_B(x,\w)$ and $\df_L(x,\w)$.
\end{table}
Tables \ref{Tab:Eig_OP} report the minimum and maximum in modulo eigenvalues of the preconditioned system using either $\POP,\POPM$ or $\POPC$. All preconditioners exhibits a $\beta$-robust spectrum, and in particular the mean preconditioner $\POPM^{-1}$ performs quite better than the algebraic one $\SLRM$ (see Table \ref{Tab:Eig_matching}).
The dependence of $\POP$ on $\sigma^2$ is weak, and similar to that of $\widetilde{\mathcal{P}}^{-1}$, analysed in Section \ref{sec:firstschur}, as Theorem \ref{thm:Operator_pre} involves the first moments of $1/\df_{\min}(\w)$ and $\df_{\max}(\w)$.
Finally Table \ref{Tab:Eig_OP} shows that all preconditioners are robust with respect to the number of collocation points.

\section{Numerical experiments}\label{sec:num}
The aim of this section is to further validate the theoretical results presented in Section \ref{sec:AlgebraicPreconditioners} and \ref{sec:OP}, and to compare the preconditioners analysed on a model problem.
We consider the domain $D=(0,1)^2$ discretized with a regular mesh of size $h$, and a finite element approximation using $\mathbb{P}_1$ finite elements.
For each preconditioner $\widetilde{P},\PLRM,\PLRC,\POPM$ and $\POPC$, we report the number of iterations and computational times in seconds to solve the saddle point system using preconditioned MINRES . 
Although it is tempting to compare the computational times and number of iterations among all preconditioners, we stress that $\POPM$ and $\POPC$ compute a different optimal control with respect to $\widetilde{P}$, $\PLRM$ and $\PLRC$, as the control belongs to $\Y$ and acts on the state equation through the Riesz map of $\Y$ (see \cite{elvetun2016pde} for an instance of application arising in electrocardiography).
In all experiments, the matrix $A$ is inverted approximately using the Fortran Algebraic MultiGrid (AMG) library \text{HSL\_MI20} \cite{https://doi.org/10.1002/nme.2758}, which is called using the Matlab interface. We specifically used two V-cycles with one iteration of the damped Jacobi smoother with parameter $\theta=\frac{8}{9}$ and $5$ levels. All other parameters are left to default values. The inverse of $C$ is computed approximately, inverting the mass matrix $M$ with 25 iterations of the Chebyshev semi-iterative method using as preconditioner the diagonal of $M$ itself. The damping parameters $\alpha$ as well as $\underline{\lambda}$ and $\overline{\lambda}$ are estimated once and for all using the mass matrix $M_s$.
The application of $A^{-1}$ and of $M^{-1}$ onto a vector is performed in parallel, using the Matlab Parallel Computing Toolbox.
Further, we compute once for all the LU decomposition of $A_0+\frac{1}{\sqrt{\beta}}M_s$, $\beta A_0+ M_s$ and $K$, which is feasible as their size corresponds to a single PDE discretization. Clearly, one could further approximate them using AMG, if $\beta$ is not too small, or using other iterative methods.
Finally, when using $\POPM$ and $\POPC$, we compute the exact action of $A^{-1}$ using eight iterations of the conjugate gradient method preconditioned by AMG, which are enough to have a (not preconditioned) residual of approximately $10^{-11}$. MINRES is stopped when the relative (not preconditioned) residual is smaller than $\tol=10^{-6}$. The simulations have been performed on a workstation equipped with an Intel® Core™ i9-10900X and 32 GB of RAM.

\subsection{Bounded random field with Stochastic Collocation}
We consider the bounded random field defined in \eqref{eq:bounded_random_field}, and we study the robustness of the preconditioners with respect to $\beta$ and $\sigma^2$. We use a full tensorized Gauss-Legendre quadrature formula with 5 points for each random variable $\xi_j(\w)$, $j=1,\dots,4$, thus $N=5^4=625$. The mesh size is $h=2^{-5}$ and $N_h=961$. The global system has approximately 1.2 millions degrees of freedom. The target state is $y_d=\sin(\pi x)\sin(\pi y)$.

First we consider $u\in L^2(D)$. Table \ref{Tab:iter_bounded_L^2} confirms that $\widetilde{P}$ is extremely efficient when $\beta$ is sufficiently large, but its performance deteriorates when $\beta\rightarrow 0$ as Theorem \ref{thm:Ptbounds} predicts. For a moderate value of $\sigma^2$, $\PLRM$ performs well unless for extremely small values of $\beta$ (e.g $\beta \approx 10^{-8}$), as remarked in Table \ref{Tab:Eig_matching}. $\PLRC$ recovers robustness at the price of additional Chebyshev semi-iterations, and leads to constant numbers of iterations and computational times as $\beta\rightarrow 0$.

Next we compare the preconditioners as $\sigma^2$ increases. We set $\beta=10^{-2}$ for $\widetilde{P}$, while $\beta=10^{-6}$ for $\PLRM$ and $\PLRC$. 
$\widetilde{P}$ exhibits a very weak dependence on $\sigma^2$. This is reflected both by the estimates of Theorem \ref{thm:Ptbounds} and by Table \ref{Tab:Eig_wathen}. Recall that $\df_B(x,\w)\geq 1$ for a.e. $\w$, so that $\EQ \LQ \frac{1}{\df_{\min}^2(\w)}\RQ$ is bounded as $\sigma^2$ grows. $\PLRM$ quickly becomes inefficient when $\sigma^2$ grows as expected, since the mean matrix $A_0^{-1}$ is a crude approximation of $\unob^\top Z A^{-1} Z\unob$ that does not take into account the variability of the stiffness matrices (see e.g. \cite{powell2009block}). The addition of the Chebyshev semi-iteration helps to reduce the computational time and number of iterations, but does not remove completely the dependence over $\sigma^2$, which however is comparable to that of $\widetilde{P}$.

\begin{table}[h]\caption{Number of iterations and computational time in seconds to reach a relative residual smaller than $10^{-6}$.}\label{Tab:iter_bounded_L^2}
\small
\centering
\begin{tabular}{| c | c | c | c | c | c|}
\hline
$\beta$ &$ 10^{-2}$ & $10^{-4}$ &  $10^{-6}$ & $10^{-8}$\\ \hline
$\widetilde{P}^{-1}\mathscr{S}$ & 29 (20.34)  &  37 (25.01) &  109 (72.28) & NA \\
$\PLRM^{-1}\mathscr{S}$ & 31 (42.69)  &  33 (44.59) &  37 (49.67) & 169 (221.39)  \\
$\PLRC^{-1}\mathscr{S}$ & 31 (89.45)  &  33 (94.51) & 31 (88.19) & 31 (87.83) \\
\hline
\end{tabular}\\
\centering
$\sigma^2=0.5$, $\gamma=10^{-1}$, $\Nit=2$. NA means MINRES did not converge in less than 200 iterations.\\\vspace{0.3cm}
\centering
\begin{tabular}{| c | c | c | c | c | c|}
\hline
 & \diaghead(-3,2){\hskip0.7cm\hsize0.7cm}%
 {$\beta$}{$\sigma^2$} & 0.1 & 0.5 &  1 & 1.5 \\ \hline
$\widetilde{P}^{-1}\mathscr{S}$ & $10^{-2}$ & 29 (20.64) &  29 (19.40) &  31 (20.80) & 33 (21.86) \\
$\PLRM^{-1}\mathscr{S}$ &  $10^{-6}$ & 27 (36.83) & 37 (48.94) & 87 (112.58) & 198 (254.53)  \\
$\PLRC^{-1}\mathscr{S}$ & $ 10^{-6}$ & 27 (76.43) & 31 (86.48) & 35 (132.94) & 39 (148.63)  \\
\hline
\end{tabular}\\
\centering
$\gamma=10^{-1}$. $\Nit=2$ for $\sigma^2\in \left\{0.1,0.5\right\}$ and $\Nit=4$ for $\sigma^2\in \left\{1,1.5\right\}$.
\end{table}

Next, we consider a control $u\in Y$ and the operator preconditioning approach. Table \ref{Tab:Iter_bounded_OP} shows that both $\POPM$ and $\POPC$ are very robust with respect to $\beta$. Interestingly, $\POPM$ performs well also for $\beta \approx 10^{-8}$ in contrast with $\PLRM$. Further, $\POPM$ is inefficient for larger values of $\sigma^2$. $\POPC$ performs better, but still exhibits a $\sigma^2$ dependence as expected, since the estimates of Theorem \ref{thm:Operator_pre} involve the first moments of $\frac{1}{\df_{\min}(\w)}$ and $\df_{\max}(\w)$.

\begin{table}[t]\caption{Number of iterations and computational time in seconds to reach a relative residual smaller than $10^{-6}$.}\label{Tab:Iter_bounded_OP}
\small
\centering
\begin{tabular}{| c | c | c | c | c | c|}
\hline
$\beta$ &$ 10^{-2}$ & $10^{-4}$ &  $10^{-6}$ & $10^{-8}$\\ \hline
$\POPM^{-1}\mathscr{S}$ & 27 (50.78)  &  43 (76.78) &  19 (35.24) & 19 (35.41) \\
$\POPC^{-1}\mathscr{S}$ & 28 (140.38)  &  46 (226.81) & 28 (139.79) & 28 (134.81) \\
\hline
\end{tabular}\\
\centering
$\sigma^2=0.5$, $\gamma=10^{-1}$, $\Nit=2$.\\\vspace{0.3cm}
\centering
\begin{tabular}{| c | c | c | c | c | c|}
\hline
$\sigma^2$  & 0.1 & 0.5 & 1 & 1.5\\ \hline
$\POPM^{-1}\mathscr{S}$ & 12 (23.54)  &  19 (35.17) &  37 (65.56) & 92 (159.50) \\
$\POPC^{-1}\mathscr{S}$ & 27 (133.96)  &  31 (138.60) & 35 (216.86) & 39 (236.68) \\
\hline
\end{tabular}\\
\centering
$\beta=10^{-6}$ and $\gamma=10^{-1}$. $\Nit=2$ for $\sigma^2\in \left\{0.1,0.5\right\}$ and $\Nit=4$ for $\sigma^2\in \left\{1,1.5\right\}$.\\\vspace*{0.3cm}
\end{table}

\subsection{Log-normal field with Monte Carlo sampling}\label{sec:num_lognormal}
In this subsection, we consider the log-normal field $\df_L(x,\w)$ defined in \eqref{eq:log-normal_field} with covariance function ${Cov}_g(x,y)=\sigma^2\exp\left(-\frac{\|x-y\|_2^2}{L^2}\right)$.
We consider a relatively small correlation length, setting $L^2=0.025$ and $\sigma^2=0.5$. Fig. \ref{Fig:log-normal_fields} shows two random realizations of $\df_L(x,\w)$.
\begin{figure}
\includegraphics[scale=0.3]{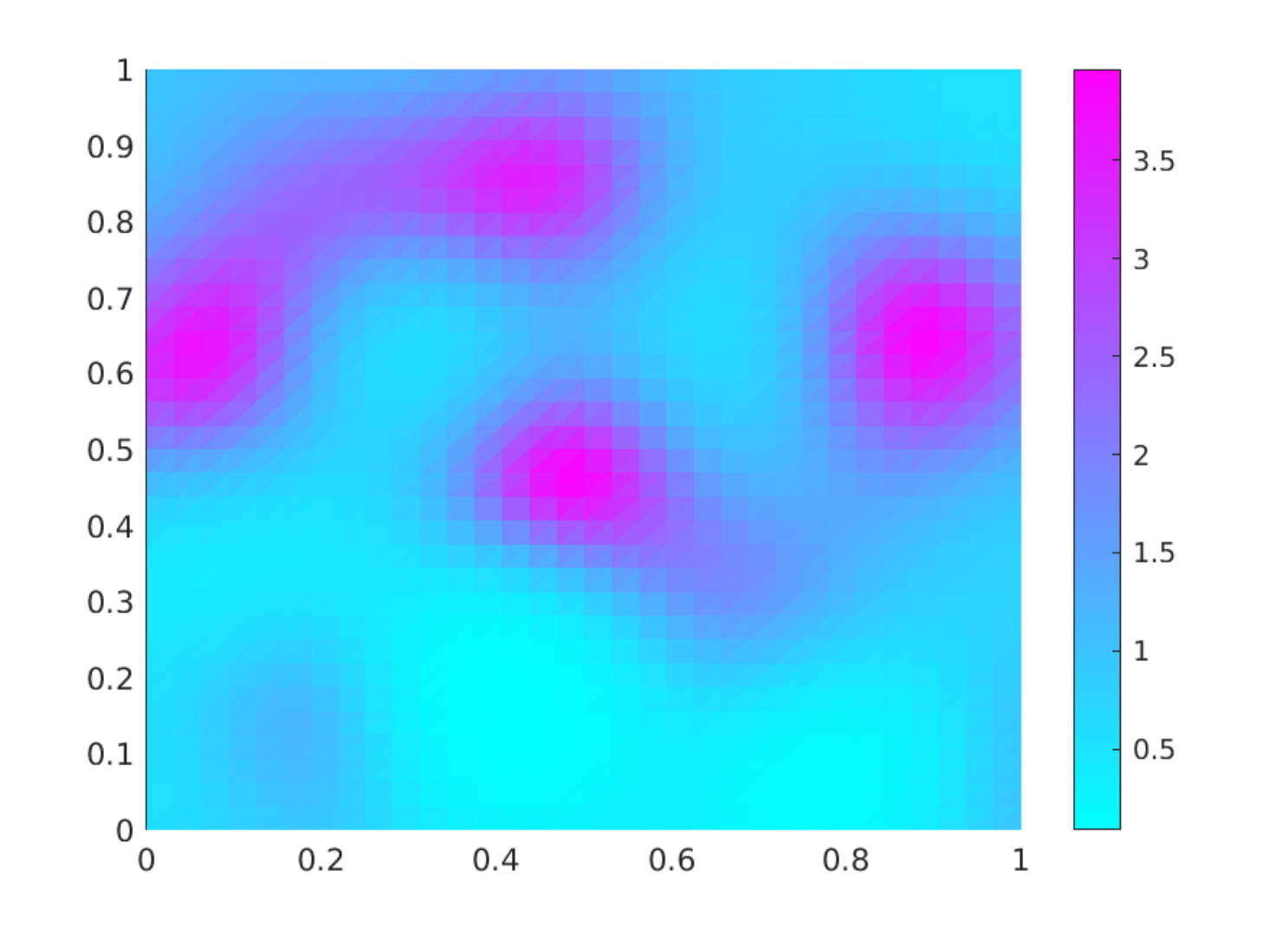}
\includegraphics[scale=0.3]{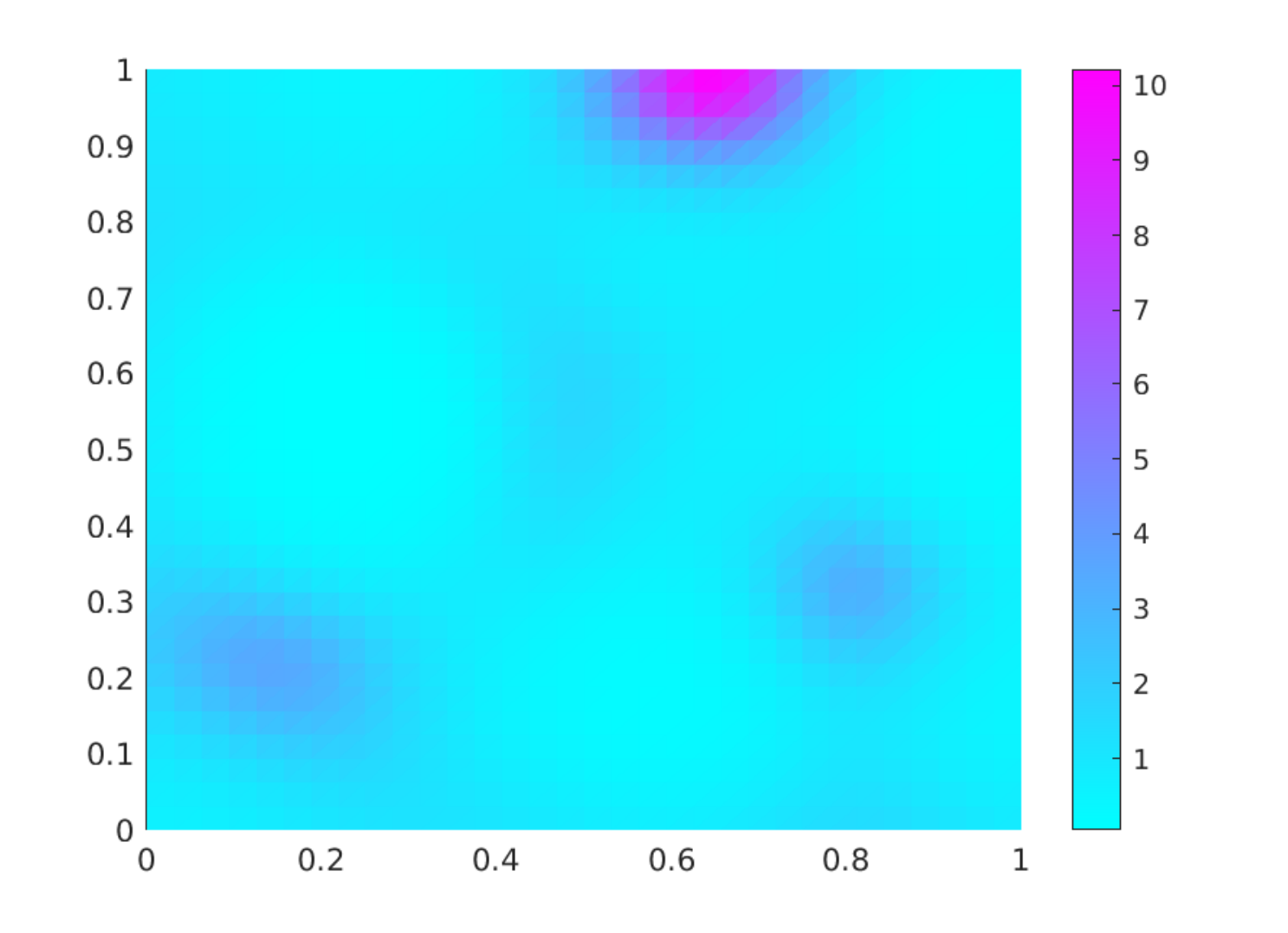}
\includegraphics[scale=0.3]{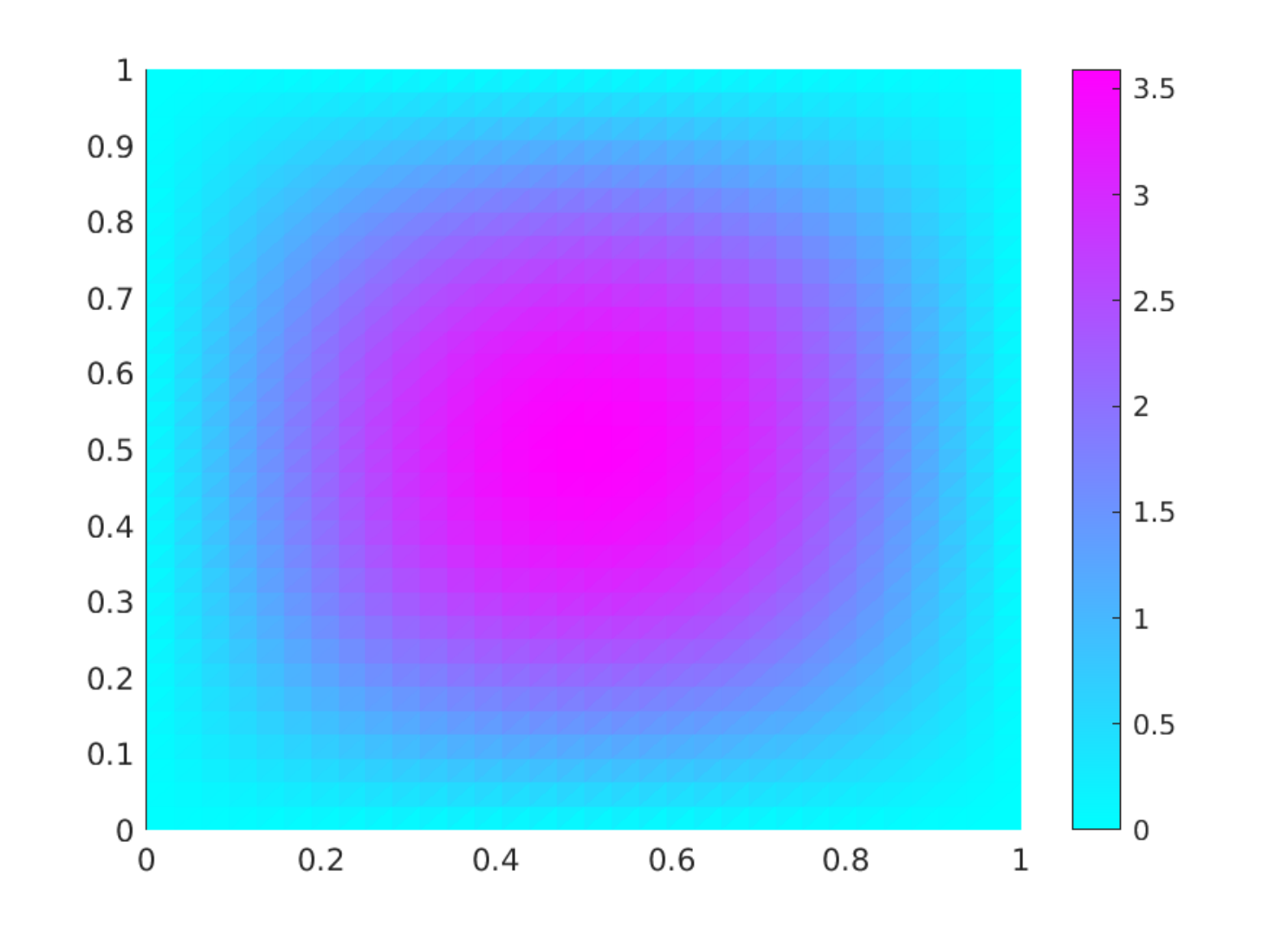}
\caption{The left and center panels show two random realizations of $\df_L(x,\w)$. The right panel shows the optimal control to reach the target state $y_d=\sin(\pi y)\sin(\pi x)$. Parameters: $L^2=0.025$, $\sigma^2=0.5$, $\beta=10^{-2}$, $\gamma=0.1$ and $N=10^4$.}\label{Fig:log-normal_fields}
\end{figure}
To keep $99\%$ of the variance, we would need to retain $M=37$ components in the Karhunen-Loève expansion \eqref{eq:log-normal_field}, so that SCM on tensor grids is not feasible due to the curse of dimensionality. We thus rely on a standard Monte Carlo with $N=10^4$ samples. The saddle point system involves approximately 19.2 millions degrees of freedom, and we first consider a control $u \in L^2(D)$. 

Table \ref{Tab:Iter_lognormal_L^2} reports the number of iterations and computational times in seconds for different values of $\beta$ and $\sigma^2$. Both $\widetilde{P}$ and $\PLRM$ reach the maximum number of iterations allowed (i.e. 200) before the tolerance of MINRES is achieved. In constrat, $\PLRC$ exhibits a weak dependence on $\beta$, but still remains quite efficient for the broad range values of $\beta$.
The performance of all preconditioners instead deteriorates when increases $\sigma^2$. We remark that $\sigma^2=1.5$ is quite a challenging setting: in our experiments we had $\max_{1\leq i\leq N}\frac{\df_{\max}(\w_i)}{\df_{\min}(\w_i)}=1.08e4$, that is the random diffusion field can vary up to four order of magnitude inside the domain (the expected variation is $\EQ\LQ \frac{\df_{\max}(\w)}{\df_{\min}(\w)}\RQ=396.39$).
\begin{table}[h]\caption{Number of iterations and computational time in seconds to reach a relative residual smaller than $10^{-6}$.}\label{Tab:Iter_lognormal_L^2}
\small
\centering
\begin{tabular}{| c | c | c | c | c | c|}
\hline
$\beta$ &$ 10^{-2}$ & $10^{-4}$ &  $10^{-6}$ & $10^{-8}$\\ \hline
$\widetilde{P}^{-1}\mathscr{S}$ & 37 (393.85)  &  55 (568.85) &  NA & NA \\
$\PLRM^{-1}\mathscr{S}$ & 37 (908.84)  &  53 (1270.7) &  NA & NA  \\
$\PLRC^{-1}\mathscr{S}$ & 37 (2576.0)   &  37 (2555.8) & 41 (2817.6) & 49 (3344.4) \\
\hline
\end{tabular}
\centering\\
$\sigma^2=0.5$, $\gamma=10^{-1}$, $N=10^4$, $\Nit=4$. NA means MINRES did not converge in less than 200 iterations.\\\vspace{0.3cm}
\begin{tabular}{| c | c | c | c | c | c|}
\hline
 & \diaghead(-3,2){\hskip0.7cm\hsize0.7cm}%
 {$\beta$}{$\sigma^2$} & 0.1 & 0.5 &  1 & 1.5 \\ \hline
$\widetilde{P}^{-1}\mathscr{S}$ & $10^{-2}$ & 31 (331.83) &  37 (390.43) &  43 (474.18) & 49 (555.86) \\
$\PLRM^{-1}\mathscr{S}$ &  $10^{-6}$ & 53 (1278.9) & NA & NA & NA \\
$\PLRC^{-1}\mathscr{S}$ & $ 10^{-6}$ & 31 (1538.3) & 41 (2903.4) & 59 (7321.4) & 79 (11724.7)  \\
\hline
\end{tabular}
\\ \centering
$\Nit$ equal to $\left\{2,4,8,10\right\}$ for $\sigma^2$ equal respectively to $\left\{0.1,0.5,1,1.5\right\}$; $\gamma=0.1$. NA means MINRES did not converge in less than 200 iterations.
\end{table}

Finally, we look for a $u \in Y$. Tables \ref{Tab:Iter_lognormal_OP} further confirm that both $\POPM$ and $\POPC$ lead to a $\beta-$robust convergence. The latter is again not $\sigma^2$-robust as the theory predicts, but the increase of the number of iterations is very modest. Nevertheless, we needed an increasing number of inner iterations as the mean approximations $B_{1,M}$ and $B_{3,M}$ lose their efficacy as preconditioners inside the Chebyshev semi-iterative method, and this results in a significant increase of computational times.
The development of improved preconditioners for the Chebyshev semi-iterative method should reduce the number $\Nit$ of inner iteration, and lead to a weak dependence not only of the number of iterations, but also of the computational times.

\begin{table}[t]\caption{Number of iterations and computational time in seconds to reach a relative residual smaller than $10^{-6}$.}\label{Tab:Iter_lognormal_OP}
\small
\centering
\begin{tabular}{| c | c | c | c | c | c|}
\hline
$\beta$ &$ 10^{-2}$ & $10^{-4}$ &  $10^{-6}$ & $10^{-8}$\\ \hline
$\POPM^{-1}\mathscr{S}$ & 39 (1190.8)  &  44 (1340.1) &  31 (961.2) & 32 (990.0) \\
$\POPC^{-1}\mathscr{S}$ & 37 (4937.6)  &  38 (5074.6) & 29 (3896.6) & 31 (4178.0) \\
\hline
\end{tabular}
\\
\centering
$\sigma^2=0.5$, $\gamma=10^{-1}$, $\Nit=4$.\\\vspace{0.3cm}
\begin{tabular}{| c | c | c | c | c | c|}
\hline
$\sigma^2$  & 0.1 & 0.5 & 1 & 1.5\\ \hline
$\POPM^{-1}\mathscr{S}$ & 16 (518.2)  &  31 (980.5) &  68 (2103.3) & 145 (4534.6) \\
$\POPC^{-1}\mathscr{S}$ & 26 (2420.8)  &  29 (3922.4) & 31 (6869.8) & 39 (10438.2) \\
\hline
\end{tabular}
\\\centering
$\Nit$ equal to $\left\{2,4,8,10\right\}$ for $\sigma^2$ equal respectively to $\left\{0.1,0.5,1,1.5\right\}$; $\beta=10^{-6}$ and $\gamma=10^{-1}$. 
\label{Tab:Iter_lognormal_sigma_OP}
\end{table}

\section{Conclusion}
In this manuscript, we studied preconditioners for the large saddle point systems which arise in the context of quadratic robust OCPUU. Our theoretical analysis casts light on the dependence of these preconditioners on the regularization parameter $\beta$ and on the variance $\sigma^2$ of the random field . 
For large values of $\beta$, the coupled saddle point system can be efficiently solved by preconditioning separately and in parallel all the state and adjoint equations. 
For small values of $\beta$, robustness can be recovered using two different preconditioners which require the additional solution of a linear system (whose size is equal to a single PDE discretization) which couples all the equations and involves the sum of the inverses of the stiffness matrices. We solved such reduced system using a mean approximation or a preconditioned Chebyshev semi-iterative method.
Our theoretical analysis characterizes the dependence of the preconditioners on the variance of the random field through either the first or second moment of $1/\df_{\min}(\w$) or $\df_{\max}(\w)$. The weak dependence for physically relevant ranges of $\sigma^2$ is confirmed by our numerical experiments in terms of number of iterations, but not necessarily in terms of computational times, as one needs to increase the number of inner Chebyshev semi-iterations for large values of $\sigma^2$.  
Hence, the combination of small values of $\beta$ and large values of $\sigma^2$ is still challenging, and the development of ad-hoc preconditioners for the reduced system involving the sum of the inverses of the stiffness matrices is expected to close the gap between the theoretical results and practical implementations.

\section{Acknowledgements}
The two authors acknowledge funding from the European Union’s Horizon 2020 research and innovation programme
under grant agreement N. 800898, project ExaQUte – EXAscale Quantification of Uncertainties for Technology and
Science Simulation. The first author acknowledge funding also from the Swiss National Science Foundation under the Project n. 172678 “Uncertainty Quantification techniques for PDE constrained optimization and random evolution equations”

\bibliographystyle{plain}
\bibliography{bib}

\begin{thebibliography}{10}

\bibitem{ReporUQ2}
Q.~Ayoul-Guilmard, S.~Ganesh, F.~Nobile, R.~Rossi, and C~Soriano.
\newblock D6.2: Report on the calculation of stochastic sensitivities.
\newblock {\em Open Access Repository of the ExaQUte project: Deliverables},
  2021.

\bibitem{ReporUQ}
Q.~Ayoul-Guilmard, S.~Ganesh, F.~Nobile, R.~Rossi, and C~Soriano.
\newblock D6.3 report on stochastic optimisation for simple problems.
\newblock {\em Open Access Repository of the ExaQUte project: Deliverables},
  2021.

\bibitem{benner2016block}
P.~Benner, A.~Onwunta, and M.~Stoll.
\newblock Block-diagonal preconditioning for optimal control problems
  constrained by {PDE}s with uncertain inputs.
\newblock {\em SIAM Journal on Matrix Analysis and Applications},
  37(2):491--518, 2016.

\bibitem{bonizzoni2013analysis}
F.~Bonizzoni.
\newblock {\em Analysis and approximation of moment equations for PDEs with
  stochastic data}.
\newblock PhD thesis, Politecnico di Milano, 2013.

\bibitem{Borzi2}
A.~Borz{\`\i}.
\newblock Multigrid and sparse-grid schemes for elliptic control problems with
  random coefficients.
\newblock {\em Computing and visualization in science}, 13(4):153--160, 2010.

\bibitem{Borzi}
A.~Borzì and G.~von Winckel.
\newblock Multigrid methods and sparse-grid collocation techniques for
  parabolic optimal control problems with random coefficients.
\newblock {\em SIAM Journal on Scientific Computing}, 31(3):2172--2192, 2009.

\bibitem{https://doi.org/10.1002/nme.2758}
J.~Boyle, M.~Mihajlović, and J.~Scott.
\newblock {HSL\_MI20}: An efficient {AMG} preconditioner for finite element
  problems in 3d.
\newblock {\em International Journal for Numerical Methods in Engineering},
  82(1):64--98.

\bibitem{chapman1997deflated}
A.~Chapman and Y.~Saad.
\newblock Deflated and augmented {K}rylov subspace techniques.
\newblock {\em Numerical linear algebra with applications}, 4(1):43--66, 1997.

\bibitem{Charrier}
J.~Charrier.
\newblock Strong and weak error estimates for elliptic partial differential
  equations with random coefficients.
\newblock {\em SIAM Journal on Numerical Analysis}, 50(1):216--246, 2012.

\bibitem{Scheichl}
J.~Charrier, R.~Scheichl, and A.~L. Teckentrup.
\newblock Finite element error analysis of elliptic {PDE}s with random
  coefficients and its application to multilevel {M}onte {C}arlo methods.
\newblock {\em SIAM Journal on Numerical Analysis}, 51(1):322--352, 2013.

\bibitem{Ciarlet}
P.G. Ciarlet.
\newblock {\em Linear and Nonlinear Functional Analysis with Applications}.
\newblock Applied mathematics. SIAM, Philadelphia, PA, 2013.

\bibitem{cohn2013measure}
D.L. Cohn.
\newblock {\em Measure Theory: Second Edition}.
\newblock Birkh{\"a}user Advanced Texts Basler Lehrb{\"u}cher. Springer New
  York, 2013.

\bibitem{elman2014finite}
H.~C. Elman, D.~J. Silvester, and A.~Wathen.
\newblock {\em Finite elements and fast iterative solvers: with applications in
  incompressible fluid dynamics}.
\newblock Numerical Mathematics and Scie, 2014.

\bibitem{elvetun2016pde}
O.~L. Elvetun and B.~F. Nielsen.
\newblock {PDE}-constrained optimization with local control and boundary
  observations: Robust preconditioners.
\newblock {\em SIAM Journal on Scientific Computing}, 38(6):A3461--A3491, 2016.

\bibitem{Geiersbach}
C.~Geiersbach and W.~Wollner.
\newblock A stochastic gradient method with mesh refinement for
  {PDE}-constrained optimization under uncertainty.
\newblock {\em SIAM Journal on Scientific Computing}, 42(5):A2750--A2772, 2020.

\bibitem{ghanem}
R.~G. Ghanem and R.~M. Kruger.
\newblock Numerical solution of spectral stochastic finite element systems.
\newblock {\em Computer Methods in Applied Mechanics and Engineering},
  129(3):289--303, 1996.

\bibitem{gittelson}
C.~J. Gittelson.
\newblock Stochastic {G}alerkin discretization of the log-normal isotropic
  diffusion problem.
\newblock {\em Mathematical Models and Methods in Applied Sciences},
  20(02):237--263, 2010.

\bibitem{golub2013matrix}
G.H. Golub and C.F. Van~Loan.
\newblock {\em Matrix Computations}.
\newblock Johns Hopkins Studies in the Mathematical Sciences. Johns Hopkins
  University Press, 2013.

\bibitem{doi:10.1137/19M1294952}
P.A. Guth, V.~Kaarnioja, F.~Kuo, C.~Schillings, and I.H. Sloan.
\newblock A {Q}uasi-{M}onte {C}arlo method for optimal control under
  uncertainty.
\newblock {\em SIAM/ASA Journal on Uncertainty Quantification}, 9(2):354--383,
  2021.

\bibitem{heidel2019preconditioning}
G.~Heidel and A.~Wathen.
\newblock Preconditioning for boundary control problems in incompressible fluid
  dynamics.
\newblock {\em Numerical Linear Algebra with Applications}, 26(1):e2218, 2019.

\bibitem{hinze2008optimization}
M.~Hinze, R.~Pinnau, M.~Ulbrich, and S.~Ulbrich.
\newblock {\em Optimization with {PDE} constraints}, volume~23.
\newblock Springer Science \& Business Media, 2008.

\bibitem{khan2019robust}
A.~Khan, C.~E. Powell, and D.~J. Silvester.
\newblock Robust preconditioning for stochastic {G}alerkin formulations of
  parameter-dependent nearly incompressible elasticity equations.
\newblock {\em SIAM Journal on Scientific Computing}, 41(1):A402--A421, 2019.

\bibitem{functional_iterative}
R.~C. Kirby.
\newblock From functional analysis to iterative methods.
\newblock {\em SIAM Review}, 52(2):269--293, 2010.

\bibitem{kouri2014multilevel}
D.~P. Kouri.
\newblock A multilevel stochastic collocation algorithm for optimization of
  {PDE}s with uncertain coefficients.
\newblock {\em SIAM/ASA Journal on Uncertainty Quantification}, 2(1):55--81,
  2014.

\bibitem{Kourisparse}
D.~P. Kouri, M.~Heinkenschloss, D.~Ridzal, and B.~G. van Bloemen~Waanders.
\newblock A trust-region algorithm with adaptive stochastic collocation for
  {PDE} optimization under uncertainty.
\newblock {\em SIAM Journal on Scientific Computing}, 35(4):A1847--A1879, 2013.

\bibitem{Kouri2018}
D.~P. Kouri and D.~Ridzal.
\newblock {\em Inexact Trust-Region Methods for PDE-Constrained Optimization},
  pages 83--121.
\newblock Springer New York, New York, NY, 2018.

\bibitem{kourikkt}
D.~P. Kouri, D.~Ridzal, and R.~Tuminaro.
\newblock {KKT} preconditioners for {PDE}-constrained optimization with the
  {H}elmholtz equation.
\newblock {\em SIAM Journal on Scientific Computing}, 0(0):S225--S248, 2020.

\bibitem{Kouri_Cvar}
D.~P. Kouri and T.~M. Surowiec.
\newblock Risk-averse {PDE}-constrained optimization using the conditional
  value-at-risk.
\newblock {\em SIAM Journal on Optimization}, 26(1):365--396, 2016.

\bibitem{Kouri_ex}
D.~P. Kouri and T.~M. Surowiec.
\newblock Existence and optimality conditions for risk-averse {PDE}-constrained
  optimization.
\newblock {\em SIAM/ASA Journal on Uncertainty Quantification}, 6(2):787--815,
  2018.

\bibitem{lions1971optimal}
J.L. Lions, M.~Seeliger, and S.K. Mitter.
\newblock {\em Optimal Control of Systems Governed by Partial Differential
  Equations:}.
\newblock Die Grundlehren der mathematischen Wissenschaften in
  Einzeldarstellungen. Springer-Verlag, 1971.

\bibitem{liu2020parameter}
J.~Liu and J.~W. Pearson.
\newblock Parameter-robust preconditioning for the optimal control of the wave
  equation.
\newblock {\em Numerical Algorithms}, 83(3):1171--1203, 2020.

\bibitem{lord_powell_shardlow_2014}
G.~J. Lord, C.~E. Powell, and T.~Shardlow.
\newblock {\em An Introduction to Computational Stochastic {PDE}s}.
\newblock Cambridge Texts in Applied Mathematics. Cambridge University Press,
  2014.

\bibitem{malek2014preconditioning}
J.~Malek and Z.~Strakos.
\newblock {\em Preconditioning and the Conjugate Gradient Method in the Context
  of Solving {PDE}s}.
\newblock SIAM Spotlights. SIAM, Society for Industrial and Applied
  Mathematics, 2014.

\bibitem{mardal2011preconditioning}
K.-A. Mardal and R.~Winther.
\newblock Preconditioning discretizations of systems of partial differential
  equations.
\newblock {\em Numerical Linear Algebra with Applications}, 18(1):1--40, 2011.

\bibitem{martin2018analysis}
M.~Martin, S.~Krumscheid, and F.~Nobile.
\newblock Complexity analysis of stochastic gradient methods for
  pde-constrained optimal control problems with uncertain parameters.
\newblock {\em ESAIM: Mathematical Modelling and Numerical Analysis},
  55(4):1599--1633, 2021.

\bibitem{Martin_Nobile2}
M.~Martin and F.~Nobile.
\newblock Pde-constrained optimal control problems with uncertain parameters
  using saga.
\newblock {\em SIAM/ASA Journal on Uncertainty Quantification}, 9(3):979--1012,
  2021.

\bibitem{martinez2018optimal}
J.~Mart{\'i}nez-Frutos and F.~Esparza.
\newblock {\em Optimal Control of {PDE}s Under Uncertainty: An Introduction
  with Application to Optimal Shape Design of Structures}.
\newblock Springer, 2018.

\bibitem{murphy2000note}
M.~F. Murphy, G.~H. Golub, and A.~Wathen.
\newblock A note on preconditioning for indefinite linear systems.
\newblock {\em SIAM Journal on Scientific Computing}, 21(6):1969--1972, 2000.

\bibitem{olshanskii2010acquired}
M.~A. Olshanskii and V.~Simoncini.
\newblock Acquired clustering properties and solution of certain saddle point
  systems.
\newblock {\em SIAM journal on matrix analysis and applications},
  31(5):2754--2768, 2010.

\bibitem{pearson2017fast}
J.~W. Pearson and J.~Gondzio.
\newblock Fast interior point solution of quadratic programming problems
  arising from {PDE}-constrained optimization.
\newblock {\em Numerische Mathematik}, 137(4):959--999, 2017.

\bibitem{pearson2012new}
J.~W. Pearson and A.~Wathen.
\newblock A new approximation of the {S}chur complement in preconditioners for
  {PDE}-constrained optimization.
\newblock {\em Numerical Linear Algebra with Applications}, 19(5):816--829,
  2012.

\bibitem{pearson2018matching}
J.~W. Pearson and A.~Wathen.
\newblock Matching {S}chur complement approximations for certain saddle-point
  systems.
\newblock In {\em Contemporary Computational Mathematics-A Celebration of the
  80th Birthday of Ian Sloan}, pages 1001--1016. Springer, 2018.

\bibitem{pellissetti}
M.F Pellissetti and R.G Ghanem.
\newblock Iterative solution of systems of linear equations arising in the
  context of stochastic finite elements.
\newblock {\em Advances in Engineering Software}, 31(8):607--616, 2000.

\bibitem{powell2009block}
C.~E Powell and H.~C. Elman.
\newblock Block-diagonal preconditioning for spectral stochastic finite-element
  systems.
\newblock {\em IMA Journal of Numerical Analysis}, 29(2):350--375, 2009.

\bibitem{quarteroni2009numerical}
A.~Quarteroni and A.~Valli.
\newblock {\em Numerical Approximation of Partial Differential Equations}.
\newblock Springer Series in Computational Mathematics. Springer Berlin
  Heidelberg, 2009.

\bibitem{rees2010optimal}
T.~Rees, H.~S. Dollar, and A.~Wathen.
\newblock Optimal solvers for {PDE}-constrained optimization.
\newblock {\em SIAM Journal on Scientific Computing}, 32(1):271--298, 2010.

\bibitem{rees2010all}
T.~Rees, M.~Stoll, and A.~Wathen.
\newblock All-at-once preconditioning in {PDE}-constrained optimization.
\newblock {\em Kybernetika}, 46(2):341--360, 2010.

\bibitem{rosseel2012optimal}
E.~Rosseel and G.~N. Wells.
\newblock Optimal control with stochastic {PDE} constraints and uncertain
  controls.
\newblock {\em Computer Methods in Applied Mechanics and Engineering},
  213:152--167, 2012.

\bibitem{schillings2013sparse}
C.~Schillings and C.~Schwab.
\newblock Sparse, adaptive {S}molyak quadratures for bayesian inverse problems.
\newblock {\em Inverse Problems}, 29(6):065011, 2013.

\bibitem{schwab}
C.~Schwab and C.~J. Gittelson.
\newblock Sparse tensor discretizations of high-dimensional parametric and
  stochastic {PDE}s.
\newblock {\em Acta Numerica}, 20:291–467, 2011.

\bibitem{doi:10.1137/060660977}
J.~Schöberl and W.~Zulehner.
\newblock Symmetric indefinite preconditioners for saddle point problems with
  applications to {PDE}-constrained optimization problems.
\newblock {\em SIAM Journal on Matrix Analysis and Applications},
  29(3):752--773, 2007.

\bibitem{shapiro2014lectures}
A.~Shapiro, D.~Dentcheva, and A.~Ruszczy{\'n}ski.
\newblock {\em Lectures on stochastic programming: modeling and theory}.
\newblock SIAM, 2014.

\bibitem{smith2013uncertainty}
R.C. Smith.
\newblock {\em Uncertainty Quantification: Theory, Implementation, and
  Applications}.
\newblock Computational Science and Engineering. SIAM, 2013.

\bibitem{troltzsch2010optimal}
F.~Tr{\"o}ltzsch and J.~Sprekels.
\newblock {\em Optimal Control of Partial Differential Equations: Theory,
  Methods, and Applications}.
\newblock Graduate studies in mathematics. American Mathematical Society, 2010.

\bibitem{ullmann}
E.~Ullmann.
\newblock A {K}ronecker product preconditioner for stochastic {G}alerkin finite
  element discretizations.
\newblock {\em SIAM Journal on Scientific Computing}, 32(2):923--946, 2010.

\bibitem{van2019robust}
A.~Van~Barel and S.~Vandewalle.
\newblock Robust optimization of {PDE}s with random coefficients using a
  multilevel {M}onte {C}arlo method.
\newblock {\em SIAM/ASA Journal on Uncertainty Quantification}, 7(1):174--202,
  2019.

\bibitem{van2020mg}
A.~Van~Barel and S.~Vandewalle.
\newblock {MG/OPT} and {MLMC} for robust optimization of {PDE}s.
\newblock {\em arXiv preprint arXiv:2006.01231}, 2020.

\bibitem{wathen2008chebyshev}
A.J. Wathen and T.~Rees.
\newblock {C}hebyshev semi- iteration in preconditioning.
\newblock 2008.

\bibitem{Whittle}
P.~Whittle.
\newblock A multivariate generalization of {C}hebyshev's inequality.
\newblock 2:232--240, 1958.

\bibitem{zahr2019efficient}
M.~J. Zahr, K.~T. Carlberg, and D.~P. Kouri.
\newblock An efficient, globally convergent method for optimization under
  uncertainty using adaptive model reduction and sparse grids.
\newblock {\em SIAM/ASA Journal on Uncertainty Quantification}, 7(3):877--912,
  2019.

\bibitem{zulehner2011nonstandard}
W.~Zulehner.
\newblock Nonstandard norms and robust estimates for saddle point problems.
\newblock {\em SIAM Journal on Matrix Analysis and Applications},
  32(2):536--560, 2011.

\end{thebibliography}

\end{document}